\documentclass[reqno]{amsart}
\usepackage{cite}
\allowdisplaybreaks
\date{\today}

 \numberwithin{equation}{section}

\newcommand{\dv}{\mathrm{div}\,}

\newtheorem{Theorem}{Theorem}[section]
\newtheorem{Lemma}{Lemma}[section]

\newtheorem{Remark}{Remark}[section]
\begin{document}

\title[incompressible viscous MHD equations]% with Navier boundary conditions]
 {Zero kinematic viscosity-magnetic diffusion limit
 of the incompressible  viscous magnetohydrodynamic equations with Navier boundary conditions}

\author{Fucai Li}
\address{ Department of  Mathematics,
 Nanjing  University, Nanjing 210093, P.R.China}
\email{fli@nju.edu.cn}
%\author{Yanmin Mu}
%\address{ Department of  Mathematics,
% Nanjing  University, Nanjing 210093, P.R.China}
%\email{yminmu@126.com}
 \author[Zhipeng Zhang]{Zhipeng Zhang}% {Zhipeng Zhang$^*$}
\address{Department of Mathematics, Nanjing University, Nanjing
 210093, P.R. China}
 \email{zhpzhp@aliyun.com}
 %\thanks{$^*$Corresponding author}

%
% \author{Gen Nakamura}
%\address{Department of Mathematics, Inha University, Incheon 402-751, Korea}
% \email{gnaka@math.sci.hokudai.ac.jp}
%
% \author{Zhong Tan}
%\address{School of Mathematical Sciences,
%  Xiamen University, Xiamen, 361005, P. R. China}
% \email{tan85@xmu.edu.cn}

%%%%----------------------------------------------------------------------------------%%%%%%
\begin{abstract}
We investigate the zero kinematic viscosity-magnetic diffusion limit of the incompressible viscous magnetohydrodynamic  equations  with Navier boundary conditions in a smooth bounded domain $\Omega\subset\mathbb{R}^3$.
% and conormal Sobolev spaces. Here, $\epsilon$ (kinematic viscosity) is equal to $\nu$ (magnetic diffusivity), so we use $\epsilon$ to replace $\nu$ .
We obtain the uniform regularity of solutions with respect to the kinematic viscosity coefficient and the magnetic diffusivity coefficient.
These solutions are uniformly bounded in a conormal Sobolev space and $W^{1,\infty}(\Omega)$ which allow us to take the
zero kinematic viscosity-magnetic diffusion limit.
% to obtain the ideal incompressible magnetohydrodynamic equations.
 Moreover, we also get the rates of convergence.
%  in $L^\infty(0,T_2;L^2(\Omega))$, $L^\infty(0,T_2;H^1(\Omega))$, $L^\infty([0,T_2]\times \Omega)$ and $L^\infty(0,T_2;W^{1,p}(\Omega))$.
 \end{abstract}


\maketitle

%%%%----------------------------------------------------------------------------------%%%%%%

\section{Introduction}
We consider the following incompressible viscous magnetohydrodynamic (MHD) equations (\!\!\cite{Da,DB})
\begin{align}
&\partial_tv^\epsilon-\epsilon\Delta v^\epsilon+v^\epsilon\cdot\nabla v^\epsilon
-H^\epsilon\cdot\nabla H^\epsilon-\frac{1}{2}\nabla(|v^\epsilon|^2-|H^\epsilon|^2)+\nabla p^\epsilon=0,\label{1.1}\\
&\partial_tH^\epsilon-\epsilon\Delta H^\epsilon+v^\epsilon\cdot\nabla H^\epsilon-H^\epsilon\cdot\nabla v^\epsilon=0,\label{1.2}\\
&\dv v^\epsilon=\dv H^\epsilon=0\label{1.3}
\end{align}
in $(0,T)\times\Omega$, where $\Omega$ is a smooth bounded domain of $\mathbb{R}^3$. The unknowns $v^\epsilon$ and $H^\epsilon$ represent
the fluid velocity and  the magnetic field, respectively. The pressure  $p^\epsilon$ can be  recovered from $v^\epsilon$ and $H^\epsilon$
via an explicit Calder車n-Zygmund singular integral operator (\!\!\cite{Ch}).

We add to $v^\epsilon$ and $H^\epsilon$ the following initial and boundary conditions
\begin{align}
&v^\epsilon\cdot n=0,\quad (Sv^\epsilon\cdot n)_\tau=-\zeta v^\epsilon_\tau\quad\text{on}~~~~\partial\Omega,\label{1.4.1}\\
&H^\epsilon\cdot n=0,\quad (SH^\epsilon\cdot n)_\tau=-\zeta H^\epsilon_\tau \quad\text{on}~~~~\partial\Omega,\label{1.4.2}\\
&(v^\epsilon,H^\epsilon)|_{t=0}=(v_0,H_0)\quad\text{in}\quad\Omega,\label{1.4}
\end{align}
where $n$ stands for the outward unit normal vector to $\Omega$, $\zeta$ is a coefficient measuring the tendency of the fluid to slip on the boundary, $S$ is the strain tensor defined by
\begin{equation}
Su=\frac{1}{2}(\nabla u+\nabla u^t),\nonumber
\end{equation}
where $\nabla u^t$ denotes the transpose of the matrix $\nabla u$, and $u_\tau$ stands for the tangential part of $u$ on $\partial\Omega$, i.e.
\begin{equation}
u_\tau=u-(u\cdot n)n.\nonumber
\end{equation}

This kind boundary condition \eqref{1.4.1}
was introduced by Navier in \cite{NC} to show that the velocity is propositional to the tangential part of the stress.
It allows the fluid slip along the boundary and are often used to model rough boundaries.
The Navier boundary condition \eqref{1.4.1} can be generalized to the following form (\!\!\cite{GJ})
\begin{align}
 u\cdot n=0,~~(Su\cdot n)_\tau+Au=0,\label{1.11}
\end{align}
where $A$ is a $(1,1)$-type tensor on the boundary $\partial \Omega$. When $A=\zeta\, \text{Id}$ (here Id denotes the  identity matrix), \eqref{1.11} is  reduced to the standard Navier boundary condition.
For smooth functions, we can get the form of the vorticity
\begin{align}
 u\cdot n=0,~~n\times \omega=[Bu]_\tau\quad\text{on}\quad\partial\Omega,\label{1.12}
\end{align}
where $\omega=\nabla\times u$ is the vorticity and $B=2(A-S(n))$ (\!\!\cite{YZ}).

In this paper  we are interested in the existence of strong solution to the problem \eqref{1.1}-\eqref{1.4} with uniform bounds on an interval of time independent of $\epsilon$
and taking the  limit $\epsilon \rightarrow 0$ to obtain  the ideal incompressible MHD equations, i.e.
\begin{align}
&\partial_tv+v\cdot\nabla v
-H\cdot\nabla H-\frac{1}{2}\nabla(|v|^2-|H|^2)+\nabla p=0,\label{1.7}\\
&\partial_tH+v\cdot\nabla H-H\cdot\nabla v=0,\label{1.8}\\
&\dv v=\dv H=0\label{1.9}
\end{align}
with the following slip boundary conditions:
\begin{equation}\label{1.10}
v\cdot n=H\cdot n=0.
\end{equation}

When taking $H^\epsilon=0$ in the system \eqref{1.1}-\eqref{1.3}, it is reduced to the classical incompressible
Navier-Stokes equations and there are many
 literature on  the vanishing viscosity limit of it. In the case that there is no boundary, a uniform time of existence and the vanishing viscosity limit have been obtained, see \cite{KT,MN,SH}. When the boundary appear, it is usually difficult to do higher order energy estimates near boundary because of the appearing of the boundary layer
 \cite{OS}.
In particular,  for the incompressible Navier-Stokes equations  with  no-slip boundary condition,
the  vanishing viscosity limit of it is wildly open except when the
initial data is analytic \cite{SC1,SC2} or the initial vorticity is located away from the boundary in the
two-dimensional half plane \cite{Ma}. On the other hand, considering the incompressible Navier-Stokes system with  Navier
boundary conditions, more results are available, see, for example, \cite{XZ1,HB,HB1,HC,IS,BS}.
Xiao and Xin \cite{XZ1} investigate the vanishing viscosity limit to incompressible Navier-Stokes equation with the boundary conditions
\begin{align}
 u\cdot n=0,~~n\times \omega=0\quad\text{on} \quad\partial\Omega.\label{1.13}
\end{align}
Because the main part in the boundary layer vanishes (i.e. $V=0$  in \eqref{BL} below), they can obtain the local existence of strong solution with some uniform bounds in $H^3(\Omega)$
and the vanishing viscosity limit. Their approaches overcame the compatibility issues of the nonlinear terms with \eqref{1.13}. The authors in \cite{HB} got uniform estimates in $W^{k,p}(\Omega)$ with $k\geq3$ and $p\geq2$. The main reason is that the boundary integrals vanishes on flat portions of the boundary, see also  \cite{HB1,HC}.
Later, the results in \cite{XZ1,HB} was generalized by Berselli and Spirito  \cite{BS} to  a general bounded domain  under certain restrictions on the initial data.
In order to analysis the effect of the boundary layer in a general bounded domain,  Iftimie and  Sueur \cite{IS} constructed the boundary layer for the incompressible
Navier-Stokes equations with the Navier boundary condition \eqref{1.4.1} in the form
\begin{align}\label{BL}
u^\epsilon(t,x)=u^\epsilon(t,x)+\sqrt{\epsilon}V\Big(t,x,\frac{\phi(x)}{\sqrt{\epsilon}}\Big)+O(\epsilon),
\end{align}
% and there are many we can refer to \cite{WW} and, respectively for compressible fluids and  incompressible fluids.
where the function $V$ vanishes for $x$ outside a small neighborhood of $\partial\Omega$ and $\phi(x)$ is the distance between $x$ and $\partial\Omega$ for $x$ in a neighborhood of $\partial\Omega$. % (see \cite{IS} for details).
The layers constructed in  \cite{IS} are of width $O(\sqrt{\epsilon})$ like the Prandtl layer \cite{OS}, but are of amplitude $O(\sqrt{\epsilon})$ (The Prandtl layer is of width $O(\sqrt{\epsilon})$ and of amplitude $O(1)$). So it is impossible to obtain the $H^3(\Omega)$ or $W^{2,p}(\Omega)$ ($p$ large enough) uniform estimates for the incompressible Navier-Stokes equations.
Recently,
Masmoudi and Rousset  \cite{MR} considered the the vanishing viscosity limit for the incompressible Navier-Stokes equation with the boundary condition \eqref{1.4.1} in anisotropic conormal Sobolev spaces which can eliminate the effects of normal derivatives near the boundary. % and will be introduced following,
They obtained uniform regularity and the convergence of the viscous solutions to the inviscid ones by compactness argument.
 Recently, some results in \cite{MR} was extended to the compressible isentropic Navier-Stokes equations with Navier boundary conditions \cite{WXY,M}.
Moreover, based on the results in \cite{MR}, the rates of convergence were obtained by Gie and Kelliher \cite{GJ} and Xiao and Xin \cite{YZ}, respectively.
%When $B=0$ in \eqref{1.12}, the generalized Navier-slip boundary condition become
%follow and adapt the approach in  \cite{MR} and
%We are also aware of establishing , it is known that some convergence results for the $2D$ or $3D$ Navier-Stokes equations with Navier-slip boundary %condition have been established, see\cite{GJ}, \cite{YZ}, \cite{XZ1}, \cite{WXY} and the references therein.  We note that the uniform regularity of \cite{MR} play an important %role in \cite{YZ}, they obtain some rates of convergence in $L^{\infty}([0,T]\times\Omega)$, $W^{1,p}(\Omega)$ with $(2\leq p<\infty)$ and $L^2(\Omega)$. Our approach here is %motivated by the introduced in \cite{YZ}, and we get also s
%
%In \cite{WXY}-\cite{W}, authors study uniform regularity and the vanishing viscosity limit of the compressible Navier-Stokes equations with Navier boundary conditions in anisotropic conormal Sobolev spaces. Also, there are many articles investigating the vanishing viscosity limit for the free boundary problem Navier-Stokes equation with or without surface tension, see \cite{EL}-\cite{MWX}.

%The viscous MHD equations  have been studied extensively, especially in the whole space or with non-slip boundary condition, and there is a large literature on various topics concerning the MHD equations such as the well-posedness for the MHD equations, see \cite{HW}-\cite{FLN} and the references cited therein.

% Until now, there are only a few results on  the inviscid limit for the viscous incompressible MHD equations \eqref{1.1}-\eqref{1.3}.
 In \cite{XZW}, Xiao, Xin and Wu  studied the
inviscid limit for the  system \eqref{1.1}-\eqref{1.3} with the boundary conditions
 \begin{equation}
 \left\{
 \begin{array}{l}
 v^\epsilon\cdot n=0,\quad n\times \omega_v^\epsilon=0\quad\text{on}\quad\partial\Omega,\\
 H^\epsilon\cdot n=0,\quad n\times \omega_H^\epsilon=0\quad\text{on}\quad\partial\Omega,
 \end{array}
 \right.
 \end{equation}
  where they used the approaches similar to that in \cite{XZ1} and formulated the boundary value in a suitable functional setting so that the stokes operator is well behaved and the nonlinear terms fall into the desired functional spaces. These facts allow them to get the uniform regularity for the viscous incompressible MHD system through the Galerkin approximation and a priori energy estimates.
% In \cite{GW}, Guo and Wang also use the same methods as that in \cite{XZ1} to deal with the inviscid limit  to  the system \eqref{1.1}-\eqref{1.3} with the boundary conditions
% \eqref{1.4.1}-\eqref{1.4.2}, but they didn't get the similar estimates to  that in \cite{XZW}.

%The purpose of this paper is to
Here we investigate the
inviscid limit for the  system \eqref{1.1}-\eqref{1.3} with the Navier boundary conditions \eqref{1.4.1}-\eqref{1.4.2} in a $3D$ bounded domain in the framework of  anisotropic conormal Sobolev spaces. Due to the strong coupling between $v^\epsilon$ and $H^\epsilon$, a priori estimates become more complicated than that in \cite{MR} on the incompressible Navier-Stokes equations. We obtain uniform regularity of the solutions and, with this well-posedness theory, pursue the vanishing viscosity limit to the problem \eqref{1.1}-\eqref{1.4}.
Moreover, we also obtain   some rates of convergence for $v^\epsilon$ and $H^\epsilon$. Hence our results can be regarded as generalizations of those in \cite{GJ,MR,YZ} to incompressible MHD eqautions.
% follow and adapt the approach in  \cite{MR} and
%We are also aware of establishing the rate of convergence, it is known that some convergence results for the $2D$ or $3D$ Navier-Stokes equations with Navier-slip boundary %condition have been established, see\cite{GJ}, \cite{YZ}, \cite{XZ1}, \cite{WXY} and the references therein.  We note that the uniform regularity of \cite{MR} play an important %role in \cite{YZ}, they obtain some rates of convergence in $L^{\infty}([0,T]\times\Omega)$, $W^{1,p}(\Omega)$ with $(2\leq p<\infty)$ and $L^2(\Omega)$. Our approach here is %motivated by the introduced in \cite{YZ}, and we get also s

%To state our results, we introduce the following function spaces:
%where 步B s
%p,1 denotes the homogeneous Besov space. We shall explain these notations in
%detail in Appendix A.

Our first result of this paper reads as follows.

%The  main result in our paper can be stated as follows. % readsare the following:
\begin{Theorem}\label{Th1}
 Let $m$ be an integer satisfying $m>6$ and $\Omega$ be a $C^{m+2}$ domain. Assume that the initial data $(v_0,H_0)$ satisfy
 $$(v_0,H_0) \in \mathcal{E}^m, ~~(\nabla v_0,\nabla H_0)\in W^{1,\infty}_{co}(\Omega),$$
 $$\nabla\cdot v_0=\nabla\cdot H_0=0,~~ v_0\cdot n|_{\partial\Omega}=H_0\cdot n|_{\partial\Omega}=0.$$
Then, there exist $T_0>0$ and $\widetilde{C}$, independent of $\epsilon\in(0,1]$ and $|\zeta|\leq1$, such that there exists a unique solution of the problem \eqref{1.1}-\eqref{1.4}
satisfying
\begin{equation}
(v^\epsilon,H^\epsilon)\in C([0,T_0],\mathcal{E}^m)\nonumber
\end{equation}
and
\begin{align}\label{1.16}
&\sup_{t\in[0,T_0]}\big{\{}\|(v^\epsilon,H^\epsilon)(t)\|_m+\|(\nabla v^\epsilon,\nabla H^\epsilon)(t)\|_{m-1}+
\|(\nabla v^\epsilon,\nabla H^\epsilon)(t)\|_{1,\infty}\big{\}}\nonumber\\
&\qquad \quad +\epsilon\int^{T_0}_0(\|\nabla^2v^\epsilon(t)\|^2_{m-1}+\|\nabla^2H^\epsilon(t)\|^2_{m-1})dt\leq \widetilde{C},
\end{align}
%where $T_m\geq T$.
Here $ \mathcal{E}^m:=\{u\,|\,u\in H^m_{co}(\Omega),\nabla u\in H^{m-1}_{co}(\Omega) \}$ and
the meanings of $ W^{1,\infty}_{co}(\Omega)$, $H^m_{co}(\Omega)$, $\|\cdot\|_m$ and $\|\cdot\|_{m,\infty}$ will be explained in detail in next section.
 \end{Theorem}
 \begin{Remark}\label{R1.2}
 When the Navier boundary conditions \eqref{1.4.1} and \eqref{1.4.2} are replaced by the following
 \begin{equation}\label{GNB1}
 \left\{
 \begin{array}{l}
 v^\epsilon\cdot n=0,\quad n\times\omega_v^\epsilon=[Bv^\epsilon]_\tau\quad\text{on}\quad\partial\Omega,\\
 H^\epsilon\cdot n=0,\quad n\times\omega_H^\epsilon=[BH^\epsilon]_\tau\quad\text{on}\quad\partial\Omega,
 \end{array}
 \right.
 \end{equation}
we can also obtain the same results as those in Theorem \ref{Th1}, where $B=2(A-S(n))$ and A is a $(1,1)$-type tensor on the boundary $\partial \Omega$.
 \end{Remark}
 \begin{Remark}\label{R1.3}
 Theorem \ref{Th1} still holds if we replace the boundary conditions \eqref{1.4.1} and \eqref{1.4.2} by the slightly generalized one:
 \begin{equation}\label{GNB2}
 \left\{
 \begin{array}{l}
 v^\epsilon\cdot n=0,\quad(Sv^\epsilon\cdot n)_\tau=-\zeta_1 v^\epsilon\quad\text{on}\quad\partial\Omega,\\
 H^\epsilon\cdot n=0,\quad(SH^\epsilon\cdot n)_\tau=-\zeta_1 H^\epsilon\quad\text{on}\quad\partial\Omega,
 \end{array}
 \right.
 \end{equation}
 where $\zeta_1$ and $\zeta_2$ are two different constants.
 \end{Remark}

%To prove Theorem \ref{Th1}, we follow and adapt the method developed by  Masmoudi and Rousset  \cite{MR} to the incompressible Navier-Stokes equations
%with Navier boundary conditions.
We now give some comments on the proof of Theorem \ref{Th1}.
The main steps of the proof  are similar to those in \cite{MR} in some sense. However, due to the strong coupling between $v^\epsilon$ and $H^\epsilon$, we need to overcome some new difficulties and to face more complicated energy estimates.
First, we get a conormal energy estimates in $H^m_{co}$ (see the definition in next section) for $(v^\epsilon,H^\epsilon)$. Here, we define $P^\epsilon_1+P^\epsilon_2:=p^\epsilon-\frac{1}{2}(|v^\epsilon|^2-|H^\epsilon|^2)$, where $P^\epsilon_1$ and $P^\epsilon_2$ satisfy corresponding boundary value problems (see \eqref{3.8} and \eqref{3.9} below), respectively. By doing this decomposition, we can avoid higher order terms which are out of control.
 %However, the right-hand side of energy inequality involves the $H^{m-1}_{co}$  norm of $\nabla P^\epsilon$ and $(\nabla v^\epsilon,\nabla H^\epsilon)$ and the Lipschitz norm of $(v^\epsilon,H^\epsilon)$, and  the following work is to estimate the right-hand terms.
In the second step, we estimate $\|(\partial_n v^\epsilon,\partial_nH^\epsilon)\|_{m-1}$. Due to the incompressible conditions \eqref{1.3}, both $\partial_nv^\epsilon\cdot n$ and $\partial_nH^\epsilon\cdot n$ can be easily controlled by the $H^m_{co}$ norm of $(v^\epsilon,H^\epsilon)$. Thanks to the the Nvier-slip boundary conditions, it is convenient to study $\eta^\epsilon_v=(Sv^\epsilon n+\zeta v^\epsilon)_\tau$ and $\eta^\epsilon_H=(SH^\epsilon n+\zeta H^\epsilon)_\tau$. We find that $\eta^\epsilon_v$ and $\eta^\epsilon_H$ satisfy equations with homogeneous Dirichlet boundary conditions, and we shall thus prove a control of $\|(\eta^\epsilon_v,\eta^\epsilon_H)\|_{m-1}$ by performing energy estimates on the equations solved by $(\eta^\epsilon_v,\eta^\epsilon_H)$.  The third step is to estimate $P^\epsilon_1$ and $P^\epsilon_2$. Note that they satisfy nonhomogeneous elliptic equations with Neumann boundary conditions. By using the regularity theory of elliptic equations with Neumann boundary conditions, we get the estimates on the pressure terms. Finally, we need to estimate $\|\nabla v^\epsilon\|_{1,\infty}$ and $\|\nabla H^\epsilon\|_{1,\infty}$. Similar to the second step, we find equivalent quantities $\overline{\eta}^\epsilon_v$ and $\overline{\eta}^\epsilon_H$. However, due to the strong coupling between $\overline{\eta}^\epsilon_v$ and $\overline{\eta}^\epsilon_H$, we cannot deal with the system
on $\overline{\eta}^\epsilon_v$ and $\overline{\eta}^\epsilon_H$ directly as that in \cite{MR}. Instead, we need further to introduce another two quantities  $\eta_1:=\overline{\eta}^\epsilon_v+\overline{\eta}^\epsilon_H $ and $\eta_2:=\overline{\eta}^\epsilon_v-\overline{\eta}^\epsilon_H$. We then estimate $\eta_1$ and $\eta_2$, respectively.

%Due to Theorem \ref{Th1}, we can immediately obtain the following result:
%\begin{Theorem}\label{Th2}
%Let $m$ be an integer satisfying $m>6$ and $\Omega$ be a $C^{m+2}$ domain.
%Assume that $(v_0,H_0)$ satisfy the same conditions as in Theorem \ref{Th1}.
%Let $(v^\epsilon,H^\epsilon)$ be the solution of \eqref{1.1}-\eqref{1.4} obtained in Theorem \ref{Th1}.
% Then the ideal MHD equations \eqref{1.7}-\eqref{1.10}  with the initial data  $(v,H)|_{t=0}=(v_0,H_0)$ has  a unique solution $(v,H)$ satisfying
%\begin{equation}\label{1.17}
%(v,H)\in L^\infty([0,T_m],\mathcal{E}^m),~~ \|(\nabla v,\nabla H)\|_{1,\infty}<\infty,~~on~ [0,T]
%\end{equation}
%and
%\begin{equation}\label{1.18}
%\begin{split}
%\sup_{[0,T_m]}(\|(v^\epsilon,H^\epsilon)-(v,H)\|_{L^2}+\|(v^\epsilon,H^\epsilon)-(v,H)\|_{L^\infty}\rightarrow 0
%\end{split}
%\end{equation}
%when $\epsilon$ tends to zero.
%\end{Theorem}

Based on Theorem \ref{Th1} and Remark \ref{R1.2}, we justify the vanishing viscosity limit as follows:
\begin{Theorem}\label{Th3}
 Assume that $(v_0,H_0)$  belong to $H^3(\Omega)$ and satisfy the same assumptions as in
 Theorem \ref{Th1}. Let $(v,H)$ be the smooth solution of \eqref{1.7}-\eqref{1.10} with the initial data $(v,H)|_{t=0}=(v_0,H_0) $ on $[0,T_1]$.
  Let $(v^\epsilon,H^\epsilon)$ be the solution of \eqref{1.1}-\eqref{1.3} with the boundary condition \eqref{GNB1} and the initial data $(v^\epsilon,H^\epsilon)|_{t=0}=(v_0,H_0)$.
  Then there exists a $T_2=\min\{T_0,T_1\}>0$ such that
\begin{align}
\| v^{\epsilon}&-v \|^2_{L^2}+\| H^{\epsilon}-H\|^2_{L^2}\nonumber\\
&+\epsilon\int_0^t(\|(v^{\epsilon}-v)(s)\|_{H^1}^2+\|( H^{\epsilon}-H)(s)\|_{H^1}^2)\,ds\leq C\epsilon^\frac{3}{2}\quad on \quad [0,T_2],\label{1.19}\\
\|v^{\epsilon}&-v\|^2_{H^1}+\| H^{\epsilon}-H\|^2_{H^1}\nonumber\\
&+\epsilon\int_0^t(\|(v^{\epsilon}-v)(s)\|^2_{H^2}
+\|( H^{\epsilon}-H)(s)\|_{H^2}^2)\,ds\leq C\epsilon^\frac{1}{2}\quad on\quad[0,T_2]\label{1.20}
\end{align}
for $\epsilon$ small enough. Consequently,
\begin{equation}\label{1.21}
\begin{split}
\|v^{\epsilon}&-v\|^p_{W^{1,p}}+\|H^{\epsilon}-H\|^p_{W^{1,p}}\leq C\epsilon^\frac{1}{2}\quad on\quad[0,T_2]
\end{split}
\end{equation}
for $2\leq p<\infty$ and $\epsilon$ small enough, and
\begin{equation}\label{1.22}
\begin{split}
\|v^{\epsilon}-v\|_{L^\infty({[0,T_2]\times\Omega})}+\|H^{\epsilon}-H\|_{L^\infty([0,T_2]\times\Omega)}\leq C \epsilon^\frac{3}{10}.
\end{split}
\end{equation}
\end{Theorem}
We now outline the proof of Theorem \ref{Th3}. Our approaches are similar to those in \cite{YZ}, but due to the strong coupling between magnetic field and velocity field, we meet some new difficulties. We first give the rates of the convergence in  $L^\infty(0,T_2;L^2(\Omega))$ and $L^\infty([0,T_2]\times \Omega)$ by using an elementary energy estimate for the difference of the solutions between the incompressible viscous MHD equations and the ideal incompressible  MHD equations and the Gagliardo-Nirenberg interpolation inequality. Next, because we find that it is very difficult  to estimate some boundary terms caused by multiplying \eqref{5.2} by $\Delta(v^\epsilon-v)$ and \eqref{5.3} by $\Delta(H^\epsilon-H)$ directly in the proof of the rate of the convergence in $L^\infty(0,T_2;H^1(\Omega))$, we turn to consider the Stokes problem \eqref{4.34}-\eqref{4.36}. Indeed, we can get $\|u\|_{H^2}\leq\|P\Delta u\|+\|u\|\,\,\text{for}\,\, u\in W_{B},$ where $W_{B}$ is defined in Lemma \ref{L4.3} and $P$ is Leray projector. Finally, we replace $\Delta(v^\epsilon-v)$ and $\Delta(H^\epsilon-H)$ by $P\Delta(v^\epsilon-v)$ and $P\Delta(H^\epsilon-H)$ to do prove the rates of the convergence in $L^\infty(0,T_2;H^1(\Omega))$ and $L^\infty(0,T_2; W^{1,p}(\Omega))$ .

This paper is organized as follows. In the following section, we give some assumptions on the domain and the definitions on conormal Sobolev spaces, and present some inequalities.
 In Section \ref{Sec3}, we prove a priori energy estimates and give the proof of Theorem \ref{Th1}.
 %and \ref{Th2}.
 Finally, we prove Theorem \ref{Th3} in Section \ref{Sec4}. Throughout the paper, we shall denote by $\|\cdot\|_{H^m }$ and $\|\cdot\|_{W^{1,\infty} }$ the usual Sobolev norms
  in $\Omega$ and  $\|\cdot\|$ for the standard $L^2$ norm. The letter $C$ is a positive number which may change from line to line, but independent of $\epsilon\in(0,1]$ and $|\zeta|\leq 1$.

%%%%%%%%%%%%%%%%%%%%%%%%%%%%%%%%%%%%%%%%%%%%%%%%%%%%%%

\section{Preliminaries}\label{Sec2}
%$ \bigcup\limits_{i=1}^{2} R$睿$ \displaystyle\bigcup_{i=1}^{2} R$
%\[ K \subseteq \cup_{x \in K} \qquad \lim_{n \to \infty} m (E_n) \]

We first state the assumptions on the bounded domain $\Omega\subset\mathbb{R}^3$ and then introduce some norms.
We assume that $\Omega$  has a covering such that
\begin{equation}\label{c}
\Omega\subset\Omega_0 \cup_{k=1}^n\Omega_k,
 \end{equation}
where $\overline{\Omega_0}\subset\Omega$ and in each $\Omega_k$ there exists a function $\psi_k$ such that
\begin{align*}
\Omega\cup\Omega_k=\{\,x=(x_1,x_2,x_3)\,|\,x_3>\psi_k(x_1,x_2)\,\}\cup\Omega_k,\\
\partial\Omega\cup\Omega_k=\{\,x=(x_1,x_2,x_3)\,|\,x_3=\psi_k(x_1,x_2)\,\}\cup\Omega_k.
\end{align*}
We say that $\Omega$ is $C^m$ if the functions $\psi_k$ are $C^m$ functions.

To define the conormal Sobolev spaces, we consider $(Z_k)_{1\leq k\leq N}$, a finite set of generators of vector fields that are tangent to $\partial\Omega$, and set
\begin{align}
H^m_{co}(\Omega):=\big{\{} f\in L^2(\Omega)\,\big{|}\, Z^If\in L^2(\Omega)~~~ \text{for}~~\, |I|\leq m,\,\,\, m\in\mathbb{N}\big{\}},
\end{align}
where $I=(k_1,...,k_m)$, $Z^I:=Z_{k_1}\cdot\cdot \cdot Z_{k_m}$. We define the norm of $H^m_{co}(\Omega)$:
$$ \|f\|^2_m:=\sum_{|I|\leq m}\|Z^If\|^2_{L^2}.$$
We say a vector field, $u$, is in $H^m_{co}(\Omega)$ if each of its components is in $H^m_{co}(\Omega)$ and
$$ \|u\|^2_m:=\sum_{i=1}^3\sum_{|I|\leq m}\|Z^Iu_i\|^2_{L^2}$$
is finite. In the same way, we set
$$ \|f\|_{m,\infty}:=\sum_{|I|\leq m}\|Z^If\|_{L^\infty},$$
$$\|\nabla Z^mu\|^2:=\sum_{|I|\leq m}\|\nabla Z^Iu\|^2_{L^2},$$
and we say that $f\in W^{m,\infty}_{co}(\Omega)$ if $\|f\|_{m,\infty}$ is finite.
By using above covering of $\Omega$, we can assume that each vector field is supported in one of $\{\Omega_i\}_{i=0}^n$. Also, we note that the $\|\cdot\|_{m}$ norm yields a control of the standard $H^m$ norm in $\Omega_0$, whereas if $\Omega_i \cap \partial \Omega\neq {\emptyset}$, there is no control of the
normal derivatives.

Since $\partial\Omega$ is given locally by $x_3=\psi(x_1,x_2)$ (we omit the subscript $k$ for notational convenience), it is convenient to use the coordinates:
\begin{align}\label{3.a}
\Psi:(y,z)\mapsto(y,\psi(y)+z)=x.
\end{align}
A local basis is thus given by the vector fields $(\partial_{y^1},\partial_{y^1},\partial_z)$ where $\partial_{y^1}$ and $\partial_{y^2}$ are tangent to $\partial\Omega$ on the boundary
and in general $\partial_z$ is usually not a normal vector field. We sometimes use the notation $\partial_{y^3}$ for $\partial_z$. By using this parametrization, we can take suitable vector fields compactly supported in $\Omega_i$ in the definition of the $\|\cdot\|_m$ norms:
$$Z_i=\partial_{y^i}=\partial_i+\partial_i\psi\partial_z,~~i=1,2,\quad Z_3=\varphi(z)\partial_z,$$
where $\varphi(z)=\frac{z}{1+z}$ is a smooth and supported function in $(0,+\infty)$ and satisfies
$$\varphi(0)=0,~~\varphi'(0)>0,~~\varphi(z)>0\,\,\,\,\text{for}\,\,\,\,z>0.$$

In this paper, we shall still denote by $\partial_i, i=1,~2,~3$ or $\nabla$ the derivatives with respect to the standard coordinates of $\mathbb{R}^3$. The coordinates of a vector field $u$ in the basis $(\partial_{y^1},\partial_{y^1},\partial_z)$ will be denote by $u^i$, thus
$$u=u^1\partial_{y^1}+u^2\partial_{y^2}+u^3\partial_z.$$
We denote by $u_i$ the coordinates in the standard basis of $\mathbb{R}^3$, i.e.
$$u=u_1\partial_1+u_2\partial_2+u_3\partial_3.$$

Denote by $n$ the unit outward normal vector which is given locally by
$$n(x)=n(\Psi(y,z))=\frac{1}{\sqrt{1+|\nabla\psi(y)|^2}}
\left(
\begin{array}{c}
\partial_1\psi(y)\\
 \partial_2\psi(y)\\
 -1
\end{array}
\right)$$
and by $\Pi$ the orthogonal projection
$$\Pi(x)=\Pi(\Psi(y,z))u=u-[u\cdot n(\Psi(y,z))]n(\Psi(y,z) $$
which gives the orthogonal projector onto the tangent space of the boundary. Note that $n$ and $\Pi$ are defined in the whole $\Omega_k$ and do not depend on $z$. By using these notations, the Navier boundary conditions \eqref{1.4.1} and \eqref{1.4.2} read:
\begin{align}
&v^\epsilon\cdot n=0, \quad \ \Pi\partial_nv^\epsilon=\theta(v^\epsilon)-2\zeta\Pi v^\epsilon,\label{3.2}\\
&H^\epsilon\cdot n=0, \quad \Pi\partial_nH^\epsilon=\theta(H^\epsilon)-2\zeta\Pi H^\epsilon,\label{3.2.1}
\end{align}
where $\theta$ is the shape operator (second fundamental form) of the boundary, $\theta(v^\epsilon):=\Pi((\nabla n)v^\epsilon)$ and $\theta(H^\epsilon):=\Pi((\nabla n)H^\epsilon)$.

First, we introduce a well-known inequality.
\begin{Lemma}[\!\!\cite{RT,XZ1}]\label{L2.5} For $u\in H^s(\Omega) \,(s\geq1)$, we have
\begin{align}
\|u\|_{H^s(\Omega)}\leq\,C\,(\|\nabla\times u\|_{H^{s-1}(\Omega)}+\|\nabla\cdot u\|_{H^{s-1}(\Omega)}+\|u\|_{H^{s-1}(\Omega)}+|u\cdot n|_{H^{s-\frac{1}{2}}(\partial\Omega)}).\nonumber
\end{align}
\end{Lemma}
Next, we introduce the Korn's inequlity which play an important role in energy estimates below.
\begin{Lemma} [Korn's inequality\cite{FP}]\label{L2.1} Let
$\Omega$ be a bounded Lipschitz domain of $\mathbb{R}^3$. There exists a constant $C>0$ depending only on $\Omega$ such that
\begin{equation}\nonumber
\|u\|_{H^1(\Omega)}\leq C\,(\|u\|_{L^2(\Omega)}+\|S(u)\|_{L^2(\Omega)}),\quad \forall~ u\in(H^1(\Omega))^3.
\end{equation}
\end{Lemma}
Third, we also need the following anistropic Sobolev embedding and trace estimates.
\begin{Lemma}[\!\!\cite{MR,WXY}]\label{L2.2} Let $m_1\geq0$ and $m_2\geq0$ be integers, $u\in H^{m_1}_{co}(\Omega)\cap H^{m_2}_{co}(\Omega)$ and $\nabla u\in H^{m_2}_{co}(\Omega)$.
Then we have
\begin{align*}
&\|u\|^2_{{L^\infty(\Omega)}}\leq C\,(\|\nabla u\|_{m_2}+\|u\|_{m_2})\|u\|_{m_1},\quad m_1+m_2\geq 3,\\
&|u|^2_{H^s(\partial\Omega)}\leq C\,(\|\nabla u\|_{m_2}+\|u\|_{m_2})\|u\|_{m_1},\quad m_1+m_2\geq 2s\geq 0.
\end{align*}
\end{Lemma}
%\begin{proof}
%We can refer to Proposition $2.2$ in \cite{MR} and Proposition $2.3$ in \cite{WXY}.
%\end{proof}
Fourth, we introduce the following Gagliardo-Nirenberg-Moser inequality which will be used frequently.
\begin{Lemma}[\!\!\cite{GO}]\label{L2.3}Let $u, v\in L^\infty(\Omega)\cap H^k_{co}(\Omega)$, we have
\begin{equation}\nonumber
\|Z^{\alpha_1}uZ^{\alpha_2}v\|\leq C\, (\|u\|_{L^\infty(\Omega)}\|v\|_k+\|v\|_{L^\infty(\Omega)}\|u\|_k),\quad |\alpha_1|+|\alpha_2|=k.
\end{equation}
\end{Lemma}

Finally, the following decomposition on $H^s$ contributes to the proof of the convergence rate in $H^1$.
\begin{Lemma}[\!\!\cite{YZ1}]\label{L2.4}For $H^s(\Omega) \,(s\geq0)$, we have
\begin{align}
H^s(\Omega)=\nabla\times(FH\cap H^{s+1}(\Omega))\oplus(HG\cap H^s(\Omega))\oplus(GG\cap H^s(\Omega)),\nonumber
\end{align}
where
\begin{align*}
&FH=\big{\{}u\, {|}\,u=\nabla\times \varphi,\,\, \varphi\in H^1(\Omega),\,\, \nabla\cdot \varphi=0,\,\, n\times \varphi=0\,\,\, \text{on} \,\,\,\partial\Omega\big{\}},\\
&HG=\Big{\{}u\, {|}\,u=\nabla\varphi,\,\,\,\Delta\varphi=0,\,\,\varphi=c_i\,\,\, \text{on} \,\,\,\Gamma_i,\,\,\,\bigcup_i\Gamma_i=\partial\Omega\Big{\}},\\
&GG=\big{\{}u\, {|}\,u=\nabla\varphi,\,\,\varphi\in H^1_0(\Omega)\big{\}}.
\end{align*}
\end{Lemma}
\section{A priori estimates and proof of Theorem \ref{Th1}}\label{Sec3}

The main aim of this section is to prove the following a priori estimates which is the crucial step in the proof of Theorem \ref{Th1}.

\begin{Theorem}\label{Th4}
For $m>6$ and a $C^m$ domain $\Omega$, there exists a constant $C>0$, independent of $\epsilon\in(0,1]$ and $|\zeta|\leq1$, such that for any sufficiently smooth solution defined on $[0,T]$ of the problem \eqref{1.1}-\eqref{1.4} in $\Omega$, we have
\begin{equation}\label{TH4}
N_m(t)\leq C\,\Big{(} N_m(0)+(1+t+\epsilon^3t^2)\int_0^t(N_m^2(s)+N_m(s))\,ds\Big{)},\quad\forall\,\, t\in[0,T],
\end{equation}
where
\begin{align}\label{N_m}
\!\!\! N_m(t):=\|v^\epsilon\|_m^2+\|\nabla v^\epsilon\|_{m-1}^2+\|\nabla v^\epsilon\|^2_{1,\infty}+\|H^\epsilon\|_m^2+\|\nabla H^\epsilon\|_{m-1}^2+\|\nabla H^\epsilon\|^2_{1,\infty}.
\end{align}
\end{Theorem}

Since the proof of Theorem \ref{Th4} is quite complicated and lengthy, we divided the proof into the following subsections.
\subsection{Conormal Energy Estimates} In this subsection, we first give the basic $L^2$ energy estimates.
\begin{Lemma}\label{L3.1} For a smooth solution of the problem \eqref{1.1}-\eqref{1.4}, we have
\begin{align}\label{3.3.1}
\frac{1 }{2}\frac{d }{dt}(\| v^\epsilon(t)\| ^{2}+\| H^\epsilon(t)\| ^{2})&+2\epsilon(\| Sv^\epsilon\|^2+\| SH^\epsilon\|^2)\nonumber\\
&+2\epsilon\zeta\int_{\partial\Omega}\big{(}\!\mid v^\epsilon_\tau\mid^{2}+\mid H^\epsilon_\tau\mid^{2}\! \big{)}=0
\end{align}
for every $\epsilon\in(0,1]$ and $|\zeta|\leq 1$.
\end{Lemma}

\begin{proof}
Multiplying \eqref{1.1} and \eqref{1.2} by $v^\epsilon$ and $H^\epsilon$ respectively, using the boundary condition, and integrating by parts, we obtain
\begin{align}
\frac{1}{2}\frac{d}{dt}(\|v^\epsilon\|^2+\|H^\epsilon\|^2)&-\epsilon(\Delta v^\epsilon,v^\epsilon)-\epsilon(\Delta H^\epsilon,H^\epsilon)\nonumber\\
&-(H^\epsilon\cdot\nabla H^\epsilon,v^\epsilon)-(H^\epsilon\cdot\nabla v^\epsilon,H^\epsilon)=0,\label{2.2}
\end{align}
where $(\cdot,\cdot)$ stands for the $L^2$ scalar product. By integrating by parts and using the boundary conditions,  we get
\begin{equation}\nonumber
(H^\epsilon\cdot\nabla H^\epsilon,v^\epsilon)+(H^\epsilon\cdot\nabla v^\epsilon,H^\epsilon)=0.
\end{equation}
Now, let us treat the terms with the viscous coefficient $\epsilon$ in \eqref{2.2}. Thanks to integrations by parts and the boundary condition \eqref{1.4.1}, we have
\begin{align}\label{2.4}
(\epsilon\Delta v^\epsilon,v^\epsilon)=2\epsilon(\nabla\cdot Sv^\epsilon,v^\epsilon)&=-2\epsilon\|Sv^\epsilon\|^2+2\epsilon\ \int_{\partial \Omega}((Sv^\epsilon)\cdot n)\cdot v^\epsilon \nonumber\\
&=-2\epsilon\| Sv^\epsilon\|^2-2\epsilon\zeta\int_{\partial\Omega}|v^\epsilon_\tau|^2.
\end{align}

Similarly, we have
\begin{equation}\label{2.5}
(\epsilon\Delta H^\epsilon,H^\epsilon)=-2\epsilon\| SH^\epsilon\|^2-2\epsilon\zeta\int_{\partial\Omega}|H^\epsilon_\tau|^2.
\end{equation}
Putting \eqref{2.4} and \eqref{2.5} into \eqref{2.2}, we then obtain \eqref{3.3.1}.
\end{proof}

Now, we turn to the higher order energy estimates.
\begin{Lemma}\label{L3.2} For every $m\geq 0$, a smooth solution of the problem \eqref{1.1}-\eqref{1.4} satisfies the estimate
\begin{align}\label{3.4}
\frac{d }{dt}&(\|v^\epsilon(t)\|_{m}^{2}+\|H^\epsilon(t)\|_{m}^{2})+ \epsilon (\|\nabla v^\epsilon\|_{m}^{2}+\|\nabla H^\epsilon\|_{m}^{2})\nonumber\\
\leq& C\,(1+\| v^\epsilon\|_{W^{1,\infty}}+\| H^\epsilon\|_{W^{1,\infty}})( \| v^\epsilon\|_{m}^{2}+ \|\nabla v^\epsilon\|_{m-1}^{2}+\| H^\epsilon\|_{m}^{2}+ \| \nabla H^\epsilon\|_{m-1}^{2})\nonumber\\
&+C\,\| \nabla^2 P^\epsilon_1\|_{m-1}\| v^\epsilon\|_m+C\epsilon^{-1}\| \nabla P^\epsilon_2\|_{m-1}^2,
\end{align}
where the pressure $P^\epsilon:=p^\epsilon-\frac{1}{2}(|v^\epsilon|^2-|H^\epsilon|^2):=P^\epsilon_1+P^\epsilon_2$. Here, $P^\epsilon_1$ is the``Euler" part of the pressure which solves
\begin{equation}\label{3.8}
\left\{\begin{array}{l}
\Delta P^\epsilon_1=-\nabla\cdot(v^\epsilon\cdot\nabla v^\epsilon-H^\epsilon\cdot\nabla H^\epsilon)\quad \text{in}\quad \Omega,\\
\partial_nP^\epsilon_1=-(v^\epsilon\cdot\nabla v^\epsilon-H^\epsilon\cdot\nabla H^\epsilon)\cdot n\quad \text{on}\quad \partial\Omega
\end{array}
\right.
\end{equation}
and $P^\epsilon_2$ is the ``Navier-Stokes" part of the pressure which solves
\begin{equation}\label{3.9}
\left\{\begin{array}{l}
\Delta P^\epsilon_2=0\quad \text{in}\quad \Omega,\\
\partial_nP^\epsilon_2=\epsilon\Delta v^\epsilon\cdot n\quad \text{on}\quad \partial\Omega.
\end{array}
\right.
\end{equation}
\end{Lemma}
\begin{proof}
The estimate for $m=0$ has been given in Lemma \ref{L3.1}. Now we assume Lemma \ref{L3.2} have been proved for $|\alpha|\leq m-1$ and prove that it holds for $|\alpha|=m$. We apply $Z^\alpha$ to \eqref{1.1}-\eqref{1.2} for $|\alpha|=m$ to obtain
\begin{align*}
&\partial_tZ^\alpha v^\epsilon+v^\epsilon\cdot\nabla Z^\alpha v^\epsilon-H^\epsilon\cdot\nabla Z^\alpha H^\epsilon+ Z^\alpha\nabla P^\epsilon=\epsilon Z^\alpha\Delta v^\epsilon+\mathcal{C}_1,\\
&\partial_tZ^\alpha H^\epsilon+v^\epsilon\cdot\nabla Z^\alpha H^\epsilon-H^\epsilon\cdot\nabla Z^\alpha v^\epsilon=\epsilon Z^\alpha \Delta H^\epsilon+\mathcal{C}_2,
\end{align*}
where
\begin{align*}
&\mathcal{C}_1:=-[Z^\alpha,v^\epsilon\cdot\nabla]v^\epsilon+[Z^\alpha,H^\epsilon\cdot\nabla]H^\epsilon,\\
&\mathcal{C}_2:=-[Z^\alpha,v^\epsilon\cdot\nabla]H^\epsilon+[Z^\alpha,H^\epsilon\cdot\nabla]v^\epsilon.
\end{align*}
Consequently, we get from the standard energy estimate that
\begin{align}\label{3.11}
 \frac{1}{2}\frac{d}{dt}(\|Z^\alpha v^\epsilon\|^2+\| Z^\alpha H^\epsilon\|^2)
=&\epsilon(Z^\alpha\Delta v^\epsilon,Z^\alpha v^\epsilon)+\epsilon(Z^\alpha\Delta H^\epsilon,Z^\alpha H^\epsilon)\nonumber\\
&+(\mathcal{C}_1,Z^\alpha v^\epsilon)+(\mathcal{C}_2,Z^\alpha H^\epsilon)-(Z^\alpha \nabla P^\epsilon,Z^\alpha v^\epsilon).
\end{align}

First, by Lemma \ref{L2.3}, we obtain
\begin{align}\label{3.12}
|(\mathcal{C}_1,Z^\alpha v^\epsilon)+(\mathcal{C}_2,Z^\alpha H^\epsilon)|\leq \,&C\,(\| v^\epsilon\|_{W^{1,\infty}}+\| H^\epsilon\|_{W^{1,\infty}})\nonumber\\
&(\| v^\epsilon\|_{m}^{2}+ \|\nabla v^\epsilon\|_{m-1}^{2}+\| H^\epsilon\|_{m}^{2}+ \| \nabla H^\epsilon\|_{m-1}^{2}).
\end{align}

Next, we estimate the terms with the viscosity coefficient $\epsilon$. We have
\begin{align}\label{3.13}
\epsilon\int_\Omega Z^\alpha \Delta v^\epsilon\cdot Z^\alpha v^\epsilon=2\epsilon\int_\Omega(\nabla\cdot Z^\alpha Sv^\epsilon)\cdot Z^\alpha v^\epsilon
+2\epsilon\int_\Omega([Z^\alpha,\nabla\cdot]Sv^\epsilon)\cdot Z^\alpha v^\epsilon.
\end{align}
Now, by integrating by parts, we get from the first term on the right hand side of \eqref{3.13} that
\begin{align}\label{3.14}
\epsilon\int_\Omega(\nabla\cdot Z^\alpha Sv^\epsilon)\cdot Z^\alpha v^\epsilon=&-\epsilon\int_\Omega Z^\alpha Sv^\epsilon\cdot\nabla Z^\alpha v^\epsilon
+\epsilon\int_{\partial\Omega}((Z^\alpha Sv^\epsilon)\cdot n)\cdot Z^\alpha v^\epsilon\nonumber\\
=&-\epsilon\|S(Z^\alpha v^\epsilon)\|^2-\epsilon\int_\Omega[Z^\alpha,S]v^\epsilon\cdot \nabla Z^\alpha v^\epsilon\nonumber\\
&+\epsilon\int_{\partial\Omega}((Z^\alpha Sv^\epsilon)\cdot n)\cdot Z^\alpha v^\epsilon.
\end{align}
Thanks to Lemma \ref{L2.1}, there exists a $c_0>0$ such that
\begin{align}\label{3.15}
\epsilon\int_\Omega(\nabla\cdot Z^\alpha Sv^\epsilon)\cdot Z^\alpha v^\epsilon
\leq&-c_0\epsilon\|\nabla(Z^\alpha v^\epsilon)\|^2+C\|v^\epsilon\|^2_m+C\epsilon\|\nabla Z^\alpha v^\epsilon\|\|\nabla v^\epsilon\|_{m-1}\nonumber\\
&+\epsilon\int_{\partial\Omega}((Z^\alpha Sv^\epsilon)\cdot n)\cdot Z^\alpha v^\epsilon.
\end{align}

It remains to estimate the boundary term of \eqref{3.15}. Before we treat the boundary term, we have the following observations. Due to the Navier boundary condition \eqref{3.2}, we get
\begin{align}\label{3.16}
|\Pi\partial_nv^\epsilon|_{H^m(\partial\Omega)}\leq|\theta(v^\epsilon)|_{H^m(\partial\Omega)}+2\zeta|\Pi v^\epsilon|_{H^m(\partial\Omega)}\leq C\,|v^\epsilon|_{H^m(\partial\Omega)}.
\end{align}
To estimate the normal part of $\partial_nv^\epsilon$, we can use the divergence free condition to write
\begin{equation}\label{3.17}
\nabla\cdot v^\epsilon=\partial_nv^\epsilon\cdot n+(\Pi\partial_{y_1}v^\epsilon)^1+(\Pi\partial_{y_2}v^\epsilon)^2.
\end{equation}
Hence, we easily get
\begin{equation}\label{3.18}
|\partial_nv^\epsilon\cdot n|_{H^{m-1}(\partial\Omega)}\leq C\,|v^\epsilon|_{H^{m}(\partial\Omega)}.
\end{equation}
From \eqref{3.16} and \eqref{3.18}, we have
\begin{equation}\label{3.19}
|\nabla v^\epsilon|_{H^{m-1}(\partial\Omega)}\leq C\,|v^\epsilon|_{H^{m}(\partial\Omega)}.
\end{equation}
Thanks to $v^\epsilon\cdot n=0$ on the boundary, we immediately obtain that
\begin{equation}\label{3.20}
|(Z^\alpha v^\epsilon)\cdot n|_{H^1(\partial\Omega)}\leq C\,|v^\epsilon|_{H^{m}(\partial\Omega)},\quad|\alpha|=m.
\end{equation}

Now we return to deal with the boundary term of \eqref{3.15} as follows
\begin{align*}
\int_{\partial\Omega}((Z^\alpha Sv^\epsilon)\cdot n)\cdot Z^\alpha v^\epsilon=&\int_{\partial\Omega}Z^\alpha(\Pi(Sv^\epsilon\cdot n))\cdot\Pi Z^\alpha v^\epsilon\nonumber\\
&+\int_{\partial\Omega}Z^\alpha(\partial_nv^\epsilon\cdot n)Z^\alpha v^\epsilon\cdot n+\mathcal{C}_b^v,
\end{align*}
where
\begin{align*}
\mathcal{C}_b^v=\int_{\partial\Omega}[\Pi,Z^\alpha](Sv^\epsilon\cdot n)\cdot\Pi Z^\alpha v^\epsilon+\int_{\partial\Omega}[n,Z^\alpha](Sv^\epsilon\cdot n)Z^\alpha v^\epsilon\cdot n.
\end{align*}
Due to \eqref{3.19} and \eqref{1.4.1}, we can easily obtain that
\begin{align}
&|C_b^v|\leq C\,|\nabla v^\epsilon|_{H^{m-1}(\partial\Omega)}|v^\epsilon|_{H^{m}(\partial\Omega)}\leq C\,|v^\epsilon|_{H^{m}(\partial\Omega)}^2,\label{3.22}\\
&\Big{|}\int_{\partial\Omega}Z^\alpha(\Pi(Sv^\epsilon\cdot n))\cdot\Pi Z^\alpha v^\epsilon\Big{|}\leq C\,|v^\epsilon|_{H^{m}(\partial\Omega)}^2.\label{3.23}
\end{align}
By integrating by parts along the boundary, we have that
\begin{equation}\label{3.24}
\Big{|}\int_{\partial\Omega}Z^\alpha(\partial_nv^\epsilon\cdot n)Z^\alpha v^\epsilon\cdot n\Big{|}\leq C\,|\partial_nv^\epsilon\cdot n|_{H^{m-1}(\partial\Omega)}|Z^\alpha v^\epsilon\cdot n|_{H^1(\partial\Omega)}\leq C\,|v^\epsilon|^2_{H^m(\partial\Omega)}.
\end{equation}
Hence, we get from \eqref{3.14}, \eqref{3.15}, and \eqref{3.22}-\eqref{3.24} that
\begin{align}\label{3.25}
\epsilon\int_\Omega(\nabla\cdot Z^\alpha Sv^\epsilon)\cdot Z^\alpha v^\epsilon
\leq\, &C\,(\|v^\epsilon\|^2_m+\epsilon\|\nabla Z^m v^\epsilon\|\|\nabla v^\epsilon\|_{m-1}+\epsilon|v^\epsilon|^2_{H^m(\partial\Omega)})\nonumber\\
&-c_0\epsilon\|\nabla(Z^\alpha v^\epsilon)\|^2.
\end{align}

Next, we deal with the second term of the right hand side of \eqref{3.13}, i.e.\linebreak
$\epsilon\int_\Omega([Z^\alpha,\nabla\cdot]Sv^\epsilon)\cdot Z^\alpha v^\epsilon$. We can expand it as a sum of terms under the form
\begin{equation}
\epsilon\int_\Omega\beta_k\partial_k(Z^{\tilde{\alpha}}Sv^\epsilon)\cdot Z^\alpha v^\epsilon,\quad |\tilde{\alpha}|\leq m-1.\nonumber
\end{equation}
By using integrations by parts and \eqref{3.19}, we have
\begin{equation}\label{3.27}
\epsilon\Big{|}\int_\Omega\beta_k\partial_k(Z^{\tilde{\alpha}}Sv^\epsilon)\cdot Z^\alpha v^\epsilon\Big{|}\leq\, C\,\epsilon(\|\nabla Z^{m-1}v^\epsilon\|\|\nabla Z^mv^\epsilon\|
+\|v^\epsilon\|^2_m+|v^\epsilon|_{H^{m}(\partial\Omega)}^2).
\end{equation}
Consequently, from \eqref{3.25} and \eqref{3.27}, we get
\begin{align}\label{3.28}
\epsilon\Big{|}\int_\Omega Z^\alpha \Delta v^\epsilon\cdot Z^\alpha v^\epsilon\Big{|}\leq&\, C\,\big{\{}\|v^\epsilon\|^2_m+\epsilon\|\nabla Z^m v^\epsilon\|\|\nabla v^\epsilon\|_{m-1}+\epsilon|v^\epsilon|^2_{H^m(\partial\Omega)}\nonumber\\
&+\epsilon\|\nabla Z^m v^\epsilon\|\|\nabla Z^{m-1}v^\epsilon\|_{m-1}\big{\}}
-c_0\epsilon\|\nabla(Z^\alpha v^\epsilon)\|^2.
\end{align}

Similarly, for the term $\epsilon(Z^\alpha \Delta H^\epsilon\cdot Z^\alpha H^\epsilon)$ in the right hand side of \eqref{3.11}, we have
\begin{align}\label{3.29}
\epsilon\Big{|}\int_\Omega Z^\alpha \Delta H^\epsilon\cdot Z^\alpha H^\epsilon\Big{|}\leq &\,C\,\big{\{}\|H^\epsilon\|^2_m+\epsilon\|\nabla Z^m H^\epsilon\|\|\nabla H^\epsilon\|_{m-1}+\epsilon|H^\epsilon|^2_{H^m(\partial\Omega)}\nonumber\\
&+\epsilon\|\nabla Z^m H^\epsilon\|\|\nabla Z^{m-1}H^\epsilon\|_{m-1}\big{\}}
-c_0\epsilon\|\nabla(Z^\alpha H^\epsilon)\|^2.
\end{align}

Finally, we estimate the term involving the pressure $P^\epsilon$ in \eqref{3.11}. We have
\begin{align}\label{3.30}
\Big{|}\int_{\Omega}Z^\alpha \nabla P^\epsilon\cdot Z^\alpha v^\epsilon\Big{|}\leq&\, \|\nabla^2P_1^\epsilon\|_{m-1}\|v^\epsilon\|_m+\Big{|}\int_{\Omega}Z^\alpha \nabla P^\epsilon_2\cdot Z^\alpha v^\epsilon\Big{|}\nonumber\\
\leq\,&\|\nabla^2P_1^\epsilon\|_{m-1}\|v^\epsilon\|_m+C\|\nabla P_2^\epsilon\|_{m-1}\|v^\epsilon\|_m\nonumber\\
&+\Big{|}\int_{\Omega}\nabla Z^\alpha  P^\epsilon_2\cdot Z^\alpha v^\epsilon\Big{|}.
\end{align}
Now, we focus on the last term of \eqref{3.30}. By integrating by parts, we obtain
\begin{equation}
\Big{|}\int_{\Omega}\nabla Z^\alpha  P^\epsilon_2\cdot Z^\alpha v^\epsilon\Big{|}\leq C\,\|\nabla P^\epsilon_2\|_{m-1}\|\nabla Z^\alpha v^\epsilon\|+\Big{|}\int_{\partial\Omega}Z^\alpha  P^\epsilon_2Z^\alpha v^\epsilon\cdot n\Big{|}.\nonumber
\end{equation}
To estimate the boundary term, we note that when $m=1$, \eqref{3.4} can be obtained easily. Here, we assume that $m\geq2$. By integrating by parts along the boundary, we get
\begin{equation}\nonumber
\Big{|}\int_{\partial\Omega}Z^\alpha  P^\epsilon_2Z^\alpha v^\epsilon\cdot n\Big{|}\leq C\,|Z^{\widetilde{\alpha}}P^\epsilon_2|_{L^2(\partial\Omega)}|Z^\alpha v^\epsilon\cdot n|_{H^1(\partial\Omega)},
\end{equation}
where $|\widetilde{\alpha}|=m-1$. By using \eqref{3.20} and Lemma \ref{L2.2}, we have
\begin{align}\label{3.33}
\Big{|}\int_{\Omega}Z^\alpha \nabla P^\epsilon\cdot Z^\alpha v^\epsilon\Big{|}\leq&\|\nabla^2P_1^\epsilon\|_{m-1}\|v^\epsilon\|_m+C\,\|\nabla P_2^\epsilon\|_{m-1}\|v^\epsilon\|_m\nonumber\\
&+C\,\|\nabla P^\epsilon_2\|_{m-1}\|\nabla Z^\alpha v^\epsilon\|+\epsilon^{-1}\|\nabla P^\epsilon_2\|_{m-1}^2\nonumber\\
&+\epsilon(\|\nabla v^\epsilon\|_m\|v^\epsilon\|_{m}+\|v^\epsilon\|^2_{m}).
\end{align}

Consequently, from \eqref{3.12}, \eqref{3.28}-\eqref{3.30} and \eqref{3.33}, we have
\begin{align*}
\frac{1}{2}\frac{d}{dt}&(\|Z^\alpha v^\epsilon\|^2+\| Z^\alpha H^\epsilon\|^2)+c_0\epsilon\|\nabla(Z^\alpha v^\epsilon)\|^2
+c_0\epsilon\|\nabla(Z^\alpha v^\epsilon)\|^2\nonumber\\
\leq&\, C\,(1+\| v^\epsilon\|_{W^{1,\infty}}+\| H^\epsilon\|_{W^{1,\infty}})(\| v^\epsilon\|_{m}^{2}+ \|\nabla v^\epsilon\|_{m-1}^{2}+\| H^\epsilon\|_{m}^{2}+ \| \nabla H^\epsilon\|_{m-1}^{2})\nonumber\\
&+C\,\big{\{}\epsilon\|\nabla Z^m v^\epsilon\|\|\nabla v^\epsilon\|_{m-1}+\epsilon|v^\epsilon|^2_{H^m(\partial\Omega)}
+\epsilon\|\nabla Z^m v^\epsilon\|\|\nabla Z^{m-1}v^\epsilon\|_{m-1}\nonumber\\
&+\epsilon\|\nabla Z^m H^\epsilon\|\|\nabla H^\epsilon\|_{m-1}+\epsilon|H^\epsilon|^2_{H^m(\partial\Omega)}
+\epsilon\|\nabla Z^m H^\epsilon\|\|\nabla Z^{m-1}H^\epsilon\|_{m-1}\nonumber\\
&+\|\nabla^2P_1^\epsilon\|_{m-1}\|v^\epsilon\|_m+\|\nabla P_2^\epsilon\|_{m-1}\|v^\epsilon\|_m
+\|\nabla P^\epsilon_2\|_{m-1}\|\nabla Z^mv^\epsilon\|\nonumber\\
&+\epsilon^{-1}\|\nabla P^\epsilon_2\|_{m-1}^2+\epsilon(\|\nabla v^\epsilon\|_m\|v^\epsilon\|_{m}+\|v^\epsilon\|^2_{m})\big{\}}.
\end{align*}
Next, by using Lemma \ref{L2.2}, Young's inequality, the assumptions with respect to $|\alpha|\leq m-1$, we have
\begin{align*}
\frac{1}{2}\frac{d}{dt}&(\|v^\epsilon\|^2_m+\|H^\epsilon\|^2_m)+c_0\epsilon\|\nabla v^\epsilon\|^2_{m-1}
+c_0\epsilon\|\nabla v^\epsilon\|^2_{m-1}\nonumber\\
\leq&\, C(1+\| v^\epsilon\|_{W^{1,\infty}}+\| H^\epsilon\|_{W^{1,\infty}})(\| v^\epsilon\|_{m}^{2}+ \|\nabla v^\epsilon\|_{m-1}^{2}+\| H^\epsilon\|_{m}^{2}+ \| \nabla H^\epsilon\|_{m-1}^{2})\nonumber\\
&+ C(\|\nabla^2P_1^\epsilon\|_{m-1}\|v^\epsilon\|_m+\epsilon^{-1}\|\nabla P_2^\epsilon\|_{m-1}^2).
\end{align*}
This ends the proof of Lemma \ref{L3.2}.
\end{proof}

\subsection{Normal Derivative Estimates.}
In this subsection, we provide the estimates for $\|\nabla v^\epsilon\|_{m-1}$ and $\|\nabla H^\epsilon\|_{m-1}$. Noticing that
$$\|\chi\partial_{y^i}v^\epsilon\|_{m-1}\leq C\,\|v^\epsilon\|_{m},\quad\|\chi\partial_{y^i}H^\epsilon\|_{m-1}\leq C\,\|H^\epsilon\|_{m},\quad i=1,2,$$
it suffices to estimate $\|\chi\partial_nv^\epsilon\|_{m-1}$ and $\|\chi\partial_nH^\epsilon\|_{m-1}$, where $\chi$ is compactly supported in one of the $\Omega_i$ and with value one in a vicinity of the boundary. We shall thus use the local coordinates \eqref{3.a}.

Due to \eqref{3.17}, we immediately obtain that
\begin{align}\label{3.37}
\|\chi\partial_nv^\epsilon\cdot n\|_{m-1}\leq C\,\|v^\epsilon\|_m,\quad \|\chi\partial_nH^\epsilon\cdot n\|_{m-1}\leq C\,\|H^\epsilon\|_m.
\end{align}
Thus, it remains to estimate $\|\chi\Pi\partial_nv^\epsilon\|_{m-1}$ and $\|\chi\Pi\partial_nH^\epsilon\|_{m-1}$. We define
\begin{align*}
&\eta^\epsilon_v:=\chi\Pi((\nabla v^\epsilon+(\nabla v^\epsilon)^t)n)+2\zeta\chi\Pi v^\epsilon,\\
&\eta^\epsilon_H:=\chi\Pi((\nabla H^\epsilon+(\nabla H^\epsilon)^t)n)+2\zeta\chi\Pi H^\epsilon.
\end{align*}
In view of the Navier boundary conditions \eqref{1.4.1} and \eqref{1.4.2}, we have
\begin{align*}
\eta^\epsilon_v=0,\quad\eta^\epsilon_H=0\quad \text{on}\quad\partial\Omega.
\end{align*}
Moreover, since $\eta^\epsilon_v$ and $\eta^\epsilon_H$ have another forms in the vicinity of the boundary $\partial\Omega$:
\begin{align}
&\eta^\epsilon_v=\chi\Pi\partial_nv^\epsilon+\chi\Pi(\nabla(v^\epsilon\cdot n)-\nabla n\cdot v^\epsilon-v^\epsilon\times(\nabla\times n)+2\zeta v^\epsilon)\label{3.41},\\
&\eta^\epsilon_H=\chi\Pi\partial_nH^\epsilon+\chi\Pi(\nabla(H^\epsilon\cdot n)-\nabla n\cdot H^\epsilon-H^\epsilon\times(\nabla\times n)+2\zeta H^\epsilon)\label{3.42},
\end{align}
we easily get that
\begin{align*}
\|\chi\Pi\partial_nv^\epsilon\|_{m-1}\leq&\, C\,(\|\eta^\epsilon_v\|_{m-1}+\|v^\epsilon\|_{m}+\|\partial_nv^\epsilon\cdot n\|_{m-1})\nonumber\\
\leq&\, C\,(\|\eta^\epsilon_v\|_{m-1}+\|v^\epsilon\|_{m}),\\
\|\chi\Pi\partial_nH^\epsilon\|_{m-1}\leq&\, C\,(\,\|\eta^\epsilon_H\|_{m-1}+\|H^\epsilon\|_{m}+\|\partial_nH^\epsilon\cdot n\|_{m-1})\nonumber\\
\leq&\, C\,(\|\eta^\epsilon_H\|_{m-1}+\|H^\epsilon\|_{m}).
\end{align*}
Hence, it remains to estimate $\|\eta^\epsilon_v\|_{m-1}$ and $\|\eta^\epsilon_H\|_{m-1}$.\par

We have the following conormal estimates for $\eta^\epsilon_v$ and $\eta^\epsilon_H$.
\begin{Lemma}\label{L3.3} For every $m\geq1$, we have
\begin{align}\label{3.3}
\frac{1}{2}\frac{d}{dt}&(\|\eta^\epsilon_v\|_{m-1}^2+\|\eta^\epsilon_H\|_{m-1}^2)+\epsilon(\|\nabla\eta^\epsilon_v\|^2_{m-1}+\|\nabla \eta^\epsilon_H\|^2_{m-1})\nonumber\\
\leq &\,C\,(1+\|v^\epsilon\|_{2,\infty}+\|\nabla v^\epsilon\|_{1,\infty}+\|H^\epsilon\|_{2,\infty}+\|\nabla H^\epsilon\|_{1,\infty})\nonumber\\
&\times(\|\eta^\epsilon_v\|_{m-1}^2+\|\eta^\epsilon_H\|_{m-1}^2+\|v^\epsilon\|_m^2+\|H^\epsilon\|_m^2+\|\nabla v^\epsilon\|_{m-1}^2+\|\nabla H^\epsilon\|_{m-1}^2)\nonumber\\
&+C\,\big{(}(\|\eta^\epsilon_v\|_{m-1}+\|v^\epsilon\|_m)(\|\nabla^2P^\epsilon_1\|_{m-1}+\|\nabla P^\epsilon\|_{m-1})+\epsilon^{-1}\|\nabla P^\epsilon_2\|^2_{m-1}\big{)}.
\end{align}
\end{Lemma}
\begin{proof}
Setting $M_v=\nabla v^\epsilon $ and $M_H=\nabla H^\epsilon $, we get from \eqref{1.1}-\eqref{1.2} that
\begin{align*}
&\partial_tM_v-\epsilon\Delta M_v+v^\epsilon\cdot\nabla M_v-H^\epsilon\cdot\nabla M_H=( M_H)^2-(M_v)^2-\nabla^2P^\epsilon,\\
&\partial_tM_H-\epsilon\Delta M_H+v^\epsilon\cdot\nabla M_H-H^\epsilon\cdot\nabla M_v= M_v M_H-M_H M_v.
\end{align*}
Hence, $\eta^\epsilon_v$ and $\eta^\epsilon_H$ solve the equations
\begin{align}
&\partial_t\eta^\epsilon_v-\epsilon\Delta\eta^\epsilon_v+v^\epsilon\cdot\nabla\eta^\epsilon_v-H^\epsilon\cdot\nabla\eta^\epsilon_H=F_v^b+F_v^\chi+F_v^\kappa
-2\chi\Pi(\nabla^2P^\epsilon n),\label{3.48}\\
&\partial_t\eta^\epsilon_H-\epsilon\Delta\eta^\epsilon_H+v^\epsilon\cdot\nabla\eta^\epsilon_H-H^\epsilon\cdot\nabla\eta^\epsilon_v=F_H^b+F_H^\chi+F_H^\kappa,\label{3.49}
\end{align}
where
\begin{align*}
F_v^b=&-\chi\Pi((\nabla v^\epsilon)^2+((\nabla v^\epsilon)^t)^2-(\nabla H^\epsilon)^2-((\nabla H^\epsilon)^t)^2)n-2\zeta\chi\Pi\nabla P^\epsilon,\\
F_v^\chi=&-\epsilon\Delta\chi(\Pi2 Sv^\epsilon n+2\zeta\Pi v^\epsilon)-2\epsilon\nabla\chi\cdot\nabla(\Pi2 Sv^\epsilon n+2\zeta\Pi v^\epsilon)\\
&+(v^\epsilon\cdot\nabla\chi)\Pi(2Sv^\epsilon n+2\zeta v^\epsilon)-(H^\epsilon\cdot\nabla\chi)\Pi(2SH^\epsilon n+2\zeta H^\epsilon),\\
F_v^\kappa=&\chi(v^\epsilon\cdot\nabla\Pi)(2Sv^\epsilon n+2\zeta v^\epsilon)+\chi\Pi(2Sv^\epsilon (v^\epsilon\cdot\nabla)n)-\epsilon\chi(\Delta\Pi)(2Sv^\epsilon n+2\zeta v^\epsilon)\\
&-2\epsilon\chi\nabla\Pi\cdot\nabla(2Sv^\epsilon n+2\zeta v^\epsilon)-\epsilon\chi\Pi(2Sv^\epsilon\Delta n+2\nabla Sv^\epsilon\cdot\nabla n)\\
&-\chi(H\cdot\nabla\Pi)(2SH^\epsilon n+2\zeta H^\epsilon)-\chi\Pi(2SH^\epsilon(H\cdot\nabla)n),\\
F_H^b=&-\chi\Pi(M_H M_v+M_v^tM_H^t-M_v M_H-M_H^t M_v^t)n,\\
F_H^\chi=&-\epsilon\Delta\chi(\Pi2 SH^\epsilon n+2\zeta H^\epsilon)-2\epsilon\nabla\chi\cdot\nabla(\Pi2 SH^\epsilon n+2\zeta H^\epsilon)\\
&+(v^\epsilon\cdot\nabla\chi)\Pi(2SH^\epsilon n+2\zeta\Pi H^\epsilon)-(H^\epsilon\cdot\nabla\chi)\Pi(2Sv^\epsilon n+2\zeta\Pi v^\epsilon),\\
F_H^\kappa=&\chi(v^\epsilon\cdot\nabla\Pi)(2SH^\epsilon n+2\zeta H^\epsilon)+\chi\Pi(2SH^\epsilon (v^\epsilon\cdot\nabla)n)-\epsilon\chi(\Delta\Pi)(2SH^\epsilon n\\
&+2\zeta H^\epsilon)-2\epsilon\chi\nabla\Pi\cdot\nabla(2SH^\epsilon n+2\zeta H^\epsilon)-\epsilon\chi\Pi(2SH^\epsilon\Delta n+2\nabla SH^\epsilon\cdot\nabla n)\\
&-\chi(H\cdot\nabla\Pi)(2Sv^\epsilon n+2\zeta v^\epsilon)-\chi\Pi(2Sv^\epsilon(H\cdot\nabla)n).
\end{align*}

Let us start with the case of $m=1$. By using the standard $L^2$ energy estimate, we get
\begin{align}\label{3.52}
\frac{1}{2}&\frac{d}{dt}(\|\eta^\epsilon_v\|^2+\|\eta^\epsilon_H\|^2)+\epsilon(\|\nabla\eta^\epsilon_v\|^2+\|\nabla\eta^\epsilon_H\|^2)\nonumber\\
&=(F_v^b+F_v^\chi+F_v^\kappa,\eta^\epsilon_v)+(F_H^b+F_H^\chi+F_H^\kappa,\eta^\epsilon_H)-2(\chi\Pi(\nabla^2P^\epsilon n),\eta^\epsilon_v).
\end{align}
Now we estimate the right-hand side terms of \eqref{3.52}. We easily arrive at
\begin{align}
\|F_v^b\|_{m-1}+\|F_H^b\|_{m-1}\leq \,&C\,(\|v^\epsilon\|_{W^{1,\infty}}+\|H^\epsilon\|_{W^{1,\infty}})(\|\nabla v^\epsilon\|_{m-1}+\|\nabla H^\epsilon\|_{m-1})\nonumber\\
&+C\,\|\nabla P^\epsilon\|_{m-1},\label{3.53}\\
\|F_v^\kappa\|_{m-1}+\|F_H^\kappa\|_{m-1}\leq \,&C\,\epsilon(\|\chi\nabla^2v^\epsilon\|_{m-1}+\|\chi\nabla^2H^\epsilon\|_{m-1}+\|\nabla H^\epsilon\|_{m-1}\nonumber\\
&+\|\nabla v^\epsilon\|_{m-1}+\| v^\epsilon\|_{m}+\|H^\epsilon\|_{m})\nonumber\\
&+C\,(\|v^\epsilon\|_{W^{1,\infty}}+\|H^\epsilon\|_{W^{1,\infty}})(\| v^\epsilon\|_{m-1}+\|H^\epsilon\|_{m-1}\nonumber\\
&+\|\nabla v^\epsilon\|_{m-1}+\|\nabla H^\epsilon\|_{m-1}).\label{3.54}
\end{align}

Next, since $F_v^\chi$ and $F_H^\chi$ are supported away from the boundary, we can control any derivatives by the norm $\|\cdot\|_m$. We immediately get
\begin{align}\label{3.55}
\|F_v^\chi\|_{m-1}+\|F_H^\chi\|_{m-1}\leq&\,C\,\epsilon(\|\nabla v^\epsilon\|_m+\|\nabla H^\epsilon\|_m)\nonumber\\
&+C\,(1+\|v^\epsilon\|_{W^{1,\infty}}+\|H^\epsilon\|_{W^{1,\infty}})(\|v^\epsilon\|_m+\|H^\epsilon\|_m).
\end{align}

Finally, we estimate $(\chi\Pi(\nabla^2P^\epsilon n),\eta^\epsilon_v)$. Noting that $P^\epsilon=P^\epsilon_1+P^\epsilon_2$, we get
\begin{equation}\label{3.56}
|(\chi\Pi(\nabla^2P^\epsilon n),\eta^\epsilon_v)|\leq \|\nabla^2P^\epsilon_1\|\|\eta^\epsilon_v\|+\Big{|}\int_{\Omega}\chi\Pi(\nabla^2P^\epsilon_2 n)\cdot\eta^\epsilon_v\Big{|}.
\end{equation}
Since $\eta^\epsilon_v=0$ on the boundary, we can integrate by the parts the last term in \eqref{3.56} to obtain
\begin{equation}\label{3.57}
\Big{|}\int_{\Omega}\chi\Pi(\nabla^2P^\epsilon_2 n)\cdot\eta^\epsilon_v\Big{|}\leq C\|\nabla P^\epsilon_2\|(\|\nabla\eta^\epsilon_v\|+\|\eta^\epsilon_v\|).
\end{equation}

Consequently, from \eqref{3.53}-\eqref{3.55}, \eqref{3.56}, \eqref{3.57}, we have
\begin{align}\label{3.58}
\frac{1}{2}\frac{d}{dt}&(\|\eta^\epsilon_v\|^2+\|\eta^\epsilon_H\|^2)+\epsilon(\|\nabla\eta^\epsilon_v\|^2+\|\nabla\eta^\epsilon_H\|^2)\nonumber\\
\leq&\, C\,\epsilon(\|\chi\nabla^2v^\epsilon\|+\|\chi\nabla^2H^\epsilon\|+\|\nabla v^\epsilon\|_1+\|\nabla H^\epsilon\|_1)(\|\eta^\epsilon_v\|+\|\eta^\epsilon_H\|)\nonumber\\
&+C\,(1+\|v^\epsilon\|_{W^{1,\infty}}+\|H^\epsilon\|_{W^{1,\infty}})(\| v^\epsilon\|^2_1+\|H^\epsilon\|^2_1+\|\nabla v^\epsilon\|^2+\|\nabla H^\epsilon\|^2\nonumber\\
&+\|\eta^\epsilon_v\|^2+\|\eta^\epsilon_H\|^2)+\|\nabla P^\epsilon_2\|(\|\nabla\eta^\epsilon_v\|+\|\eta^\epsilon_v\|)+\|\nabla^2P^\epsilon_1\|\|\eta^\epsilon_v\|+\|\nabla P^\epsilon\|\|\eta^\epsilon_v\|.
\end{align}

Due to \eqref{3.37} and \eqref{3.41}, we get that
\begin{align}\label{3.59}
\epsilon\|\chi\nabla^2v^\epsilon\|_{m-1}&\leq C\,\epsilon(\|\nabla\eta^\epsilon_v\|_{m-1}+\|\nabla v^\epsilon\|_m+\| v^\epsilon\|_m).
\end{align}

Similarly, we get
\begin{equation}\label{3.60}
\epsilon\|\chi\nabla^2H^\epsilon\|_{m-1}\leq C\,\epsilon(\|\nabla\eta^\epsilon_H\|_{m-1}+\|\nabla H^\epsilon\|_m+\| H^\epsilon\|_m).
\end{equation}

By using \eqref{3.59}, \eqref{3.60} and Young's inequality, we have
\begin{align}\label{3.61}
\frac{1}{2}\frac{d}{dt}&(\|\eta^\epsilon_v\|^2+\|\eta^\epsilon_H\|^2)+\epsilon(\|\nabla\eta^\epsilon_v\|^2+\|\nabla\eta^\epsilon_H\|^2)\nonumber\\
\leq&\, C\big{\{}\epsilon(\|\nabla v^\epsilon\|_1+\|\nabla H^\epsilon\|_1)(\|\eta^\epsilon_v\|+\|\eta^\epsilon_H\|)\nonumber\\
&+(1+\|v^\epsilon\|_{W^{1,\infty}}+\|H^\epsilon\|_{W^{1,\infty}})(\| v^\epsilon\|^2_1+\|H^\epsilon\|^2_1+\|\nabla v^\epsilon\|^2+\|\nabla H^\epsilon\|^2\nonumber\\
&+\|\eta^\epsilon_v\|^2+\|\eta^\epsilon_H\|^2)+\epsilon^{-1}\|\nabla P^\epsilon_2\|^2+\|\eta^\epsilon_v\|(\|\nabla P^\epsilon\|+\|\nabla^2P^\epsilon_1\|)\big{\}}.
\end{align}
Since $\epsilon(\|\nabla v^\epsilon\|_1+\|\nabla H^\epsilon\|_1)$ has been estimated in Lemma \ref{L3.2}, this yields \eqref{3.3} for the case of $m=1$.

Now we assume that Lemma \ref{L3.3} is true for $|\alpha|\leq m-2$ and let us consider the situation of $|\alpha|= m-1$. By applying $Z^\alpha$ to \eqref{3.48}-\eqref{3.49}, we have
\begin{align}
\partial_tZ^\alpha\eta^\epsilon_v-&\epsilon Z^\alpha\Delta\eta^\epsilon_v+v^\epsilon\cdot\nabla Z^\alpha\eta^\epsilon_v-H^\epsilon\cdot\nabla Z^\alpha\eta^\epsilon_H\nonumber\\
&=Z^\alpha F_v^b+Z^\alpha F_v^\chi+Z^\alpha F_v^\kappa-Z^\alpha(\chi\Pi(\nabla^2P^\epsilon n))+\mathcal{C}_3,\label{3.62}\\
\partial_tZ^\alpha\eta^\epsilon_H-&\epsilon Z^\alpha\Delta\eta^\epsilon_H+v^\epsilon\cdot\nabla Z^\alpha\eta^\epsilon_H-H^\epsilon\cdot\nabla Z^\alpha\eta^\epsilon_v\nonumber\\
&=Z^\alpha F_H^b+Z^\alpha F_H^\chi+Z^\alpha F_H^\kappa+\mathcal{C}_4,\label{3.63}
\end{align}
where
\begin{equation}\nonumber
\begin{split}
&\mathcal{C}_3:=-[Z^\alpha,v^\epsilon\cdot\nabla]\eta^\epsilon_v+[Z^\alpha,H^\epsilon\cdot\nabla]\eta^\epsilon_H,\\
&\mathcal{C}_4:=-[Z^\alpha,v^\epsilon\cdot\nabla]\eta^\epsilon_H+[Z^\alpha,H^\epsilon\cdot\nabla]\eta^\epsilon_v.
\end{split}
\end{equation}
From the standard energy estimate, we get
\begin{align}\label{3.65}
\frac{1}{2}\frac{d}{dt}&(\|Z^\alpha\eta^\epsilon_v\|^2+\|Z^\alpha\eta^\epsilon_H\|^2)\nonumber\\
\leq&\,\epsilon \,(Z^\alpha\Delta\eta^\epsilon_v,Z^\alpha\eta^\epsilon_v)
+\epsilon (Z^\alpha\Delta\eta^\epsilon_H,Z^\alpha\eta^\epsilon_H)\nonumber\\
&+(C_1,Z^\alpha\eta^\epsilon_v)+(C_2,Z^\alpha\eta^\epsilon_H)-2(Z^\alpha(\chi\Pi(\nabla^2P^\epsilon n)),Z^\alpha\eta^\epsilon_v)\nonumber\\
&+(Z^\alpha F_v^b+Z^\alpha F_v^\chi+Z^\alpha F_v^\kappa,Z^\alpha\eta^\epsilon_v)+(Z^\alpha F_H^b+Z^\alpha F_H^\chi+Z^\alpha F_H^\kappa,Z^\alpha\eta^\epsilon_H).
\end{align}

First, let us estimate $\epsilon (Z^\alpha\Delta\eta^\epsilon_v,Z^\alpha\eta^\epsilon_v)$ and $\epsilon (Z^\alpha\Delta\eta^\epsilon_H,Z^\alpha\eta^\epsilon_H)$. We observe that
\begin{align}\label{3.66}
\int_\Omega Z^\alpha\partial_{ii}\eta^\epsilon_v\cdot Z^\alpha\eta^\epsilon_v=
& -\int_\Omega|\partial_iZ^\alpha\eta^\epsilon_v|^2-\int_\Omega[Z^\alpha,\partial_i]\eta^\epsilon_v\cdot\partial_iZ^\alpha\eta^\epsilon_v\nonumber\\
&+\int_\Omega[Z^\alpha,\partial_i]
\partial_i\eta^\epsilon_v\cdot Z^\alpha\eta^\epsilon_v,
\end{align}
where $i=1, 2, 3$. To estimate the last two terms on the right hand side of \eqref{3.66}, we use the structure of the commutator $[Z^\alpha,\partial_i]$ and the expansion
$\partial_i=\beta^1\partial_{y^1}+\beta^2\partial_{y^2}+\beta^3\partial_{y^3}$ in the local basis. We have the following expansion
\begin{align}
[Z^\alpha,\partial_i]\eta^\epsilon_v=\sum_{\gamma,|\gamma|\leq|\alpha|-1}c_\gamma\partial_zZ^\gamma\eta^\epsilon_v+\sum_{\beta,|\beta|\leq|\alpha|}c_\beta Z^\beta\eta^\epsilon_v.\nonumber
\end{align}
This yields the estimates
\begin{align}
&\Big{|}\int_\Omega[Z^\alpha,\partial_i]\eta^\epsilon_v\cdot\partial_iZ^\alpha\eta^\epsilon_v\Big{|}\leq C\,\|\nabla Z^{m-1}\eta^\epsilon_v\|(\|\nabla\eta^\epsilon_v\|_{m-2}+\|\eta^\epsilon_v\|_{m-1}),\label{3.68}\\
&\Big{|}\int_\Omega[Z^\alpha,\partial_i]\partial_i\eta^\epsilon_v\cdot Z^\alpha\eta^\epsilon_v\Big{|}\leq C\,\|\nabla\eta^\epsilon_v\|_{m-1}(\|\nabla\eta^\epsilon_v\|_{m-2}+\|\eta^\epsilon_v\|_{m-1}).\label{3.69}
\end{align}
Taking the same argument as above, we have
\begin{align}
&\Big{|}\int_\Omega[Z^\alpha,\partial_i]\eta^\epsilon_H\cdot\partial_iZ^\alpha\eta^\epsilon_H\Big{|}\leq C\,\|\nabla Z^{m-1}\eta^\epsilon_H\|(\|\nabla\eta^\epsilon_H\|_{m-2}+\|\eta^\epsilon_H\|_{m-1}),\label{3.70}\\
&\Big{|}\int_\Omega[Z^\alpha,\partial_i]\partial_i\eta^\epsilon_H\cdot Z^\alpha\eta^\epsilon_H\Big{|}\leq C\,\|\nabla\eta^\epsilon_H\|_{m-1}(\|\nabla\eta^\epsilon_H\|_{m-2}+\|\eta^\epsilon_H\|_{m-1}).\label{3.71}
\end{align}
Consequently, we get from \eqref{3.65}, \eqref{3.68}-\eqref{3.71} and Young's inequality that
\begin{align}\label{3.72}
\frac{1}{2}\frac{d}{dt}&(\|Z^\alpha\eta^\epsilon_v\|^2+\|Z^\alpha\eta^\epsilon_H\|^2)+\frac{\epsilon}{2}(\|\nabla Z^{m-1}\eta^\epsilon_v\|^2+\|\nabla Z^{m-1}\eta^\epsilon_H\|^2)\nonumber\\
\leq&\,C\epsilon\,(\|\nabla\eta^\epsilon_H\|_{m-1}^2+\|\eta^\epsilon_H\|_{m-1}^2+\|\nabla\eta^\epsilon_v\|_{m-1}^2+\|\eta^\epsilon_v\|_{m-1}^2)\nonumber\\
&+(\|F_v^b\|_{m-1}+\|F_v^\chi\|_{m-1}+\|F_v^\kappa\|_{m-1})\|\eta^\epsilon_v\|_{m-1}+\|\mathcal{C}_3\|\|\eta^\epsilon_v\|_{m-1}\nonumber\\
&+(\|F_H^b\|_{m-1}+\|F_H^\chi\|_{m-1}+\|F_H^\kappa\|_{m-1})\|\eta^\epsilon_H\|_{m-1}+\|\mathcal{C}_4\|\|\eta^\epsilon_H\|_{m-1}\nonumber\\
&-2(Z^\alpha(\chi\Pi(\nabla^2P^\epsilon n)),Z^\alpha\eta^\epsilon_v).
\end{align}

Second, we get from \eqref{3.53}-\eqref{3.55}, \eqref{3.59}, and \eqref{3.60} that
%\begin{align}
%\|F_v^b&\|_{m-1}+\|F_v^\chi\|_{m-1}+\|F_v^\kappa\|_{m-1}\nonumber\\
%\leq &\, C\,(1+\|v^\epsilon\|_{W^{1,\infty}}+\|H^\epsilon\|_{W^{1,\infty}})(\|v^\epsilon\|_m+\|\nabla v^\epsilon\|_{m-1}
%+\|H^\epsilon\|_m+\|\nabla H^\epsilon\|_{m-1})\nonumber\\
%&+\epsilon \,C\,\|\nabla v^\epsilon\|_m+\epsilon \,C\,\|\chi\nabla^2v^\epsilon\|_{m-1}+\|\nabla P^\epsilon\|_{m-1}.\label{3.73}\\
%\|F_H^b&\|_{m-1}+\|F_H^\chi\|_{m-1}+\|F_H^\kappa\|_{m-1}\nonumber\\
%\leq &\, C\,(1+\|v^\epsilon\|_{W^{1,\infty}}+\|H^\epsilon\|_{W^{1,\infty}})(\|v^\epsilon\|_m
%+\|\nabla v^\epsilon\|_{m-1}+\|H^\epsilon\|_m+\|\nabla H^\epsilon\|_{m-1})\nonumber\\
%&+\epsilon \,C\,\|\nabla H^\epsilon\|_m+\epsilon\, C\,\|\chi\nabla^2H^\epsilon\|_{m-1}.\label{3.74}
%\end{align}
%It remains to estimate $\epsilon\|\chi\nabla^2v^\epsilon\|_{m-1}$. we easily obtain
%\begin{equation}\label{3.75}
%\epsilon\,\|\chi\nabla^2v^\epsilon\|_{m-1}\leq\epsilon\,\|\chi\nabla \partial_nv^\epsilon\|_{m-1}+\epsilon\, C\,(\|v^\epsilon\|_m+\|\nabla v^\epsilon\|_{m}).
%\end{equation}
%Due to \eqref{3.37} and \eqref{3.41}, we get
%\begin{equation}\label{3.76}
%\epsilon\,\|\chi\nabla \partial_nv^\epsilon\|_{m-1}\leq C \,\epsilon\,(\|\nabla v\|_m+\|v\|_m+\|\nabla \eta^\epsilon_v\|_{m-1}).
%\end{equation}
%So we get from  that
\begin{align}
\|F_v^b&\|_{m-1}+\|F_v^\chi\|_{m-1}+\|F_v^\kappa\|_{m-1}\nonumber\\
\leq\,&C\,(1+\|v^\epsilon\|_{W^{1,\infty}}+\|H^\epsilon\|_{W^{1,\infty}})(\|v^\epsilon\|_m+\|\nabla v^\epsilon\|_{m-1}+\|H^\epsilon\|_m+\|\nabla H^\epsilon\|_{m-1})\nonumber\\
&+\epsilon\, C\,\|\nabla v^\epsilon\|_m+\epsilon\, C\,\|\nabla \eta^\epsilon_v\|_{m-1}+\|\nabla P^\epsilon\|_{m-1},\label{3.77}\\
\|F_H^b&\|_{m-1}+\|F_H^\chi\|_{m-1}+\|F_H^\kappa\|_{m-1}\nonumber\\
\leq \,&C\,(1+\|v^\epsilon\|_{W^{1,\infty}}+\|H^\epsilon\|_{W^{1,\infty}})(\|v^\epsilon\|_m+\|\nabla v^\epsilon\|_{m-1}+\|H^\epsilon\|_m+\|\nabla H^\epsilon\|_{m-1})\nonumber\\
&+\epsilon\, C\,\|\nabla H^\epsilon\|_m+\epsilon\, C\,\|\nabla \eta^\epsilon_H\|_{m-1}.\label{3.78}
\end{align}

Next, we estimate $\|\mathcal{C}_3\|$ and $\|\mathcal{C}_4\|$. In the local coordinates, we observe
\begin{align}
f\cdot\nabla g=f_1\partial_{y^1}g+f_2\partial_{y^2}g+f\cdot N\partial_zg.\nonumber
\end{align}
Hence
\begin{align}
&[Z^\alpha,v^\epsilon\cdot\nabla]\eta^\epsilon_v\nonumber\\
=&\sum_{i=1,2}\sum_{|\beta|\geq1,|\beta|+|\gamma|\leq|\alpha|}Z^\beta v_i^\epsilon Z^\gamma Z_i\eta^\epsilon_v+\sum_{|\beta|\geq1,|\beta|+|\gamma|\leq|\alpha|}Z^\beta (v_3^\epsilon\cdot N) Z^\gamma\partial_z\eta^\epsilon_v\nonumber\\
=&\sum_{i=1,2}\sum_{|\beta|\geq1,|\beta|+|\gamma|\leq|\alpha|}Z^\beta v_i^\epsilon Z^\gamma Z_i\eta^\epsilon_v+\sum_{|\widetilde{\beta}|\geq1,|\widetilde{\beta}|+|\widetilde{\gamma}|\leq|\alpha|}Z^{\widetilde{\beta }}( \frac{v_3^\epsilon\cdot N}{\varphi(z)}) Z^{\widetilde{\gamma}} Z_3\eta^\epsilon_v.
\end{align}
We can do similar caculations for other terms in $\mathcal{C}_3$ and $\mathcal{C}_4$.
Consequently, from \eqref{1.4.1}, \eqref{1.4.2} and Lemma \ref{L2.3}, we get
\begin{align}
\|\mathcal{C}_3\|\leq& \,C\,(\|v^\epsilon\|_{2,\infty}+\|v^\epsilon\|_{w^{1,\infty}}+\|Z\eta^\epsilon_v\|_{L^\infty})(\|\eta^\epsilon_v\|_{m-1}+\|v^\epsilon\|_m)\nonumber\\
&+C\,(\|H^\epsilon\|_{2,\infty}+\|H^\epsilon\|_{w^{1,\infty}}+\|Z\eta^\epsilon_H\|_{L^\infty})(\|\eta^\epsilon_H\|_{m-1}+\|H^\epsilon\|_m),\label{3.80}\\
\|\mathcal{C}_4\|\leq&\, C\,(\|v^\epsilon\|_{2,\infty}+\|v^\epsilon\|_{w^{1,\infty}}+\|Z\eta^\epsilon_H\|_{L^\infty})(\|\eta^\epsilon_H\|_{m-1}+\|v^\epsilon\|_m)\nonumber\\
&+C\,(\|H^\epsilon\|_{2,\infty}+\|H^\epsilon\|_{w^{1,\infty}}+\|Z\eta^\epsilon_v\|_{L^\infty})(\|\eta^\epsilon_v\|_{m-1}+\|H^\epsilon\|_m).\label{3.81}
\end{align}

Final, it remains to deal with the terms involving the pressure $P^\epsilon$. As above, we use the split $P^\epsilon=P^\epsilon_1+P^\epsilon_2$ and we integrate by parts the terms involving $P^\epsilon_2$. We have
\begin{align}\label{3.82}
|\big{(}Z^\alpha(\chi\Pi(\nabla^2P^\epsilon n)),Z^\alpha\eta^\epsilon_v\big{)}|\leq \,&C\,\big{(}\|\nabla^2P^\epsilon_1\|_{m-1}\|\eta^\epsilon_v\|_{m-1}\nonumber\\
&+\|\nabla P^\epsilon_2\|_{m-1}(\|\nabla Z^{m-1}\eta^\epsilon_v\|+\|\eta^\epsilon_v\|_{m-1})\big{)}.
\end{align}

By combining \eqref{3.72}, \eqref{3.77}, \eqref{3.78}, \eqref{3.80}, \eqref{3.81}, \eqref{3.82} and using the induction assumption and Young's inequality, we complete the proof of Lemma \ref{L3.3}.
\end{proof}

\subsection{Pressure Estimates}It remains to estimate the pressure terms and the $L^\infty$ norms, the aim of this subsection is to give the pressure estimates.
\begin{Lemma}\label{L3.4} For every $m\geq2$, we have the following estimates:
\begin{align}
\|\nabla P^\epsilon_1\|_{m-1}+\|\nabla^2 P^\epsilon_1\|_{m-1}\leq\, &C\,(1+\|v^\epsilon\|_{W^{1,\infty}})(\|v^\epsilon\|_m+\|\nabla v^\epsilon\|_{m-1})\nonumber\\
&+C\,(1+\|H^\epsilon\|_{W^{1,\infty}})(\|H^\epsilon\|_m+\|\nabla H^\epsilon\|_{m-1}),\label{3.83}\\
\|\nabla P^\epsilon_2\|_{m-1}\leq &\,C\,\epsilon\,(\|v^\epsilon\|_m+\|\nabla v^\epsilon\|_{m-1}).\label{3.84}
\end{align}
\end{Lemma}
\begin{proof}
Recall that $P^\epsilon=P^\epsilon_1+P^\epsilon_2$ and $P^\epsilon_1$, $P^\epsilon_2$ are defined in \eqref{3.8} and \eqref{3.9}, respectively. %\begin{equation}\nonumber
%\left\{\begin{array}{l}
%\Delta P^\epsilon_1=-\nabla\cdot(v^\epsilon\cdot\nabla v^\epsilon-H^\epsilon\cdot\nabla H^\epsilon)\quad \text{in}\quad \Omega,\\
%\partial_nP^\epsilon_1=-(v^\epsilon\cdot\nabla v^\epsilon-H^\epsilon\cdot\nabla H^\epsilon)\cdot n\quad \text{on}\quad \partial\Omega
%\end{array}
%\right.
%\end{equation}
%and
%\begin{equation}\nonumber
%\left\{\begin{array}{l}
%\Delta P^\epsilon_2=0\quad \text{in}\quad \Omega,\\
%\partial_nP^\epsilon_2=\epsilon\Delta v^\epsilon\cdot n,\quad \text{on}\quad\partial\Omega.
%\end{array}
%\right.
%\end{equation}
From the standard elliptic regularity results with Neumann boundary conditions, we obtain that
\begin{align*}
&\|\nabla  P^\epsilon_1\|_{m-1}+\|\nabla^2  P^\epsilon_1\|_{m-1}\nonumber\\
\leq\, &C\,\big{(}\|\nabla v\cdot\nabla v-\nabla H\cdot\nabla H\|_{m-1}+\|v^\epsilon\cdot\nabla v^\epsilon-H^\epsilon\cdot\nabla H^\epsilon\|\nonumber\\
&+|(v^\epsilon\cdot\nabla v^\epsilon-H^\epsilon\cdot\nabla H^\epsilon)\cdot n|_{H^{m-\frac{1}{2}}(\partial\Omega)}\big{)}.
\end{align*}
Due to $v^\epsilon\cdot n=0$, $H^\epsilon\cdot n=0$ and Lemma \ref{L2.2}, we get that
\begin{align*}
|(v^\epsilon\cdot\nabla v^\epsilon-H^\epsilon\cdot\nabla H^\epsilon)\cdot n|_{H^{m-\frac{1}{2}}(\partial\Omega)})\leq \,&C\,(\|\nabla (v\otimes v)\|_{m-1}+\|v\otimes v\|_{m}\nonumber\\
&+\|\nabla (H\otimes H)\|_{m-1}+\|H\otimes H\|_{m}).
\end{align*}
Using Lemma \ref{L2.3}, we get \eqref{3.83}.

It remains to estimate $P^\epsilon_2$. By using the standard elliptic regularity results with Neumann boundary conditions again, we obtain
\begin{align}
\|\nabla P^\epsilon_2\|_{m-1}\leq C\,\epsilon\, |\Delta v^\epsilon\cdot n|_{H^{m-\frac{3}{2}}(\partial\Omega)}.\nonumber
\end{align}
Since
\begin{align}
\Delta v^\epsilon\cdot n=2\Big{(}\nabla\cdot(Sv^\epsilon n)-\sum_j(Sv^\epsilon\partial_jn)_j\Big{)},\nonumber
\end{align}
we can get
\begin{align}
|\Delta v^\epsilon\cdot n|_{H^{m-\frac{3}{2}}(\partial\Omega)}\leq C\,|\nabla\cdot(Sv^\epsilon n)|_{H^{m-\frac{3}{2}}(\partial\Omega)}+C\,|\nabla v^\epsilon|_{H^{m-\frac{3}{2}}(\partial\Omega)}.\nonumber
\end{align}
Due to \eqref{3.2} and \eqref{3.17}, we can further arrive at
\begin{align}
|\Delta v^\epsilon\cdot n|_{H^{m-\frac{3}{2}}(\partial\Omega)}\leq C\,|\nabla\cdot(Sv^\epsilon n)|_{H^{m-\frac{3}{2}}(\partial\Omega)}+C\,|v^\epsilon|_{H^{m-\frac{1}{2}}(\partial\Omega)}.\nonumber
\end{align}

Let us estimate $|\nabla\cdot(Sv^\epsilon n)|_{H^{m-\frac{3}{2}}(\partial\Omega)}$. We can use \eqref{3.17} to obtain
\begin{align*}
|\nabla\cdot(Sv^\epsilon n)|_{H^{m-\frac{3}{2}}(\partial\Omega)}\leq \,&C\,|\partial_n(Sv^\epsilon n)\cdot n|_{H^{m-\frac{3}{2}}(\partial\Omega)}\nonumber\\
&+C\,(|\Pi(Sv^\epsilon n)|_{H^{m-\frac{1}{2}}(\partial\Omega)}+|\nabla v^\epsilon|_{H^{m-\frac{3}{2}}(\partial\Omega)}).
\end{align*}
Also, due to \eqref{3.2}, \eqref{3.17} and the Navier boundary conditions, we get
\begin{equation}\label{3.94}
|\nabla\cdot(Sv^\epsilon n)|_{H^{m-\frac{3}{2}}(\partial\Omega)}\leq \,C\,|\partial_n(Sv^\epsilon n)\cdot n|_{H^{m-\frac{3}{2}}(\partial\Omega)}+|v^\epsilon|_{H^{m-\frac{1}{2}}(\partial\Omega)}.
\end{equation}
The first term of the right-hand side of \eqref{3.94} have the following estimates
\begin{align*}
|\partial_n(Sv^\epsilon n)\cdot n|_{H^{m-\frac{3}{2}}(\partial\Omega)}&\leq\, C\,|\partial_n(\partial_nv^\epsilon\cdot n)|_{H^{m-\frac{3}{2}}(\partial\Omega)}+C\,|\nabla v^\epsilon|_{H^{m-\frac{3}{2}}(\partial\Omega)}\nonumber\\
&\leq\, C\,|\partial_n(\partial_nv^\epsilon\cdot n)|_{H^{m-\frac{3}{2}}(\partial\Omega)}+C\,|v^\epsilon|_{H^{m-\frac{1}{2}}(\partial\Omega)}.
\end{align*}
By taking the normal derivative of \eqref{3.17} and using \eqref{3.2}, we obtain
\begin{align*}
|\partial_n(\partial_nv^\epsilon\cdot n)|_{H^{m-\frac{3}{2}}(\partial\Omega)}&\leq C\,|\Pi\partial_nv^\epsilon|_{H^{m-\frac{1}{2}}(\partial\Omega)}+C \,|\nabla v^\epsilon|_{H^{m-\frac{3}{2}}(\partial\Omega)}\nonumber\\
&\leq C\,|v^\epsilon|_{H^{m-\frac{1}{2}}(\partial\Omega)}.
\end{align*}
Consequently, we have
\begin{align}
|\Delta v^\epsilon\cdot n|_{H^{m-\frac{3}{2}}(\partial\Omega)}\leq C\,|v^\epsilon|_{H^{m-\frac{1}{2}}(\partial\Omega)}.\nonumber
\end{align}
By Lemma \ref{L2.2}, we finally get \eqref{3.84} which complete the proof of Lemma \ref{L3.4}.
\end{proof}

%\subsection{$L^\infty$ Estimates.}
We can get from Lemmas \ref{L3.2}-\ref{L3.4} that
\begin{align}\label{3.98}
&\|v^\epsilon\|_m^2+\|H^\epsilon\|_m^2+\|\nabla v^\epsilon\|^2_{m-1}+\|\nabla H^\epsilon\|^2_{m-1}+\epsilon\int^t_0(\|\nabla^2v^\epsilon\|_{m-1}+\|\nabla^2H^\epsilon\|_{m-1})\nonumber\\
\leq\, &C\,(\|v^\epsilon(0)\|_m^2+\|H^\epsilon(0)\|_m^2+\|\nabla v^\epsilon(0)\|^2_{m-1}+\|\nabla H^\epsilon(0)\|^2_{m-1})\nonumber\\
&+C\,\int^t_0(1+\|v^\epsilon\|_{2,\infty}+\|\nabla v^\epsilon\|_{1,\infty}+\|H^\epsilon\|_{2,\infty}+\|\nabla H^\epsilon\|_{1,\infty})\nonumber\\
&\times(\|v^\epsilon\|_m^2+\|H^\epsilon\|_m^2+\|\nabla v^\epsilon\|^2_{m-1}+\|\nabla H^\epsilon\|^2_{m-1}).
\end{align}
\subsection{$L^\infty$ estimates}
In order to close the estimates in \eqref{3.98}, we need to give the $L^\infty$ estimates on $\nabla v^\epsilon$ and $\nabla H^\epsilon$. We have
%\begin{equation}\label{3.99}
%N_m(t)=\|v^\epsilon\|_m^2+\|\nabla v^\epsilon\|_{m-1}^2+\|\nabla v^\epsilon\|^2_{1,\infty}+\|H^\epsilon\|_m^2+\|\nabla H^\epsilon\|_{m-1}^2+\|\nabla H^\epsilon\|^2_{1,\infty}.
%\end{equation}
\begin{Lemma}\label{L3.5} For $m_0>1$, we have the following estimates:
\begin{align}
\|&v^\epsilon\|_{2,\infty}\leq C(\|v^\epsilon\|_m+\|\nabla v^\epsilon\|_{m-1})\leq C N_m(t)^{\frac{1}{2}}\quad m\geq m_0+3,\label{3.100}\\
\|&H^\epsilon\|_{2,\infty}\leq C(\|H^\epsilon\|_m+\|\nabla H^\epsilon\|_{m-1})\leq C N_m(t)^{\frac{1}{2}}\quad m\geq m_0+3,\label{3.1001}\\
\|&v^\epsilon\|_{W^{1,\infty}}\leq C(\|v^\epsilon\|_m+\|\nabla v^\epsilon\|_{m-1}+\|\partial_zv^\epsilon\|_{L^\infty})\leq C N_m(t)^{\frac{1}{2}}\quad m\geq m_0+2,\label{3.1000}\\
\|&H^\epsilon\|_{W^{1,\infty}}\leq C(\|H^\epsilon\|_m+\|\nabla H^\epsilon\|_{m-1}+\|\partial_zH^\epsilon\|_{L^\infty})\leq C N_m(t)^{\frac{1}{2}}\quad m\geq m_0+2,\label{3.10011}
\end{align}
where  $N_m(t)$ is defined in \eqref{N_m}.
\end{Lemma}
\begin{proof}
By using lemma \ref{L2.2}, we can obtain \eqref{3.100}-\eqref{3.1001}, and \eqref{3.1000}-\eqref{3.10011} are obvious.
\end{proof}

\begin{Lemma}\label{L3.5.1} For $m>6$, we have the following estimate:
\begin{align*}
\|\nabla v&^\epsilon\|_{1,\infty}^2+\|\nabla H^\epsilon\|_{1,\infty}^2\leq C\Big{(}N_m(0)+(1+t+\epsilon^3t^2)\int^t_0(N_m(s)+N_m(s)^2)ds\Big{)}.
\end{align*}
\end{Lemma}
\begin{proof}
We observe that, away from the boundary, the following estimates hold:
\begin{equation}
\|\beta_i\nabla v^\epsilon\|_{1,\infty}+\|\beta_i \nabla H^\epsilon\|_{1,\infty}\leq C\,(\|v^\epsilon\|_m+\|H^\epsilon\|_m),\quad m\geq 4,\nonumber
\end{equation}
where $\{\beta_i\}$ is a partition of unity subordinated to the covering \eqref{c}. In order to estimate the near boundary parts, we adopt the ideas in the Proposition $21$ of \cite{MR}. Here, we use a local parametrization in the vicinity of the boundary given by a normal geodesic system:
$$ \Psi^n(y,z)=
\left(
\begin{array}{c}
y\\
\psi(y)\\
\end{array}
\right)
-zn(y),$$
where
$$n(y)=\frac{1}{\sqrt{1+|\nabla\psi(y)|^2}}
\left(
\begin{array}{c}
\partial_1\psi(y)\\
 \partial_2\psi(y)\\
 -1
\end{array}
\right).$$
Now, we can extend $n$ and $\Pi$ in the interior by setting
$$n(\Psi^n(y,z))=n(y),\quad\Pi(\Psi^n(y,z))=\Pi(y).$$
We observe
$\partial_z=\partial_n$ and
$$\left(
\begin{array}{c}
\partial_{y^i}
\end{array}
\right)\Big{|}_{\Psi^n(y,z)}\cdot
\left(
\begin{array}{c}
\partial_z
\end{array}
\right)\Big{|}_{\Psi^n(y,z)}=0.\nonumber
$$
Hence, the Riemann metric $g$ has the following form
\begin{equation}\nonumber
g(y,z)=
\left(
\begin{matrix}
\widetilde{g}(y,z)&0\\
0&1\\
\end{matrix}
\right).
\end{equation}
Consequently, the Laplacian in this coordinate system reads:
\begin{equation}\nonumber
\Delta f=\partial_{zz}f+\frac{1}{2}\partial_z(\ln|g|)\partial_zf+\Delta_{\widetilde{g}}f,
\end{equation}
where $|g|$ is the determinant of the matrix $g$ and $\Delta_{\widetilde{g}}$ is defined by
\begin{align}\label{g}
\Delta_{\widetilde{g}}f=\frac{1}{|\widetilde{g}|^{\frac{1}{2}}}\sum_{1\leq i,j\leq2}\partial_{y^i}(\widetilde{g}^{ij}|\widetilde{g}|^{\frac{1}{2}}\partial_{y^j}f).
\end{align}
Here, $\{\widetilde{g}^{ij}\}$ is the inverse matrix to $g$ and \eqref{g} only involves tangential derivatives.

With these preparation, we now turn to estimate the near boundary parts. Due to \eqref{3.17}, \eqref{3.100} and \eqref{3.1001} , we have
\begin{align}
&\|\chi\nabla v^\epsilon\|_{1,\infty}\leq \,C\,(\|\chi \Pi\partial_nv^\epsilon\|_{1,\infty}+\|v^\epsilon\|_m+\|\nabla v^\epsilon\|_{m-1}),\label{3.106}\\
&\|\chi\nabla H^\epsilon\|_{1,\infty}\leq \,C\,(\|\chi \Pi\partial_nH^\epsilon\|_{1,\infty}+\|H^\epsilon\|_m+\|\nabla H^\epsilon\|_{m-1}).\label{3.107}
\end{align}
Hence, we need to estimate $\|\chi \Pi\partial_nv^\epsilon\|_{1,\infty}$ and $\|\chi \Pi\partial_nH^\epsilon\|_{1,\infty}$. To this end, we first introduce the vorticity
$$\omega_v^\epsilon=\nabla\times v^\epsilon,\quad \omega_H^\epsilon=\nabla\times H^\epsilon.$$
We find that
\begin{align}\label{3.108}
\Pi(\omega_v^\epsilon\times n)&=\Pi(\nabla v^\epsilon-(\nabla v^\epsilon)^t)n\nonumber\\
&=\Pi(\partial_nv^\epsilon-\nabla(v^\epsilon\cdot n)+v^\epsilon\cdot\nabla n+v^\epsilon\times(\nabla\times n)).
\end{align}
Consequently, we have
\begin{equation}\label{3.109}
\|\chi \Pi\partial_nv^\epsilon\|_{1,\infty}\leq C\,(\|\chi\Pi(\omega_v^\epsilon\times n)\|_{1,\infty}+\|v^\epsilon\|_{2,\infty}).
\end{equation}
By using \eqref{3.100} again, we get
\begin{equation}\label{3.110}
\|\chi\nabla v^\epsilon\|_{1,\infty}\leq\, C\,(\|\chi\Pi(\omega_v^\epsilon\times n)\|_{1,\infty}+\|v^\epsilon\|_m+\|\nabla v^\epsilon\|_{m-1}).
\end{equation}

Similar to $v^\epsilon$, we have the following estimates for $H^\epsilon$,
\begin{equation}\label{3.112}
\|\chi\nabla H^\epsilon\|_{1,\infty}\leq \,C\,(\|\chi\Pi(\omega_H^\epsilon\times n)\|_{1,\infty}+\|H^\epsilon\|_m+\|\nabla H^\epsilon\|_{m-1}).
\end{equation}

Below we estimate $\|\chi\Pi(\omega_v^\epsilon\times n)\|_{1,\infty}$ and $\|\chi\Pi(\omega_H^\epsilon\times n)\|_{1,\infty}$. We know that $\omega_v^\epsilon$ and $\omega_H^\epsilon$ satisfy
\begin{align*}
&\partial_t\omega_v^\epsilon-\epsilon\Delta\omega_v^\epsilon+v^\epsilon\cdot\nabla\omega_v^\epsilon-H^\epsilon\cdot\nabla\omega_H^\epsilon+\omega_H^\epsilon\cdot\nabla H^\epsilon-\omega_v^\epsilon\cdot\nabla v^\epsilon=0,\\
&\partial_t\omega_H^\epsilon-\epsilon\Delta\omega_H^\epsilon+v^\epsilon\cdot\nabla\omega_H^\epsilon-H^\epsilon\cdot\nabla\omega_v^\epsilon+[\nabla\times,v^\epsilon\cdot\nabla] H^\epsilon-[\nabla\times,H^\epsilon\cdot\nabla] v^\epsilon=0.
\end{align*}
By setting
\begin{align}
&\widetilde{\omega}_v^\epsilon(y,z):=\omega_v^\epsilon(\Psi^n(y,z)),\quad\,\,\,\,\widetilde{v}^\epsilon(y,z):=v^\epsilon(\Psi^n(y,z)),\nonumber\\
&\widetilde{\omega}_H^\epsilon(y,z):=\omega_H^\epsilon(\Psi^n(y,z)),\quad\widetilde{H}^\epsilon(y,z):=H^\epsilon(\Psi^n(y,z)),\nonumber
\end{align}
we have
\begin{align*}
\partial_t\widetilde{\omega}_v^\epsilon+&(\widetilde{v}^\epsilon)^1\partial_{y^1}\widetilde{\omega}_v^\epsilon+(\widetilde{v}^\epsilon)^2\partial_{y^2}\widetilde{\omega}_v^\epsilon
+\widetilde{v}^\epsilon\cdot n\partial_z
\widetilde{\omega}_v^\epsilon-(\widetilde{H}^\epsilon)^1\partial_{y^1}\widetilde{\omega}_H^\epsilon-(\widetilde{H}^\epsilon)^2\partial_{y^2}\widetilde{\omega}_H^\epsilon\\
&-\widetilde{H}^\epsilon\cdot n\partial_z
\widetilde{\omega}_H^\epsilon
=\epsilon(\partial_{zz}\widetilde{\omega}_v^\epsilon+\frac{1}{2}\partial_z(\ln|g|)\partial_z\widetilde{\omega}_v^\epsilon+\Delta_{\widetilde{g}}\widetilde{\omega}_v^\epsilon)+\overline{F}^v,\\
\partial_t\widetilde{\omega}_H^\epsilon+&(\widetilde{v}^\epsilon)^1\partial_{y^1}\widetilde{\omega}_H^\epsilon+(\widetilde{v}^\epsilon)^2\partial_{y^2}\widetilde{\omega}_H^\epsilon
+\widetilde{v}^\epsilon\cdot n\partial_z
\widetilde{\omega}_H^\epsilon-(\widetilde{H}^\epsilon)^1\partial_{y^1}\widetilde{\omega}_v^\epsilon-(\widetilde{H}^\epsilon)^2\partial_{y^2}\widetilde{\omega}_v^\epsilon\\
&-\widetilde{H}^\epsilon\cdot n\partial_z\widetilde{\omega}_v^\epsilon
=\epsilon(\partial_{zz}\widetilde{\omega}_H^\epsilon+\frac{1}{2}\partial_z(\ln|g|)\partial_z\widetilde{\omega}_H^\epsilon+\Delta_{\widetilde{g}}\widetilde{\omega}_H^\epsilon)+\overline{F}^H,\\
\partial_t\widetilde{v}^\epsilon+&(\widetilde{v}^\epsilon)^1\partial_{y^1}\widetilde{v}+(\widetilde{v}^\epsilon)^2\partial_{y^2}\widetilde{v}+\widetilde{v}^\epsilon\cdot n\partial_z
\widetilde{v}-(\widetilde{H}^\epsilon)^1\partial_{y^1}\widetilde{H}-(\widetilde{H}^\epsilon)^2\partial_{y^2}\widetilde{H}-\widetilde{H}^\epsilon\cdot n\partial_z
\widetilde{H}\\
&=\epsilon(\partial_{zz}\widetilde{v}^\epsilon+\frac{1}{2}\partial_z(\ln|g|)\partial_z\widetilde{v}^\epsilon+\Delta_{\widetilde{g}}\widetilde{v}^\epsilon)-(\nabla P^\epsilon)\circ\Psi^n,\\
\partial_t\widetilde{H}^\epsilon+&(\widetilde{v}^\epsilon)^1\partial_{y^1}\widetilde{H}+(\widetilde{v}^\epsilon)^2\partial_{y^2}\widetilde{H}+\widetilde{v}^\epsilon\cdot n\partial_z
\widetilde{H}-(\widetilde{H}^\epsilon)^1\partial_{y^1}\widetilde{v}-(\widetilde{H}^\epsilon)^2\partial_{y^2}\widetilde{v}-\widetilde{H}^\epsilon\cdot n\partial_z
\widetilde{v}\\
&=\epsilon(\partial_{zz}\widetilde{H}^\epsilon+\frac{1}{2}\partial_z(\ln|g|)\partial_z\widetilde{H}^\epsilon+\Delta_{\widetilde{g}}\widetilde{H}^\epsilon),
\end{align*}
where
\begin{align*}
\overline{F}^v:=\omega_v^\epsilon\cdot\nabla v^\epsilon-\omega_H^\epsilon\cdot\nabla H^\epsilon,\quad\ \overline{F}^H:=[\nabla\times,H^\epsilon\cdot\nabla] v^\epsilon-[\nabla\times,v^\epsilon\cdot\nabla] H^\epsilon.
\end{align*}

By using \eqref{3.2} and \eqref{3.108} on the boundary, we have
$$\Pi(\widetilde{\omega}_v^\epsilon\times n)=2\Pi(\widetilde{v}^\epsilon\cdot\nabla n-\zeta\widetilde{v}^\epsilon),\quad\Pi(\widetilde{\omega}_H^\epsilon\times n)=2\Pi(\widetilde{H}^\epsilon\cdot\nabla n-\zeta\widetilde{H}^\epsilon),~~z=0.$$
Consequently, we introduce the following quantities
\begin{align*}
&\widetilde{\eta}_v^\epsilon(y,z):=\chi\Pi(\widetilde{\omega}_v^\epsilon\times n-2\widetilde{v}^\epsilon\cdot\nabla n+2\zeta\widetilde{v}^\epsilon),\\
&\widetilde{\eta}_H^\epsilon(y,z):=\chi\Pi(\widetilde{\omega}_H^\epsilon\times n-2\widetilde{H}^\epsilon\cdot\nabla n+2\zeta\widetilde{H}^\epsilon).
\end{align*}
Noting that $\widetilde{\eta}_v^\epsilon(y,0)=0$ and $\widetilde{\eta}_H^\epsilon(y,0)=0$, we easily get
\begin{align}
&\!\!\!\partial_t\widetilde{\eta}_v^\epsilon+(\widetilde{v}^\epsilon)^1\partial_{y^1}\widetilde{\eta}_v^\epsilon+(\widetilde{v}^\epsilon)^2\partial_{y^2}\widetilde{\eta}_v^\epsilon
+\widetilde{v}^\epsilon\cdot n\partial_z
\widetilde{\eta}_v^\epsilon-(\widetilde{H}^\epsilon)^1\partial_{y^1}\widetilde{\eta}_H^\epsilon-(\widetilde{H}^\epsilon)^2\partial_{y^2}\widetilde{\eta}_H^\epsilon\nonumber\\
&\!\!\!\ \  -\widetilde{H}^\epsilon\cdot n\partial_z
\widetilde{\eta}_H^\epsilon
=\epsilon(\partial_{zz}\widetilde{\eta}_v^\epsilon+\frac{1}{2}\partial_z(\ln|g|)\partial_z\widetilde{\eta}_v^\epsilon)+\chi\Pi\overline{F}^v\times n+\overline{F}_v^v+\overline{F}_v^\chi+\overline{F}_v^\kappa,\label{3.120}\\
&\!\!\!\!\partial_t\widetilde{\eta}_H^\epsilon+(\widetilde{v}^\epsilon)^1\partial_{y^1}\widetilde{\eta}_H^\epsilon+(\widetilde{v}^\epsilon)^2\partial_{y^2}\widetilde{\eta}_H^\epsilon
+\widetilde{v}^\epsilon\cdot n\partial_z
\widetilde{\eta}_H^\epsilon-(\widetilde{H}^\epsilon)^1\partial_{y^1}\widetilde{\eta}_v^\epsilon-(\widetilde{H}^\epsilon)^2\partial_{y^2}\widetilde{\eta}_v^\epsilon\nonumber\\
&\!\!\!\ \ -\widetilde{H}^\epsilon\cdot n\partial_z
\widetilde{\eta}_v^\epsilon
=\epsilon(\partial_{zz}\widetilde{\eta}_H^\epsilon+\frac{1}{2}\partial_z(\ln|g|)\partial_z\widetilde{\eta}_H^\epsilon)+\overline{F}_H^\chi+\overline{F}_H^\kappa+\chi\Pi \overline{F}^H\times n,\label{3.124}
\end{align}
where
\begin{align*}
\overline{F}_v^v=&\,2\chi\Pi(\nabla P^\epsilon\cdot\nabla n-\zeta\nabla P^\epsilon)\circ\Psi^n,\\
\overline{F}_v^\chi=&\,(((\widetilde{v}^\epsilon)^1\partial_{y^1}+(\widetilde{v}^\epsilon)^2\partial_{y^2}+\widetilde{v}^\epsilon\cdot n\partial_z)\chi)\Pi(\widetilde{\omega}_v^\epsilon\times n-2\widetilde{v}^\epsilon\cdot\nabla n+2\zeta\widetilde{v}^\epsilon)\\
&-(((\widetilde{H}^\epsilon)^1\partial_{y^1}+(\widetilde{H}^\epsilon)^2\partial_{y^2}+\widetilde{H}^\epsilon\cdot n\partial_z)\chi)\Pi(\widetilde{\omega}_H^\epsilon\times n-2\widetilde{H}^\epsilon\cdot\nabla n+2\zeta\widetilde{H}^\epsilon)\\
&-\epsilon(\partial_{zz}\chi+2\partial_z\chi\partial_z+\frac{1}{2}\partial_z(\ln|g|)\partial_z\chi)\Pi(\widetilde{\omega}_v^\epsilon\times n-2\widetilde{v}^\epsilon\cdot\nabla n+2\zeta\widetilde{v}^\epsilon),\\
\overline{F}_v^\kappa=&\chi(((\widetilde{v}^\epsilon)^1\partial_{y^1}+(\widetilde{v}^\epsilon)^2\partial_{y^2})\Pi)(\widetilde{\omega}_v^\epsilon\times n-2\widetilde{v}^\epsilon\cdot\nabla n+2\zeta\widetilde{v}^\epsilon)+\epsilon\chi\Pi(\Delta_{\widetilde{g}}\widetilde{\omega}_v^\epsilon\times n)\\
&-\chi(((\widetilde{H}^\epsilon)^1\partial_{y^1}+(\widetilde{H}^\epsilon)^2\partial_{y^2})\Pi)(\widetilde{\omega}_H^\epsilon\times n-2\widetilde{H}^\epsilon\cdot\nabla n+2\zeta\widetilde{H}^\epsilon)\\
&+\chi\Pi((((\widetilde{H}^\epsilon)^1\partial_{y^1}+(\widetilde{H}^\epsilon)^2\partial_{y^2})\nabla n)\widetilde{H}^\epsilon)-\chi\Pi(\widetilde{\omega}_H^\epsilon\times((\widetilde{H}^\epsilon)^1\partial_{y^1}\\
&+(\widetilde{H}^\epsilon)^2\partial_{y^2})n)-2\epsilon\chi\Pi(\Delta_{\widetilde{g}}\widetilde{v}^\epsilon\nabla n)-\chi\Pi((((\widetilde{v}^\epsilon)^1\partial_{y^1}+(\widetilde{v}^\epsilon)^2\partial_{y^2})\nabla n)\widetilde{v}^\epsilon)\\
&+\chi\Pi(\widetilde{\omega}_v^\epsilon\times((\widetilde{v}^\epsilon)^1\partial_{y^1}+(\widetilde{v}^\epsilon)^2\partial_{y^2})n)+2\zeta\epsilon\chi\Pi(\Delta_{\widetilde{g}}\widetilde{v}^\epsilon)
,\\
\overline{F}_H^\chi=&(((\widetilde{v}^\epsilon)^1\partial_{y^1}+(\widetilde{v}^\epsilon)^2\partial_{y^2}+\widetilde{v}^\epsilon\cdot n\partial_z)\chi)\Pi(\widetilde{\omega}_H^\epsilon\times n-2\widetilde{H}^\epsilon\cdot\nabla n+2\zeta\widetilde{H}^\epsilon)\\
&-(((\widetilde{H}^\epsilon)^1\partial_{y^1}+(\widetilde{H}^\epsilon)^2\partial_{y^2}+\widetilde{H}^\epsilon\cdot n\partial_z)\chi)\Pi(\widetilde{\omega}_v^\epsilon\times n-2\widetilde{v}^\epsilon\cdot\nabla n+2\zeta\widetilde{v}^\epsilon)\\
&-\epsilon(\partial_{zz}\chi+2\partial_z\chi\partial_z+\frac{1}{2}\partial_z(\ln|g|)\partial_z\chi)\Pi(\widetilde{\omega}_H^\epsilon\times n-2\widetilde{H}^\epsilon\cdot\nabla n+2\zeta\widetilde{H}^\epsilon),\\
\overline{F}_H^\kappa=&\chi(((\widetilde{v}^\epsilon)^1\partial_{y^1}+(\widetilde{v}^\epsilon)^2\partial_{y^2})\Pi)(\widetilde{\omega}_H^\epsilon\times n-\widetilde{H}^\epsilon\cdot\nabla n+\zeta\widetilde{H}^\epsilon)+2\zeta\epsilon\chi\Pi(\Delta_{\widetilde{g}}\widetilde{H}^\epsilon)\\
&-\chi(((\widetilde{H}^\epsilon)^1\partial_{y^1}+(\widetilde{H}^\epsilon)^2\partial_{y^2})\Pi)(\widetilde{\omega}_v^\epsilon\times n-\widetilde{v}^\epsilon\cdot\nabla n+\zeta\widetilde{H}^\epsilon)
+\epsilon\chi\Pi(\Delta_{\widetilde{g}}\widetilde{\omega}_H^\epsilon\times n)\\
&+\chi\Pi((((\widetilde{H}^\epsilon)^1\partial_{y^1}+(\widetilde{H}^\epsilon)^2\partial_{y^2})\nabla n)\widetilde{v}^\epsilon)-\chi\Pi(\widetilde{\omega}_v^\epsilon\times((\widetilde{H}^\epsilon)^1\partial_{y^1}+(\widetilde{H}^\epsilon)^2\partial_{y^2})n)\\
&-2\epsilon\chi\Pi(\Delta_{\widetilde{g}}\widetilde{H}^\epsilon\nabla n)-\chi\Pi((((\widetilde{v}^\epsilon)^1\partial_{y^1}+(\widetilde{v}^\epsilon)^2\partial_{y^2})\nabla n)\widetilde{H}^\epsilon)\\
&+\chi\Pi(\widetilde{\omega}_v^\epsilon\times((\widetilde{v}^\epsilon)^1\partial_{y^1}+(\widetilde{v}^\epsilon)^2\partial_{y^2})n).
\end{align*}
We know that $\Pi$ and $n$ do not dependent the normal variable. Due to $\Delta_{\widetilde{g}}$ only involving the tangential derivatives and the derivatives of $\chi$ compactly supported away from the boundary, we easily obtain that
\begin{align}
\|\overline{F}^v_v\|_{1,\infty}\leq \,&C\,(\|\Pi \nabla P^\epsilon\|_{1,\infty},\label{3.127}\\
\|\overline{F}^v_\chi\|_{1,\infty}\leq \,&C\,(\|v^\epsilon\|_{1,\infty}\|v^\epsilon\|_{2,\infty}+\|H^\epsilon\|_{1,\infty}\|H^\epsilon\|_{2,\infty}+\epsilon\|v^\epsilon\|_{3,\infty}),\label{3.1271}\\
\|\overline{F}_v^\kappa\|_{1,\infty}\leq\, &C\,(\|v^\epsilon\|_{1,\infty}\|\nabla v^\epsilon\|_{1,\infty}+\|H^\epsilon\|_{1,\infty}\|\nabla H^\epsilon\|_{1,\infty}+\|v^\epsilon\|_{1,\infty}^2+\|H^\epsilon\|_{1,\infty}^2\nonumber\\
&+\epsilon\|v^\epsilon\|_{3,\infty}+\epsilon\|\nabla v^\epsilon\|_{3,\infty}),\label{3.1272}\\
\|\overline{F}^H_\chi\|_{1,\infty}\leq\, &C\,(\|v^\epsilon\|_{1,\infty}\|H^\epsilon\|_{2,\infty}+\|H^\epsilon\|_{1,\infty}\|v^\epsilon\|_{2,\infty}+\epsilon\|H^\epsilon\|_{3,\infty}),\label{3.128}\\
\|\overline{F}_H^\kappa\|_{1,\infty}\leq\, &C\,(\|v^\epsilon\|_{1,\infty}\|\nabla H^\epsilon\|_{1,\infty}+\|H^\epsilon\|_{1,\infty}\|\nabla v^\epsilon\|_{1,\infty}+\|v^\epsilon\|_{1,\infty}^2+\|H^\epsilon\|_{1,\infty}^2\nonumber\\
&+\epsilon\|H^\epsilon\|_{3,\infty}+\epsilon\|\nabla H^\epsilon\|_{3,\infty}).\label{3.1281}
\end{align}

A crucial estimate towards the proof of Lemma \ref{L3.5.1} is the following:
\begin{Lemma}[\!\!\cite{MR}]\label{L3.6} Let $\rho$ is a smooth solution of
\begin{equation}\nonumber
\partial_t\rho+u\cdot\nabla\rho=\epsilon\partial_{zz}\rho+f,\quad z>0,\quad\rho(t,y,0)=0,
\end{equation}
where $u$ satisfies the divergence free condition and $u\cdot n$ vanishes on the boundary.
Assume that $\rho$ and $f$ are compactly supported with respect to $z$. Then, we have the estimate:
\begin{align*}
\|\rho\|_{1,\infty}\leq \,C&\|\rho(0)\|_{1,\infty}
+C\int_0^t\big{\{}(\|u\|_{2,\infty}+\|\partial_zu\|_{1,\infty})\nonumber\\
&\quad \qquad\qquad \qquad \times(\|\rho\|_{1,\infty}+\|\rho\|_{m_0+3})+\|f\|_{1,\infty}\big{\}}\quad \text{for}\quad m_0>2.
\end{align*}
\end{Lemma}
In order to use Lemma \ref{L3.6}, we shall eliminate $\partial_z(\ln|g|)\partial_z\widetilde{\eta}_v^\epsilon$ in \eqref{3.120} and
$\partial_z(\ln|g|)\partial_z\widetilde{\eta}_H^\epsilon$ in \eqref{3.124}, respectively. We set $$\widetilde{\eta}_v^\epsilon=\frac{1}{|g|^{\frac{1}{4}}}\overline{\eta}_v^\epsilon=\gamma\overline{\eta}_v^\epsilon,
\quad\widetilde{\eta}_H^\epsilon=\frac{1}{|g|^{\frac{1}{4}}}\overline{\eta}_H^\epsilon=\gamma\overline{\eta}_H^\epsilon.$$
We note that
\begin{equation}\label{3.129}
\|\widetilde{\eta}_v^\epsilon\|_{1,\infty}\sim\|\overline{\eta}_v^\epsilon\|_{1,\infty},\quad\|\widetilde{\eta}_H^\epsilon\|_{1,\infty}\sim\|\overline{\eta}_H^\epsilon\|_{1,\infty}
\end{equation}
and $\overline{\eta}_v^\epsilon$ and $\overline{\eta}_H^\epsilon$ solve the equations
\begin{align}
\partial_t\overline{\eta}_v^\epsilon+((\widetilde{v}^\epsilon)^1\partial_{y^1}&+(\widetilde{v}^\epsilon)^2\partial_{y^2}
+\widetilde{v}^\epsilon\cdot n\partial_z)
\overline{\eta}_v^\epsilon-((\widetilde{H}^\epsilon)^1\partial_{y^1}+(\widetilde{H}^\epsilon)^2\partial_{y^2}
+\widetilde{H}^\epsilon\cdot n\partial_z)\overline{\eta}_H^\epsilon\nonumber\\
-\epsilon\partial_{zz}\overline{\eta}_v^\epsilon
=&\frac{1}{\gamma}(\chi\Pi \overline{F}^v\times n+\overline{F}_v^v+\overline{F}_v^\chi+\overline{F}_v^\kappa+\epsilon\partial_{zz}\gamma\overline{\eta}_v^\epsilon+\frac{\epsilon}{2}\partial_z(\ln|g|)\partial_z\gamma\overline{\eta}_v^\epsilon\nonumber\\
&-(\widetilde{v}^\epsilon\cdot\nabla\gamma)\overline{\eta}_v^\epsilon+(\widetilde{H}^\epsilon\cdot\nabla\gamma)\overline{\eta}_H^\epsilon):=S_1,\label{3.130}\\
\partial_t\overline{\eta}_H^\epsilon+((\widetilde{v}^\epsilon)^1\partial_{y^1}&+(\widetilde{v}^\epsilon)^2\partial_{y^2}
+\widetilde{v}^\epsilon\cdot n\partial_z)
\overline{\eta}_H^\epsilon-((\widetilde{H}^\epsilon)^1\partial_{y^1}+(\widetilde{H}^\epsilon)^2\partial_{y^2}
+\widetilde{H}^\epsilon\cdot n\partial_z)
\overline{\eta}_v^\epsilon\nonumber\\
-\epsilon\partial_{zz}\overline{\eta}_H^\epsilon
=&\frac{1}{\gamma}(\chi\Pi \overline{F}^H\times n+\overline{F}_H^\chi+\overline{F}_H^\kappa+\epsilon\partial_{zz}\gamma\overline{\eta}_H^\epsilon+\frac{\epsilon}{2}\partial_z(\ln|g|)\partial_z\gamma\overline{\eta}_H^\epsilon
\nonumber\\
&-(\widetilde{v}^\epsilon\cdot\nabla\gamma)\overline{\eta}_v^\epsilon+(\widetilde{H}^\epsilon\cdot\nabla\gamma)\overline{\eta}_H^\epsilon):=S_2.\label{3.131}
\end{align}
Finally, we set $$\eta_1:=\overline{\eta}_v^\epsilon+\overline{\eta}_H^\epsilon,\quad\eta_2:=\overline{\eta}_v^\epsilon-\overline{\eta}_H^\epsilon$$ and easily find
\begin{align}
\partial_t\eta_1+((\widetilde{v}^\epsilon)^1\partial_{y^1}&+(\widetilde{v}^\epsilon)^2\partial_{y^2}
+\widetilde{v}^\epsilon\cdot n\partial_z)
\eta_1\nonumber\\
&-((\widetilde{H}^\epsilon)^1\partial_{y^1}+(\widetilde{H}^\epsilon)^2\partial_{y^2}
+\widetilde{H}^\epsilon\cdot n\partial_z)\eta_1-\epsilon\partial_{zz}\eta_1
=S_1+S_2,\label{3.132}\\
\partial_t\eta_2+((\widetilde{v}^\epsilon)^1\partial_{y^1}&+(\widetilde{v}^\epsilon)^2\partial_{y^2}
+\widetilde{v}^\epsilon\cdot n\partial_z)
\eta_2\nonumber\\
&-((\widetilde{H}^\epsilon)^1\partial_{y^1}+(\widetilde{H}^\epsilon)^2\partial_{y^2}
+\widetilde{H}^\epsilon\cdot n\partial_z)\eta_2-\epsilon\partial_{zz}\eta_1
=S_1-S_2.\label{3.133}
\end{align}
By applying Lemma \ref{L3.6} to \eqref{3.132}, we directly obtain
\begin{align*}
\|\eta_1\|_{1,\infty}\leq& \,C\,\|\eta_1(0)\|_{1,\infty}+\int_0^t\big{\{}(\|v^\epsilon\|_{2,\infty}+\|H^\epsilon\|_{2,\infty}+\|\nabla v^\epsilon\|_{1,\infty}+\|\nabla H^\epsilon\|_{1,\infty})\nonumber\\
&\times(\|\eta_1\|_{1,\infty}+\|\eta_1\|_{m_0+3})+\|S_1\|_{1,\infty}+\|S_2\|_{1,\infty}\big{\}}\quad\text{for}\quad m_0>2.
\end{align*}
From \eqref{3.100}-\eqref{3.10011} and \eqref{3.127}-\eqref{3.1281}, we get
\begin{align}\label{3.142}
\|\eta_1\|_{1,\infty}\leq& \,C\,\|\eta_1(0)\|_{1,\infty}+\int_0^t\big{\{}(\|v^\epsilon\|_{2,\infty}+\|H^\epsilon\|_{2,\infty}+\|\nabla v^\epsilon\|_{1,\infty}+\|\nabla H^\epsilon\|_{1,\infty})\nonumber\\
&\times(\|\eta_1\|_{1,\infty}+\|\eta_1\|_{m_0+3}+\|\eta_2\|_{1,\infty}+\|\eta_2\|_{m_0+3}+N_m^{\frac{1}{2}})+N_m\nonumber\\
&+N_m^{\frac{1}{2}}+\epsilon(\|\nabla v^\epsilon\|_{3,\infty}+\|\nabla H^\epsilon\|_{3,\infty})+\|\Pi \nabla P^\epsilon\|_{1,\infty}\nonumber\\
&+\|\Pi(\nabla P^\epsilon\cdot\nabla n)\|_{1,\infty}\big{\}}\quad\text{for}\quad m_0>2.
\end{align}
Due to Lemmas \ref{L2.2} and  \ref{L3.4}, we have
\begin{equation}\label{3.143}
\|\Pi \nabla P^\epsilon\|_{1,\infty}\leq C( N_m^{\frac{1}{2}}(t)+N_m(t))\quad \text{for}\quad m\geq4.
\end{equation}
Now, we deal with the terms with the coefficient $\epsilon$. From Lemma \ref{L2.2}, we get
\begin{align}\label{3.144}
&\Big{(}\epsilon\int_0^t(\|\nabla v^\epsilon\|_{3,\infty}+\|\nabla H^\epsilon)\|_{3,\infty}\Big{)}^2\nonumber\\
&\leq C\epsilon^2\Big{(}\int_0^t\| \nabla^2v^\epsilon\|_{m-1}^{\frac{1}{2}}+\|\nabla^2H^\epsilon)\|_{m-1}^{\frac{1}{2}})N_m^{\frac{1}{4}}\Big{)}^2+C\epsilon^2t\int_0^tN_m\nonumber\\
&\leq C\epsilon^2t\Big{(}\int_0^t(\|\nabla^2v^\epsilon\|_{m-1}^2+\|\nabla^2H^\epsilon\|_{m-1}^2)\Big{)}^{\frac{1}{2}}\Big{(}\int_0^tN_m\Big{)}^{\frac{1}{2}}+C\epsilon^2t\int_0^tN_m\nonumber\\
&\leq C\epsilon\int_0^t(\|\nabla^2v^\epsilon\|_{m-1}^2+\|\nabla^2H^\epsilon\|_{m-1}^2)+C(\epsilon^2t+\epsilon^3t^2)\int_0^tN_m
\end{align}
for $m\geq m_0+4$.
Consequently, we get from \eqref{3.98}, \eqref{3.100}-\eqref{3.10011} and  \eqref{3.142}-\eqref{3.144} that
\begin{align*}
\|\eta_1\|_{1,\infty}^2\leq C N(0)+C(1+t+\epsilon^3t^2)\int_0^t(N_m^2+N_m).
\end{align*}

Similarly, we also get
\begin{align*}
\|\eta_2\|_{1,\infty}^2\leq C N(0)+C(1+t+\epsilon^4t^2)\int_0^t(N_m^2+N_m).
\end{align*}
Therefore, we complete the proof of Lemma \ref{L3.5.1}.
\end{proof}

\subsection{Proof of Theorem \ref{Th4}} Based on Lemma \ref{L3.5}, Lemma \ref{L3.5.1} and \eqref{3.98}, we can easily prove Theorem \ref{Th4}. We omit the details here.

\subsection{Proof of Theorem \ref{Th1}} %and Theorem 2}
By smoothing the initial data and using the a priori estimates obtained in Theorem \ref{Th4} and the strong compactness argument, we can prove Theorem \ref{Th1} in the same spirit of \cite{MR}. Hence we omit it here.%and Theorem \ref{Th2}.
%we can refer to \cite{MR}.

\section{Proof of Theorem \ref{Th3}}\label{Sec4}
In this section, we shall establish the convergence with a rate for the solution $(v^{\epsilon},H^{\epsilon})$ to $(v,H)$. We start with the rate of convergence in $L^{2}$.
\begin{Lemma}\label{L5.1} Under the assumptions in the Theorem \ref{Th3}, we have
% Suppose that $(v_0,H_0)\in H^3(\Omega)$ satisfy the assumptions in Theorem \ref{Th1}. Let $(v,H)$ be the solution to the ideal MHD equations \eqref{1.7}-\eqref{1.10} on $[0,T_1]$ with the initial data $(v,H)|_{t=0}=(v_0,H_0)$, and $(v^{\epsilon},H^{\epsilon})$ be the solution to the MHD equations \eqref{1.1}-\eqref{1.4}. Then we have
\begin{equation}\nonumber
\|v^{\epsilon}-v \|^2+\|H^{\epsilon}-H\|^2+\epsilon\int_o^t(\|v^{\epsilon}-v\|_{H^1}^2
+\| H^{\epsilon}-H\|_{H^1}^2)\leq C\epsilon^\frac{3}{2}\quad on\quad [0,T_2],
\end{equation}
where $\epsilon$ small enough and $T_2=\min\{T_0,T_1\}$. Consequently, we have
\begin{equation}\nonumber
\|v^{\epsilon}-v\|_{L^\infty([0,T_2]\times\Omega)}+\|H^{\epsilon}-H\|_{L^\infty([0,T_2]\times\Omega)}\leq C\epsilon^{\frac{3}{10}}.
\end{equation}
\end{Lemma}
\begin{proof}
 We note that $v^{\epsilon}-v$ and $H^{\epsilon}-H$ satisfy
\begin{align}
&     \partial_t(v^{\epsilon}-v)-\epsilon\Delta(v^{\epsilon}-v)+\Phi_1+\nabla(p^{\epsilon}-p)=\epsilon\Delta v \quad \text{in}\quad \Omega,\label{5.2}\\
&     \partial_t(H^{\epsilon}-H)-\epsilon\Delta(H^{\epsilon}-H)+\Phi_2=\epsilon\Delta H \quad \text{in}\quad\Omega,\label{5.3}\\
&    \nabla\cdot v^{\epsilon}=0 ~~,~~ \nabla\cdot H^{\epsilon} =0 \quad \text{in}\quad\Omega,\label{5.4}\\
&    (v^{\epsilon}-v)\cdot n=0,\ \  n\times(\omega^{\epsilon}_v-\omega_v)=[B(v^{\epsilon}-v)+Bv]_\tau-n\times \omega_v\ \ \text{on}\ \ \partial\Omega,\label{5.5}\\
&   (H^{\epsilon}-H)\cdot n=0,\ \  n\times(\omega^{\epsilon}_H-\omega_H)=[B(H^{\epsilon}-H)+BH]_\tau-n\times \omega_H\ \ \text{on}\ \ \partial\Omega,\label{5.6}
\end{align}
where $\omega^{\epsilon}_v=\nabla\times v^{\epsilon},\,\,\,\omega^{\epsilon}_H=\nabla\times H^{\epsilon},\,\,\,\omega_v=\nabla\times v,\,\,\,\omega_H=\nabla\times H$, and
\begin{align*}
\Phi_1:=\,&v\cdot\nabla(v^{\epsilon}-v)+(v^{\epsilon}-v)\cdot\nabla v+(v^{\epsilon}-v)\cdot\nabla(v^{\epsilon}-v)\\
&-H\cdot\nabla(H^{\epsilon}-H)-(H^{\epsilon}-H)\cdot\nabla H-(H^{\epsilon}-H)\cdot\nabla(H^{\epsilon}-H)\\
&+\frac{1}{2}\nabla(|H^{\epsilon}|^2-|H|^2)-\frac{1}{2}\nabla(|v^{\epsilon}|^2-|v|^2),\\
\Phi_2:=\,&(v^{\epsilon}-v)\cdot\nabla H+(v^{\epsilon}-v)\cdot\nabla (H^{\epsilon}-H)+v\cdot\nabla(H^{\epsilon}-H)\\
&-(H^{\epsilon}-H)\cdot\nabla v-(H^{\epsilon}-H)\cdot\nabla (v^{\epsilon}-v)-H\cdot\nabla(v^{\epsilon}-v).
\end{align*}
Doing basic $L^{2}$-estimate, we obtain the following identity:
\begin{align*}
\frac{1}{2}&\frac{d}{dt}(\| v^{\epsilon}-v \|^2+\|H^{\epsilon}-H\|^2)+\epsilon(\| \nabla\times(v^{\epsilon}-v) \|^2
+\|\nabla\times( H^{\epsilon}-H)\|^2)\nonumber\\
&+(\Phi_1,v^{\epsilon}-v )+(\Phi_2,H^{\epsilon}-H)+B_1+B_2=(\epsilon\Delta v,v^{\epsilon}-v)+(\epsilon\Delta H,H^{\epsilon}-H),
\end{align*}
where
\begin{align*}
B_1&:=\epsilon\int_{\partial\Omega} n\times(\omega^{\epsilon}_v-\omega_v)(v^{\epsilon}-v )\\
&=\epsilon\int_{\partial\Omega}(B(v^{\epsilon}-v )+Bv-n\times\omega_v)(v^{\epsilon}-v ),\\
B_2&:=\epsilon\int_{\partial\Omega} n\times(\omega^{\epsilon}_H-\omega_H)(H^{\epsilon}-H )\\
&=\epsilon\int_{\partial\Omega}(B(H^{\epsilon}-H )+BH-n\times\omega_H)(H^{\epsilon}-H ).
\end{align*}

First, we easily note that
\begin{equation}\label{5.12}
|(\epsilon\Delta v,v^{\epsilon}-v)|+|(\epsilon\Delta H,H^{\epsilon}-H)|\leq C(\| v^{\epsilon}-v \|^2+\| H^{\epsilon}-H\|^2)+\epsilon^2.
\end{equation}

Next, we deal with the boundary terms $B_1$ and  $B_2$. For $B_1$, we have
\begin{align*}
B_1=&\,\epsilon\int_{\partial\Omega}(B(v^{\epsilon}-v )+Bv-n\times\omega_v)(v^{\epsilon}-v )\\
\leq&\, C\,\epsilon\int_{\partial\Omega}(|v^{\epsilon}-v |^2+|v^{\epsilon}-v |).
\end{align*}
Due to the trace theorem:
\begin{equation}\label{5.14}
|u|_{L^1(\partial\Omega)}\leq C\,|u|_{L^2(\partial\Omega)}\leq C\|u\|_{H^\frac{1}{2}}
\end{equation}
and the interpolation inequality:
\begin{equation}\label{5.15}
\| u\|_{H^\frac{1}{2}(\Omega)}\leq C\,\| u\|^{\frac{1}{2}}\| u\|_{H^1}^{\frac{1}{2}},
\end{equation}
we further obtain that
\begin{align}\label{5.16}
B_1\leq\,& C\epsilon\,(\|v^{\epsilon}-v \| \| \omega^{\epsilon}_v-\omega_v\|+|v^{\epsilon}-v |_{L^1(\partial\Omega)})\nonumber\\
\leq&\,2\delta\epsilon\|\omega^{\epsilon}_v-\omega_v\|^2+C_\delta\| v^{\epsilon}-v \|^2+\epsilon^\frac{3}{2}.
\end{align}

Similarly, we also get that
\begin{equation}\label{5.17}
B_2\leq2\delta\epsilon\| \omega^{\epsilon}_H-\omega_H\|^2+C_\delta\| H^{\epsilon}-H \|^2+\epsilon^\frac{3}{2}.
\end{equation}

Finally, we deal with $(\Phi_1, v^{\epsilon}-v)$ and $(\Phi_2, H^{\epsilon}-H)$. We have
\begin{align*}
&|(\Phi_1, v^{\epsilon}-v)+(\Phi_2, H^{\epsilon}-H)|\nonumber\\
=\,&\big{|}(v^{\epsilon}-v,v\cdot\nabla(v^{\epsilon}-v)+(v^{\epsilon}-v)\cdot\nabla v+(v^{\epsilon}-v)\cdot\nabla(v^{\epsilon}-v)-\frac{1}{2}\nabla(|v^{\epsilon}|^2-|v|^2)\nonumber\\
&-H\cdot\nabla(H^{\epsilon}-H)-(H^{\epsilon}-H)\cdot\nabla H-(H^{\epsilon}-H)\cdot\nabla(H^{\epsilon}-H)\nonumber\\
&+\frac{1}{2}\nabla(|H^{\epsilon}|^2-|H|^2))+( H^{\epsilon}-H,(v^{\epsilon}-v)\cdot\nabla H+(v^{\epsilon}-v)\cdot\nabla (H^{\epsilon}-H)\nonumber\\
&+v\cdot\nabla(H^{\epsilon}-H)-(H^{\epsilon}-H)\cdot\nabla v-(H^{\epsilon}-H)\cdot\nabla (v^{\epsilon}-v)-H\cdot\nabla(v^{\epsilon}-v))\big{|}.
\end{align*}
We note that
\begin{align*}
&(\frac{1}{2}\nabla(|v^{\epsilon}|^2-|v|^2)-\frac{1}{2}\nabla(|H^{\epsilon}|^2-|H|^2),v^{\epsilon}-v)=0,\\
&(v^{\epsilon}-v,v\cdot\nabla(v^{\epsilon}-v))=0,\quad(v^{\epsilon}-v,(v^{\epsilon}-v)\cdot\nabla(v^{\epsilon}-v))=0,\\
&(H^{\epsilon}-H,v\cdot\nabla(H^{\epsilon}-H))=0,\quad(H^{\epsilon}-H,(v^{\epsilon}-v)\cdot\nabla(H^{\epsilon}-H) )=0,\\
&((H^{\epsilon}-H)\cdot\nabla (v^{\epsilon}-v),H^{\epsilon}-H)+((H^{\epsilon}-H)\cdot\nabla H^{\epsilon}-H,v^{\epsilon}-v)=0,\\
&(H^{\epsilon}-H,H\cdot\nabla(v^{\epsilon}-v))+(v^{\epsilon}-v,H\cdot\nabla(H^{\epsilon}-H))=0.
\end{align*}
Consequently, one has
\begin{equation}\label{5.24}
|(\Phi_1, v^{\epsilon}-v)+(\Phi_2, H^{\epsilon}-H)|\leq C(\|v^{\epsilon}-v\|^2+\| H^{\epsilon}-H\|^2).
\end{equation}

From \eqref{5.12}, \eqref{5.16}, \eqref{5.17} and \eqref{5.24}, we get
\begin{align*}
\frac{1}{2}\frac{d}{dt}(\| v^{\epsilon}-v\|^2+\|H^{\epsilon}-H\|^2)+\epsilon(\|\nabla\times(v^{\epsilon}-v) \|^2
+\|\nabla\times( H^{\epsilon}-H)\|^2)\nonumber\\
\leq C\|v^{\epsilon}-v\|^2+\|H^{\epsilon}-H\|^2+\epsilon^\frac{3}{2}.
\end{align*}
Then, by using Gronwall's inequality, we arrive at
\begin{align*}
\|v^{\epsilon}-v \|^2+\|H^{\epsilon}-H\|^2+\epsilon\int_0^t(\|v^{\epsilon}-v\|_{H^1}^2
+\|H^{\epsilon}-H\|_{H^1}^2)\leq C\epsilon^\frac{3}{2}.
\end{align*}

Consequently, by using the Gagliardo-Nirenberg interpolation inequality, we have
\begin{align*}
\|v^{\epsilon}-v\|_{L^\infty}+\|H^{\epsilon}-H\|_{L^\infty}\leq&\, C(\|v^{\epsilon}-v\|^{\frac{2}{5}}\|v^{\epsilon}-v\|^{\frac{3}{5}}_{W^{1,\infty}}\nonumber\\
&+\|H^{\epsilon}-H\|^{\frac{2}{5}}\|H^{\epsilon}-H\|^{\frac{3}{5}}_{W^{1,\infty}})\leq C\epsilon^{\frac{3}{10}}.
\end{align*}
\end{proof}
Before we go to prove the rate of the convergence in $H^1$, we have the following observation.
\begin{Lemma}\label{L4.3} We have
\begin{align}
\|u\|_{H^2}\leq\|P\Delta u\|+\|u\|,\quad\forall \,u\in W_{B},\nonumber
\end{align}
where
\begin{align}
W_{B}=\big{\{}u\in H^2(\Omega)\,\big{|}\,\nabla\cdot u=0\,\,\,\text{in}\,\,\,\Omega,\,\,u\cdot n=0,\,\,\,n\times(\nabla\times u)=[Bu]_\tau\,\,\,\text{on}\,\,\,\partial\Omega\big{\}}.\nonumber
\end{align}
\begin{proof}
We consider the following boundary value problem:
\begin{align}
&\gamma I-\Delta u +\nabla p=f\quad \text{in}\quad\Omega,\label{4.34}\\
&\nabla\cdot u=0\quad \text{in}\quad\Omega,\label{4.35}\\
&u\cdot n=0,\quad n\times(\nabla\times u)=[Bu]_\tau\quad \text{on}\quad\Omega,\label{4.36}
\end{align}
where $\gamma$ is a large enough positive constant. Define a bilinear form as
\begin{align}\label{4.39}
\mathcal{B}(u,\phi)=\gamma(u,\phi)+(\nabla\times u,\nabla\times \phi)+\int_{\partial\Omega}Au\cdot\phi
\end{align}
with the domain $D(\mathcal{B})=\big{\{}u\in H^1(\Omega)\,\big{|}\,\nabla\cdot u=0\,\,\,\text{in}\,\,\,\Omega,\,\,u\cdot n=0\,\,\,\text{on}\,\,\,\partial\Omega\big{\}}.$

It is clear that $\mathcal{B}(u,\phi)$ with domain $D(\mathcal{B})$ is a positive densely defined closed bilinear form. Let $\mathcal{O}$ be the self-extension of $\mathcal{B}(u,\phi)$. We find that
$W_{B}\subset D(\mathcal{O})$ and $\mathcal{O}u=\gamma u+P(-\Delta u)$ for any $u\in D(\mathcal{O})$. Let $u\in W_{B}$ and $\mathcal{O}u=f$. It follows from \eqref{4.39} and Lemma \ref{L2.5} that
\begin{align}\label{4.38}
\|u\|_{H^1}\leq C \,\|f\|.
\end{align}

Now, let $n(x)$ and $B(x)$ be the internal smooth extensions of the normal vector $n$ and $B$ in \eqref{4.36}. Based on Lemma \ref{L2.4}, we have
\begin{align}
B(x)u\times n(x)=\nabla\times k+\nabla h+\nabla g,\nonumber
\end{align}
where $k\in FH\cap H^2$, $\nabla h\in HG$ and $\nabla g\in GG$. We find
\begin{align*}
&\Delta g=\nabla\cdot(B(x)u\times n(x))\quad\text{in}\quad\Omega,\\
&g=0\quad\text{on}\quad\partial\Omega.
\end{align*}
From the elliptic regularity theory, we obtain
\begin{align}
\|\nabla g\|_{H^1}\leq C\|u\|_{H^1}.\nonumber
\end{align}
Since $HG$ is finite dimensional, the following inequality holds
\begin{align}
\|\nabla h\|_{H^1}\leq C\|u\|.\nonumber
\end{align}
Further, it follows from Lemma \ref{L2.5} and Poincar$\acute{e}$ type inequality in Lemma 3.3 of \cite{YZ1} that
\begin{align}\label{4.44}
\|k\|_{H^2}\leq C\|\nabla\times k\|_{H^1}\leq C\|u\|_{H^1}\leq C\|f\|.
\end{align}
Integrating by parts and noting that $n\times\nabla h=0$, $n\times\nabla g=0$ on the boundary, we have
\begin{align}
\int_\Omega(\nabla\times k)\cdot(\nabla\times \phi)+\int_{\partial\Omega}n\times(Bu\times n)\cdot\phi=(-\Delta k,\phi)\nonumber
\end{align}
for any $\phi\in H^1$. We observe that $n\times(Bu\times n)=Bu$, so we have
\begin{align}
\int_\Omega(\nabla\times (u-k))\cdot(\nabla\times \phi)=(P_{FH}(f-u+\Delta k),\phi),\quad\forall\,\phi\in H^1\cap FH,\nonumber
\end{align}
 where $P_{FH}$ denotes the projection on $FH$. Further, due to $\nabla\times u=\nabla\times P_{FH}(u)$, we get
\begin{align}
\int_\Omega(\nabla\times (P_{FH}(u)-k))\cdot(\nabla\times \phi)=(P_{FH}(f-\gamma u+\Delta k),\phi),\quad\forall\,\phi\in H^1\cap FH\nonumber
\end{align}
From Theorem 3.1 in \cite{YZ1}, we obtain
\begin{align}\label{4.48}
\|P_{FH}(u)-k\|_{H^2}\leq \,C\,(\|f\|+\|\Delta k\|+\|u\|).
\end{align}
Since $HH$ is finite dimensional, the following inequality holds
\begin{align}\label{4.49}
\|P_{HH}(u)\|_{H^2}\leq \,C\,\|u\|,
\end{align}
where
\begin{align}
&HH=\big{\{}u\in L^2(\Omega)\,\big{|}\,\nabla\cdot u=0,\,\,\nabla\times u=0\,\,\,\text{in}\,\,\,\Omega,\,\,u\cdot n=0\,\,\,\text{on}\,\,\,\partial\Omega\big{\}},\,\,\,\,\nonumber\\
&\mathbb{H}=HH\oplus FH.\nonumber
\end{align}

We get from \eqref{4.38}, \eqref{4.44}, \eqref{4.48} and \eqref{4.49} that
\begin{align}
\|u\|_{H^2}\leq \,C\,\|f\|.\nonumber
\end{align}
Consequently, we complete the proof of Lemma \ref{L4.3}.
\end{proof}
\end{Lemma}
Now we turn to prove the rate of convergence in $H^1(\Omega)$.
\begin{Lemma}\label{L5.2} %Let $(v_0,H_0)\in H^3(\Omega)$ satisfy the assumptions in Theorem \ref{Th1} and $(v,H)$ be the solution to the ideal MHD equations \eqref{1.7}-\eqref{1.10} on $[0,T]$ with $(v(0),H(0))=(v_0,H_0)$, and $(v^{\epsilon},H^{\epsilon})$ be the solution to the MHD equations \eqref{1.1}-\eqref{1.4.2}. Then we have
Under the assumptions in Theorem \ref{Th3}, we have
\begin{align}
\|v^{\epsilon}-v \|_{H^1}^2&+\|H^{\epsilon}-H\|_{H^1}^2\nonumber\\
&+\epsilon\int_0^t(\|v^{\epsilon}-v\|_{H^2}^2
+\| H^{\epsilon}-H\|_{H^2}^2)\leq C\epsilon^\frac{1}{2}\quad\text{on}\quad [0,T_2],
\end{align}
where $\epsilon$ small enough and $T_2=\min\{T_0,T_1\}$. Also, we have
\begin{equation}\nonumber
\|v^{\epsilon}-v\|^p_{W^{1,p}}+\|H^{\epsilon}-H\|^p_{W^{1,p}}\leq C\epsilon^\frac{1}{2}\quad\text{on}\quad[0,T_2]
\end{equation}
for $2\leq p<\infty.$
\end{Lemma}
\begin{proof}We note
\begin{equation}\nonumber
\partial_t(v^{\epsilon}-v)\cdot n=0,\quad\partial_t(H^{\epsilon}-H)\cdot n=0.
\end{equation}
It follows from \eqref{5.2}-\eqref{5.6} that
\begin{align*}
\frac{1}{2}&\frac{d}{dt}(\|\omega_ v^{\epsilon}-\omega_v \|^2+\|\omega_ H^{\epsilon}-\omega_H\|^2)+\epsilon(\|P\Delta(v^{\epsilon}-v) \|^2
+\| P\Delta( H^{\epsilon}-H)\|^2)\nonumber\\
=&(\Phi_1,P\Delta (v^{\epsilon}-v ))+(\Phi_2,P\Delta (H^{\epsilon}-H))+B_1+B_2-(\epsilon\Delta v,P\Delta(v^{\epsilon}-v))\nonumber\\
&-(\epsilon\Delta H,P\Delta(H^{\epsilon}-H)),
\end{align*}
where $\Phi_1$ and $\Phi_2$ are as same as these in Lemma \ref{L5.1}, but $B_1$ and $B_2$ have different forms:
\begin{equation}\nonumber
B_1:=\int_{\partial\Omega}\partial_t(v^{\epsilon}-v )\cdot(n\times(\omega^{\epsilon}_v-\omega_v)),\quad B_2:=\int_{\partial\Omega}\partial_t(H^{\epsilon}-H )\cdot(n\times(\omega^{\epsilon}_H-\omega_H)).
\end{equation}
Now, let us deal with these two boundary terms as follows
\begin{align}\label{5.31}
&B_1+B_2\nonumber\\
=&\int_{\partial\Omega}\partial_t(v^{\epsilon}-v )\cdot(B(v^{\epsilon}-v )+Bv-n\times\omega_v)\nonumber\\
&+\int_{\partial\Omega}\partial_t(H^{\epsilon}-H )\cdot(B(H^{\epsilon}-H )+BH-n\times\omega_H)\nonumber\\
=&\frac{1}{2}\frac{d}{dt}\Big{(}\int_{\partial\Omega}B(v^{\epsilon}-v )\cdot(v^{\epsilon}-v )+2\int_{\partial\Omega}(v^{\epsilon}-v)\cdot(Bv-n\times\omega_v)\Big{)}-\widetilde{B}_1\nonumber\\
&+\frac{1}{2}\frac{d}{dt}\Big{(}\int_{\partial\Omega}B(H^{\epsilon}-H )\cdot(H^{\epsilon}-H )+2\int_{\partial\Omega}(H^{\epsilon}-H)\cdot(BH-n\times\omega_H)\Big{)}-\widetilde{B}_2,
\end{align}
where
\begin{equation}\label{5.32}
\widetilde{B}_1:=\int_{\partial\Omega}(v^{\epsilon}-v )\cdot\partial_t(Bv-n\times\omega_v),\quad\widetilde{B}_2:=\int_{\partial\Omega}(H^{\epsilon}-H )\cdot\partial_t(BH-n\times\omega_H).
\end{equation}
It follows from Lemma \ref{L5.1}, \eqref{5.14} and \eqref{5.15} that
\begin{equation}\label{5.33}
\begin{split}
|\widetilde{B}_1+\widetilde{B}_2|\leq&\, C\Big{(}\int_{\partial\Omega}|v^{\epsilon}-v|^{2}\Big{)}^\frac{1}{2}+\,C\Big{(}\int_{\partial\Omega}|H^{\epsilon}-H|^2\Big{)}^\frac{1}{2}\\
\leq&\,\delta\,(\|\omega_ v^{\epsilon}-\omega_v \|^2+\|\omega_ H^{\epsilon}-\omega_H \|^2)+C\epsilon^\frac{1}{2}.
\end{split}
\end{equation}
We easily get
\begin{align}\label{5.34}
&|(\epsilon\Delta v,P\Delta(v^{\epsilon}-v))+(\epsilon\Delta H,P\Delta(H^{\epsilon}-H))|\nonumber\\
\leq&\,\frac{\epsilon}{2}(\| P\Delta(v^{\epsilon}-v)\|^2+\| P\Delta(H^{\epsilon}-H)\|^2)+C\,\epsilon
\end{align}
and
\begin{align}\label{5.35}
&-(\Phi_1,P\Delta (v^{\epsilon}-v ))-(\Phi_2,P\Delta (H^{\epsilon}-H))\nonumber\\
=&(P\Phi_1,-\Delta (v^{\epsilon}-v ))+(P\Phi_2,-\Delta (H^{\epsilon}-H))\nonumber\\
=&(\nabla\times\Phi_1,\omega_ v^{\epsilon}-\omega_v)+\int_{\partial\Omega}n\times(\omega_ v^{\epsilon}-\omega_v)\cdot P\Phi_1\nonumber\\
&+(\nabla\times\Phi_2,\omega_ H^{\epsilon}-\omega_H)+\int_{\partial\Omega}n\times(\omega_ H^{\epsilon}-\omega_H)\cdot P\Phi_2\nonumber\\
=&(\nabla\times\Phi_1,\omega_ v^{\epsilon}-\omega_v)+\int_{\partial\Omega}(B(v^{\epsilon}-v )+Bv-n\times\omega_v)\cdot P\Phi_1\nonumber\\
&+(\nabla\times\Phi_2,\omega_ H^{\epsilon}-\omega_H)+\int_{\partial\Omega}(B(H^{\epsilon}-H)+BH-n\times\omega_H)\cdot P\Phi_2.
\end{align}
From \eqref{5.31}, \eqref{5.33}, \eqref{5.34}, and \eqref{5.35}, we arrive at
\begin{align}\label{5.36}
\frac{1}{2}\frac{d}{dt}E&+\frac{\epsilon}{2}(\|P\Delta(v^{\epsilon}-v) \|^2
+\| P\Delta( H^{\epsilon}-H)\|^2)\nonumber\\
&\leq I_1 + I_2 + I_3 +C(\|\omega_ v^{\epsilon}-\omega_v \|^2+\|\omega_ H^{\epsilon}-\omega_H \|^2+\epsilon^\frac{1}{2}),
\end{align}
where
\begin{align*}
E:=&\|\omega_ v^{\epsilon}-\omega_v \|^2+\|\omega_ H^{\epsilon}-\omega_H\|^2\\
&-\int_{\partial\Omega}B(v^{\epsilon}-v )\cdot(v^{\epsilon}-v )-2\int_{\partial\Omega}(v^{\epsilon}-v)\cdot(Bv-n\times\omega_v)\\
&-\int_{\partial\Omega}B(H^{\epsilon}-H )\cdot(H^{\epsilon}-H )-2\int_{\partial\Omega}(H^{\epsilon}-H)\cdot(BH-n\times\omega_H),\\
I_1:=&|(\nabla\times\Phi_1,\omega_ v^{\epsilon}-\omega_v)+(\nabla\times\Phi_2,\omega_ H^{\epsilon}-\omega_H)|,\\
I_2:=&\Big{|}\int_{\partial\Omega}B(v^{\epsilon}-v )\cdot P\Phi_1+\int_{\partial\Omega}B(H^{\epsilon}-H)\cdot P\Phi_2\Big{|},\\
I_3:=&\Big{|}\int_{\partial\Omega}(Bv-n\times\omega_v)\cdot P\Phi_1+\int_{\partial\Omega}(BH-n\times\omega_H)\cdot P\Phi_2\Big{|}.
\end{align*}
Now we estimate the terms $I_1$, $I_2$ and $I_3$ in turn. The term $I_1$ can be estimated easily by using Sobolev inequalities and the obtained uniform bounds for $v^{\epsilon}$ and $H^{\epsilon}$ in Theorem \ref{Th1}. We have
\begin{equation}\nonumber
I_1=|(\nabla\times\Phi_1,\omega_ v^{\epsilon}-\omega_v)+(\nabla\times\Phi_2,\omega_ H^{\epsilon}-\omega_H)|\leq I_{11}+I_{12},
\end{equation}
where
\begin{align*}
I_{11}=&\big{|}(v\cdot\nabla(\omega_ v^{\epsilon}-\omega_v)+(v^{\epsilon}-v)\cdot\nabla \omega_v+(v^{\epsilon}-v)\cdot\nabla(\omega_ v^{\epsilon}-\omega_v)\nonumber\\
 &-H\cdot\nabla(\omega_ H^{\epsilon}-\omega_H)-(H^{\epsilon}-H)\cdot\nabla \omega_H-(H^{\epsilon}-H)\cdot\nabla(\omega_ H^{\epsilon}-\omega_H),\omega_ v^{\epsilon}-\omega_v)\nonumber\\
&+((v^{\epsilon}-v)\cdot\nabla \omega_H+(v^{\epsilon}-v)\cdot\nabla (\omega_ H^{\epsilon}-\omega_H)+v\cdot\nabla(\omega_ H^{\epsilon}-\omega_H)\nonumber\\
&-(H^{\epsilon}-H)\cdot\nabla \omega_v-(H^{\epsilon}-H)\cdot\nabla (\omega_ v^{\epsilon}-\omega_v)-H\cdot\nabla(\omega_ v^{\epsilon}-\omega_v),\omega_ H^{\epsilon}-\omega_H)\big{|},\\
I_{12}=&\big{|}([\nabla\times,v\cdot\nabla](v^{\epsilon}-v)+[\nabla\times,(v^{\epsilon}-v)\cdot\nabla] v+[\nabla\times,(v^{\epsilon}-v)\cdot\nabla](v^{\epsilon}-v)\nonumber\\ &-[\nabla\times,H\cdot\nabla](H^{\epsilon}-H)-[\nabla\times,(H^{\epsilon}-H)\cdot\nabla]H\nonumber\\
&-[\nabla\times,(H^{\epsilon}-H)\cdot\nabla](H^{\epsilon}-H),\omega_ v^{\epsilon}-\omega_v)\nonumber\\
&+([\nabla\times,(v^{\epsilon}-v)\cdot\nabla ]H+[\nabla\times,(v^{\epsilon}-v)\cdot\nabla ](H^{\epsilon}-H)+[\nabla\times,v\cdot\nabla](H^{\epsilon}-H)\nonumber\\
&-[\nabla\times,(H^{\epsilon}-H)\cdot\nabla] v-[\nabla\times,(H^{\epsilon}-H)\cdot\nabla](v^{\epsilon}-v)\nonumber\\
&-[\nabla\times ,H\cdot\nabla](v^{\epsilon}-v),\omega_ H^{\epsilon}-\omega_H)\big{|}.
\end{align*}
We observe that
\begin{align*}
&(v\cdot\nabla(\omega_ v^{\epsilon}-\omega_v),\omega_ v^{\epsilon}-\omega_v)=0,\quad\quad((v^{\epsilon}-v)\cdot\nabla(\omega_ v^{\epsilon}-\omega_v),\omega_ v^{\epsilon}-\omega_v)=0,\\
&(v\cdot\nabla(\omega_ H^{\epsilon}-\omega_H),\omega_ H^{\epsilon}-\omega_H)=0,\quad((v^{\epsilon}-v)\cdot\nabla (\omega_ H^{\epsilon}-\omega_H),\omega_ H^{\epsilon}-\omega_H)=0,\\
&(H\cdot\nabla(\omega_ H^{\epsilon}-\omega_H)+(H^{\epsilon}-H)\cdot\nabla(\omega_ H^{\epsilon}-\omega_H),\omega_ v^{\epsilon}-\omega_v)\nonumber\\
&\quad\quad\qquad\qquad+((H^{\epsilon}-H)\cdot\nabla (\omega_ v^{\epsilon}-\omega_v)+H\cdot\nabla(\omega_ v^{\epsilon}-\omega_v),\omega_ H^{\epsilon}-\omega_H)=0.
\end{align*}
Hence
\begin{equation}\label{5.48}
I_1\leq C(\|\omega_ v^{\epsilon}-\omega_v \|^2+\|\omega_ H^{\epsilon}-\omega_H \|^2).
\end{equation}

Next, we estimate the term $I_2$. We note that
\begin{equation}\nonumber
P\Phi=\Phi+\nabla\phi
\end{equation}
holds for any function $\Phi\in L^2(\Omega)$, so we need to estimate the scalar function $\phi$ which is difficult to estimate on the boundary. In order to overcome this difficulty, we need to transform it to an estimate on $\Omega$. First, we should extend $n$ and $B$ to the interior of $\Omega$ as follows:
\begin{equation}\nonumber
n(x)=\varphi (r(x))\nabla(r(x)),\quad B(x)=\varphi (r(x))B(\Pi x),
\end{equation}
where
\begin{align*}
r(x)=\min_{y\in\partial\Omega}d(x,y),\quad
\Pi x=y_x\in\partial\Omega
\end{align*}
such that
\begin{equation}\nonumber
r(x):=d(x,y_x)
\end{equation}
is well-defined in $\Omega_\sigma=\{x\in\Omega, r(x)\leq2\sigma\}$ for some $\sigma>0$ and  $\varphi(s)\in C_c^\infty{[0,2\sigma)}$ satisfying
\begin{equation}\nonumber
\varphi(s)=1\quad\text{in}\quad[0,\sigma].
\end{equation}
Then, we can obtain that
\begin{align}\label{5.55}
I_2=\,&\Big{|}\int_{\partial\Omega}B(v^{\epsilon}-v )\cdot P\Phi_1+\int_{\partial\Omega}B(H^{\epsilon}-H)\cdot P\Phi_2\Big{|}\nonumber\\
=\,&\Big{|}\int_{\partial\Omega}\big{(}(n\times B(v^{\epsilon}-v )\cdot (n\times P\Phi_1)+(n\times B(H^{\epsilon}-H))\cdot (n\times P\Phi_2)\big{)}\Big{|}\nonumber\\
=\,&\big{|}(n\times B(v^{\epsilon}-v ),\nabla\times\Phi_1)+(n\times B(H^{\epsilon}-H ),\nabla\times\Phi_2)\nonumber\\
&-(\nabla\times (n\times B(v^{\epsilon}-v )),P\Phi_1)-(\nabla\times (n\times B(H^{\epsilon}-H )),P\Phi_2)\big{|}.
\end{align}
We easily get that
\begin{align}
&|(\nabla\times (n\times B(v^{\epsilon}-v )),P\Phi_1)+(\nabla\times (n\times B(H^{\epsilon}-H )),P\Phi_2)|\nonumber\\
\leq\,&\|n\times B(v^{\epsilon}-v )\|\|P\Phi_1\|+\|n\times B(H^{\epsilon}-H )\|\|P\Phi_2\|\nonumber\\
\leq\,& C(\|\omega_ v^{\epsilon}-\omega_v \|^2+\|\omega_ H^{\epsilon}-\omega_H \|^2).\label{5.56}
\end{align}
Now, we turn to estimate the remaining terms in \eqref{5.55}:
\begin{equation}\nonumber
|(n\times B(v^{\epsilon}-v ),\nabla\times\Phi_1)+(n\times B(H^{\epsilon}-H ),\nabla\times\Phi_2)|\leq I_{21}+I_{22},
\end{equation}
where
\begin{align*}
I_{21}=\,&\big{|}(v\cdot\nabla(\omega_ v^{\epsilon}-\omega_v)+(v^{\epsilon}-v)\cdot\nabla \omega_v+(v^{\epsilon}-v)\cdot\nabla(\omega_ v^{\epsilon}-\omega_v)\nonumber\\
&-H\cdot\nabla(\omega_ H^{\epsilon}-\omega_H)-(H^{\epsilon}-H)\cdot\nabla \omega_H\nonumber\\
&-(H^{\epsilon}-H)\cdot\nabla(\omega_ H^{\epsilon}-\omega_H),n\times B(v^{\epsilon}-v ))\nonumber\\
&+((v^{\epsilon}-v)\cdot\nabla \omega_H+(v^{\epsilon}-v)\cdot\nabla (\omega_ H^{\epsilon}-\omega_H)+v\cdot\nabla(\omega_ H^{\epsilon}-\omega_H)\nonumber\\
&-(H^{\epsilon}-H)\cdot\nabla \omega_v-(H^{\epsilon}-H)\cdot\nabla (\omega_ v^{\epsilon}-\omega_v)\nonumber\\
&-H\cdot\nabla(\omega_ v^{\epsilon}-\omega_v),n\times B(H^{\epsilon}-H ))\big{|},\\
I_{22}=\,&\big{|}([\nabla\times,v\cdot\nabla](v^{\epsilon}-v)+[\nabla\times,(v^{\epsilon}-v)\cdot\nabla] v+[\nabla\times,(v^{\epsilon}-v)\cdot\nabla](v^{\epsilon}-v)\nonumber\\ &-[\nabla\times,H\cdot\nabla](H^{\epsilon}-H)-[\nabla\times,(H^{\epsilon}-H)\cdot\nabla]H\nonumber\\
&-[\nabla\times,(H^{\epsilon}-H)\cdot\nabla](H^{\epsilon}-H),n\times B(v^{\epsilon}-v ))\nonumber\\
&+([\nabla\times,(v^{\epsilon}-v)\cdot\nabla ]H+[\nabla\times,(v^{\epsilon}-v)\cdot\nabla ](H^{\epsilon}-H)\nonumber\\
&+[\nabla\times,v\cdot\nabla](H^{\epsilon}-H)-[\nabla\times,(H^{\epsilon}-H)\cdot\nabla] v\nonumber\\
&-[\nabla\times,(H^{\epsilon}-H)\cdot\nabla](v^{\epsilon}-v)-[\nabla\times ,H\cdot\nabla](v^{\epsilon}-v),n\times B(H^{\epsilon}-H ))\big{|}.
\end{align*}
By using H\"{o}lder's inequality and Sobolev inequality, we obtain
\begin{equation}\label{5.60}
|(n\times B(v^{\epsilon}-v ),\nabla\times\Phi_1)+(n\times B(H^{\epsilon}-H ),\nabla\times\Phi_2)|\leq C(\|\omega_ v^{\epsilon}-\omega_v \|^2+\|\omega_ H^{\epsilon}-\omega_H \|^2).
\end{equation}
Based on \eqref{5.56} and \eqref{5.60}, we have
\begin{equation}\label{5.61}
I_2\leq C(\|\omega_ v^{\epsilon}-\omega_v \|^2+\|\omega_ H^{\epsilon}-\omega_H \|^2).
\end{equation}
Finally, we need to estimate the term $I_3$, i.e.
\begin{equation}\nonumber
\Big{|}\int_{\partial\Omega}(Bv-n\times\omega_v)\cdot P\Phi_1+\int_{\partial\Omega}(BH-n\times\omega_H)\cdot P\Phi_2\Big{|}.
\end{equation}
We observe that the estimate is trivial if the ideal MHD satisfies the same boundary condition as that the MHD does. However, $[Bv]_\tau-n\times\omega_v$ and $[BH]_\tau-n\times\omega_H$ may be not equal to zero. As a result, the boundary layer may occur, so we will experience more complicate estimates. Similar to the above, we get
\begin{align*}
I_3=\,&\Big{|}\int_{\partial\Omega}(Bv-n\times\omega_v)\cdot P\Phi_1+\int_{\partial\Omega}(BH-n\times\omega_H)\cdot P\Phi_2\Big{|}\nonumber\\
=\,&|\int_{\partial\Omega}(n\times(Bv-n\times\omega_v))\cdot (n\times P\Phi_1)+\int_{\partial\Omega}(n\times(BH-n\times\omega_H))\cdot (n\times P\Phi_2)|\nonumber\\
=\,&|(n\times(Bv-n\times\omega_v),\nabla\times\Phi_1)+(n\times(BH-n\times\omega_H),\nabla\times\Phi_2)\nonumber\\
&-(\nabla\times (n\times(Bv-n\times\omega_v)),P\Phi_1)-(\nabla\times(n\times(BH-n\times\omega_H)),P\Phi_2)|\nonumber\\
\leq&\,I_{31}+ I_{32},
\end{align*}
where
\begin{align*}
I_{31}=&\,|(n\times(Bv-n\times\omega_v),\nabla\times\Phi_1)+(n\times(BH-n\times\omega_H),\nabla\times\Phi_2)|,\\
I_{32}=&\,|(\nabla\times (n\times(Bv-n\times\omega_v)),P\Phi_1)+(\nabla\times(n\times(BH-n\times\omega_H)),P\Phi_2)|.
\end{align*}

We first deal with the term $I_{31}$ and note that
\begin{equation}\nonumber
I_{31}\leq L_1+L_2+L_3+L_4+L_5+L_6,
\end{equation}
where
\begin{align*}
L_1=&\,|(n\times(Bv-n\times\omega_v),\nabla\times (v\cdot\nabla(v^{\epsilon}-v))-\nabla\times(H\cdot\nabla(H^{\epsilon}-H)))|,\\
L_2=&\,|(n\times(Bv-n\times\omega_v),\nabla\times((v^{\epsilon}-v)\cdot\nabla v)-\nabla\times((H^{\epsilon}-H)\cdot\nabla H))|,\\
L_3=&\,|(n\times(Bv-n\times\omega_v),\nabla\times((v^{\epsilon}-v)\cdot\nabla(v^{\epsilon}-v))\nonumber\\
&-\nabla\times((H^{\epsilon}-H)\cdot\nabla(H^{\epsilon}-H)))|,\\
L_4=&\,|(n\times(BH-n\times\omega_H),\nabla\times(v\cdot\nabla(H^{\epsilon}-H))-\nabla\times(H\cdot\nabla(v^{\epsilon}-v)))|,\\
L_5=&\,|(n\times(BH-n\times\omega_H),\nabla\times((v^{\epsilon}-v)\cdot\nabla H)-\nabla\times((H^{\epsilon}-H)\cdot\nabla v))|,\\
L_6=&\,|(n\times(BH-n\times\omega_H),\nabla\times((H^{\epsilon}-H)\cdot\nabla (v^{\epsilon}-v))\nonumber\\
&-\nabla\times((v^{\epsilon}-v)\cdot\nabla (H^{\epsilon}-H)))|.
\end{align*}
%We easily find that $L_1$, $L_2$ and $L_3$ are respectively similar to $L_4$, $L_5$ and $L_6$, so we need only to consider $L_1$, $L_2$ and $L_3$.
We have
\begin{align*}
L_1=&\,|(n\times(Bv-n\times\omega_v),\nabla\times (v\cdot\nabla(v^{\epsilon}-v))-\nabla\times(H\cdot\nabla(H^{\epsilon}-H)))|\nonumber\\
=&\,|(n\times(Bv-n\times\omega_v),v\cdot\nabla(\omega_v^{\epsilon}-\omega_v)-H\cdot\nabla(\omega_H^{\epsilon}-\omega_H)\nonumber\\
&+[\nabla\times,v\cdot\nabla](v^{\epsilon}-v)-[\nabla\times,H\cdot\nabla](H^{\epsilon}-H))|.
\end{align*}
Here, we first deal with the terms which contain higher derivatives and get that
\begin{align*}
&\,|(n\times(Bv-n\times\omega_v),v\cdot\nabla(\omega_v^{\epsilon}-\omega_v)-H\cdot\nabla(\omega_H^{\epsilon}-\omega_H)|\\
=&\,|(v\cdot\nabla(n\times(Bv-n\times\omega_v)),\omega_v^{\epsilon}-\omega_v)-(H\cdot\nabla(n\times(Bv-n\times\omega_v)),\omega_H^{\epsilon}-\omega_H)|\nonumber\\
\leq&\,\Big{|}\int_{\partial\Omega} n\times(v^{\epsilon}-v)\cdot(v\cdot\nabla(n\times(Bv-n\times\omega_v)))\nonumber\\
&+(\nabla\times(v\cdot\nabla(n\times(Bv-n\times\omega_v))),v^{\epsilon}-v)\Big{|}\nonumber\\
&+\Big{|}\int_{\partial\Omega} n\times(H^{\epsilon}-H)\cdot(H\cdot\nabla(n\times(Bv-n\times\omega_v)))\nonumber\\
&+(\nabla\times(H\cdot\nabla(n\times(Bv-n\times\omega_v))),H^{\epsilon}-H))\Big{|}\nonumber\\
\leq&\, C\,(|v^{\epsilon}-v|_{L^2(\partial\Omega)}+|H^{\epsilon}-H|_{L^2(\partial\Omega)}+\|v^{\epsilon}-v\|+\|H^{\epsilon}-H\|)\nonumber\\
\leq&\, C\,(\|\omega_v^{\epsilon}-\omega_v\|^2+\|\omega_H^{\epsilon}-\omega_H\|^2+\epsilon^\frac{1}{2}).
\end{align*}
We also note that each component of $[\nabla\times,v\cdot\nabla](v^{\epsilon}-v)-[\nabla\times,H\cdot\nabla](H^{\epsilon}-H)$ is a combination of such terms $\partial_iv\cdot\nabla(v^{\epsilon}-v)_j$ and $\partial_kH\cdot\nabla(H^{\epsilon}-H)_l$. Without loss of generality, we consider the term
\begin{equation}\nonumber
((n\times(Bv-n\times\omega_v))_m,\partial_iv\cdot\nabla(v^{\epsilon}-v)_j-\partial_kH\cdot\nabla(H^{\epsilon}-H)_l).
\end{equation}
Since $\nabla\cdot\partial_iv=0$ and $\nabla\cdot\partial_kH=0$, we have
\begin{align*}
&|((n\times(Bv-n\times\omega_v))_m,\partial_iv\cdot\nabla(v^{\epsilon}-v)_j-\partial_kH\cdot\nabla(H^{\epsilon}-H)_l)|\nonumber\\
=\,&\big{|}(\partial_iv,\nabla((v^{\epsilon}-v)_j(n\times(Bv-n\times\omega_v))_m))\nonumber\\
&-(\partial_iv\cdot\nabla(n\times(Bv-n\times\omega_v))_m,(v^{\epsilon}-v)_j)\nonumber\\
&-(\partial_kH,\nabla((H^{\epsilon}-H)_l(n\times(Bv-n\times\omega_v))_m))\nonumber\\
&+(\partial_kH\cdot\nabla(n\times(Bv-n\times\omega_v))_m,(H^{\epsilon}-H)_l)\big{|}\nonumber\\
=\,&\Big{|}\int_{\partial\Omega}(v^{\epsilon}-v)_j(n\times(Bv-n\times\omega_v))_m\partial_iv\cdot n\nonumber\\
&-\int_{\partial\Omega}(H^{\epsilon}-H)_l(n\times(Bv-n\times\omega_v))_m\partial_kH\cdot n\nonumber\\
&-(\partial_iv\cdot\nabla(n\times(Bv-n\times\omega_v))_m,(v^{\epsilon}-v)_j)\nonumber\\
&+(\partial_iH\cdot\nabla(n\times(Bv-n\times\omega_v))_m,(H^{\epsilon}-H)_l)\Big{|}\nonumber\\
\leq& \,C\,(|v^{\epsilon}-v|_{L^1(\Omega)}+|H^{\epsilon}-H|_{L^1(\Omega)}+\|v^{\epsilon}-v\|+\|H^{\epsilon}-H\|)\nonumber\\
\leq&\,C\,(\|\omega_v^{\epsilon}-\omega_v\|^2+\|\omega_H^{\epsilon}-\omega_H\|^2+\epsilon^\frac{1}{2}).
\end{align*}
Hence, we obtain
\begin{equation}\nonumber
L_1\leq C(\|\omega_v^{\epsilon}-\omega_v\|^2+\|\omega_H^{\epsilon}-\omega_H\|^2+\epsilon^\frac{1}{2}).
\end{equation}
Compared to $L_1$, both $L_2$ and $L_3$ can be easily estimated. In fact, we have
\begin{align*}
L_2=&\,|(n\times(Bv-n\times\omega_v),\nabla\times((v^{\epsilon}-v)\cdot\nabla v)-\nabla\times((H^{\epsilon}-H)\cdot\nabla H))|\nonumber\\
=&\,\Big{|}\int_{\partial\Omega}(n\times(Bv-n\times\omega_v))(n\times((v^{\epsilon}-v)\cdot\nabla v)-n\times((H^{\epsilon}-H)\cdot\nabla H))\nonumber\\
&+(\nabla\times(n\times(Bv-n\times\omega_v)),(v^{\epsilon}-v)\cdot\nabla v-(H^{\epsilon}-H)\cdot\nabla H)\Big{|}\nonumber\\
\leq\,& C(|v^{\epsilon}-v|_{L^1(\partial\Omega)}+|H^{\epsilon}-H|_{L^1(\partial\Omega)}+\|v^{\epsilon}-v\|+\|H^{\epsilon}-H\|)\nonumber\\
\leq\,& C(\|\omega_v^{\epsilon}-\omega_v\|^2+\|\omega_H^{\epsilon}-\omega_H\|^2+\epsilon^\frac{1}{2}),\\
L_3=\,&\,\big{|}(n\times(Bv-n\times\omega_v),\nabla\times((v^{\epsilon}-v)\cdot\nabla(v^{\epsilon}-v))\nonumber\\
&-\nabla\times((H^{\epsilon}-H)\cdot\nabla(H^{\epsilon}-H)))\big{|}\nonumber\\
=\,&\,\big{|}(n\times(Bv-n\times\omega_v),[\nabla\times,(v^{\epsilon}-v)\cdot\nabla](v^{\epsilon}-v)\nonumber\\
&-[\nabla\times,(H^{\epsilon}-H)\cdot\nabla](H^{\epsilon}-H))\nonumber\\
&+(n\times(Bv-n\times\omega_v),(v^{\epsilon}-v)\cdot\nabla(\omega_v^{\epsilon}-\omega_v)-(H^{\epsilon}-H)\cdot\nabla(\omega_H^{\epsilon}-\omega_H))\big{|}\nonumber\\
=\,&\,\big{|}\big{(}n\times(Bv-n\times\omega_v),[\nabla\times,(v^{\epsilon}-v)\cdot\nabla](v^{\epsilon}-v)\nonumber\\
&-[\nabla\times,(H^{\epsilon}-H)\cdot\nabla](H^{\epsilon}-H)\big{)}\nonumber\\
&-\big{(}(v^{\epsilon}-v)\cdot\nabla(n\times(Bv-n\times\omega_v)),\omega_v^{\epsilon}-\omega_v\big{)}\nonumber\\
&+\big{(}(H^{\epsilon}-H)\cdot\nabla(n\times(Bv-n\times\omega_v),\omega_H^{\epsilon}-\omega_H\big{)}\big{|}\nonumber\\
\leq\,& C(\|\omega_v^{\epsilon}-\omega_v\|^2+\|\omega_H^{\epsilon}-\omega_H\|^2).
\end{align*}
We find that $L_4$, $L_5$ and $L_6$ have similar structures to $L_1$, $L_2$ and $L_3$ respectively, so we can get
\begin{align*}
&L_4\leq C(\|\omega_v^{\epsilon}-\omega_v\|^2+\|\omega_H^{\epsilon}-\omega_H\|^2+\epsilon^\frac{1}{2}),\\
&L_5\leq C(\|\omega_v^{\epsilon}-\omega_v\|^2+\|\omega_H^{\epsilon}-\omega_H\|^2+\epsilon^\frac{1}{2}),\\
&L_6\leq C(\|\omega_v^{\epsilon}-\omega_v\|^2+\|\omega_H^{\epsilon}-\omega_H\|^2).
\end{align*}
From the estimates of $L_i \,(i=1,\cdots, 6)$, we get
\begin{equation}\nonumber
I_{31}\leq C(\|\omega_v^{\epsilon}-\omega_v\|^2+\|\omega_H^{\epsilon}-\omega_H\|^2+\epsilon^\frac{1}{2}).
\end{equation}

Now, it remains to estimate the term $I_{32}$, i.e.
\begin{equation}\nonumber
|(\nabla\times (n\times(Bv-n\times\omega_v)),P\Phi_1)+(\nabla\times(n\times(BH-n\times\omega_H)),P\Phi_2)|.
\end{equation}
First, we consider $(\nabla\times (n\times(Bv-n\times\omega_v)),P\Phi_1)$. Because it involves Leray projection, some terms which contain higher derivatives of $v^{\epsilon}-v$ or $H^{\epsilon}-H$ can not be estimated easily. We have the observations
\begin{align}
&v\cdot\nabla(v^{\epsilon}-v)-(v^{\epsilon}-v)\cdot\nabla v=\nabla\times((v^{\epsilon}-v)\times v),\label{5.80}\\
&H\cdot\nabla(H^{\epsilon}-H)-(H^{\epsilon}-H)\cdot\nabla H=\nabla\times((H^{\epsilon}-H)\times H).\label{5.81}
\end{align}
Since $(v^{\epsilon}-v)\cdot n=0$, $v\cdot n=0$, $(H^{\epsilon}-H)\cdot n=0$ and $H\cdot n=0$, it means that
\begin{equation}\label{5.82}
(v^{\epsilon}-v)\times v=\lambda_1n,\quad(H^{\epsilon}-H)\times H=\lambda_2n.
\end{equation}
Due to \eqref{5.80}-\eqref{5.82}, we easily obtain
\begin{equation}\nonumber
\nabla\times((v^{\epsilon}-v)\times v)\in \mathbb{H},\quad\nabla\times((H^{\epsilon}-H)\times H)\in\mathbb{H},
\end{equation}
where $\mathbb{H}$ is Leray projection space. Thus we have the following equality
\begin{align*}
P\Phi_1=&v\cdot\nabla(v^{\epsilon}-v)-(v^{\epsilon}-v)\cdot\nabla v+P\Phi^1_v\nonumber\\
&-(H\cdot\nabla(H^{\epsilon}-H)-(H^{\epsilon}-H)\cdot\nabla H)-P\Phi^1_H,
\end{align*}
where
\begin{align*}
&P\Phi^1_v=P[2(v^{\epsilon}-v)\cdot\nabla v+(v^{\epsilon}-v)\cdot\nabla(v^{\epsilon}-v)],\\
&P\Phi^1_H=P[2(H^{\epsilon}-H)\cdot\nabla H+(H^{\epsilon}-H)\cdot\nabla(H^{\epsilon}-H)].
\end{align*}
Hence, we have
\begin{align*}
&(\nabla\times (n\times(Bv-n\times\omega_v)),P\Phi_1)\nonumber\\
=&(\nabla\times (n\times(Bv-n\times\omega_v)),P\Phi^1_v)-(\nabla\times (n\times(Bv-n\times\omega_v)),P\Phi^1_H)\nonumber\\
&+(\nabla\times (n\times(Bv-n\times\omega_v)),v\cdot\nabla(v^{\epsilon}-v)-(v^{\epsilon}-v)\cdot\nabla v\nonumber\\
&-(H\cdot\nabla(H^{\epsilon}-H)-(H^{\epsilon}-H)\cdot\nabla H)).
\end{align*}
First, we have
\begin{align}
\|P\Phi^1_v\|\leq C&\,(\|v\|_{W^{1,\infty}}+\|v^{\epsilon}-v\|_{W^{1,\infty}})\|v^{\epsilon}-v\|,\label{5.89}\\
\|P\Phi^1_H\|\leq C&\,(\|H\|_{W^{1,\infty}}+\|H^{\epsilon}-H\|_{W^{1,\infty}})\|H^{\epsilon}-H\|,\label{5.90}\\
|(\nabla\times (n\times(&Bv-n\times\omega_v)),P\Phi^1_v)-(\nabla\times (n\times(Bv-n\times\omega_v)),P\Phi^1_H)|\nonumber\\
&\leq C\|v\|_{H^2}(\|P\Phi^1_v\|+\|P\Phi^1_H\|).\label{5.88}
\end{align}
From \eqref{5.89}-\eqref{5.88} and Lemma \ref{L5.1}, we get
\begin{equation}\nonumber
|(\nabla\times (n\times(Bv-n\times\omega_v)),P\Phi^1_v)-(\nabla\times (n\times(Bv-n\times\omega_v)),P\Phi^1_H)|\leq C\epsilon^\frac{3}{4}.
\end{equation}
Next, note that
\begin{align*}
&|(\nabla\times (n\times(Bv-n\times\omega_v)),v\cdot\nabla(v^{\epsilon}-v)|\nonumber\\
=&|(v\cdot\nabla(\nabla\times (n\times(Bv-n\times\omega_v))),v^{\epsilon}-v)|\leq C\|v\|_{H^3}\|v^{\epsilon}-v\|\leq C\epsilon^\frac{3}{4}.
\end{align*}
Similarly, we obtain
\begin{align*}
&|(\nabla\times (n\times(Bv-n\times\omega_v)),H\cdot\nabla(H^{\epsilon}-H)|\nonumber\\
=&|(H\cdot\nabla(\nabla\times (n\times(Bv-n\times\omega_v))),H^{\epsilon}-H)|\leq C\|v\|_{H^3}\|H^{\epsilon}-H\|\leq C\epsilon^\frac{3}{4}.
\end{align*}
At the same time, we get directly that
\begin{align*}
|(\nabla\times (n\times(Bv-n\times\omega_v)),(&v^{\epsilon}-v)\cdot\nabla v+(H^{\epsilon}-H)\cdot\nabla H))|\\
\leq &\,C(\|v^{\epsilon}-v\|+\|H^{\epsilon}-H\|)\leq C\epsilon^\frac{3}{4}.
\end{align*}
Therefore,
\begin{equation}\nonumber
|(\nabla\times (n\times(Bv-n\times\omega_v)),P\Phi_1)|\leq C\epsilon^\frac{3}{4}.
\end{equation}

By using the same methods as above, we observe
\begin{equation}\nonumber
P\Phi_2=\Phi_2.
\end{equation}
Hence, we get
\begin{align*}
|(\nabla\times(n\times(BH-n\times\omega_H)),P\Phi_2)|=|(\nabla\times(n\times(BH-n\times\omega_H)),\Phi_2)|\leq C\epsilon^\frac{3}{4}.
\end{align*}
Finally, we have
\begin{equation}\nonumber
I_{32}\leq C\epsilon^\frac{3}{4}.
\end{equation}
Thus, we conclude that
\begin{equation}\label{5.99}
I_3\leq C(\|\omega_v^{\epsilon}-\omega_v\|^2+\|\omega_H^{\epsilon}-\omega_H\|^2+\epsilon^\frac{1}{2}).
\end{equation}
In conclusion, it follows from \eqref{5.48}, \eqref{5.61} and \eqref{5.99} that
\begin{align*}
\frac{1}{2}\frac{d}{dt}E+\frac{\epsilon}{2}(\| P\Delta(v^{\epsilon}&-v) \|^2
+\| P\Delta( H^{\epsilon}-H)\|^2)\nonumber\\
&\leq C(\|\omega_ v^{\epsilon}-\omega_v \|^2+\|\omega_ H^{\epsilon}-\omega_H \|^2+\epsilon^\frac{1}{2}).
\end{align*}

Now, we need to deal with the left terms in the above inequality. Let us recall that
\begin{align*}
E=\,&\|\omega_ v^{\epsilon}-\omega_v \|^2+\|\omega_ H^{\epsilon}-\omega_H\|^2\nonumber\\
&-\int_{\partial\Omega}B(v^{\epsilon}-v )\cdot(v^{\epsilon}-v )-2\int_{\partial\Omega}(v^{\epsilon}-v)\cdot(Bv-n\times\omega_v)\nonumber\\
&-\int_{\partial\Omega}B(H^{\epsilon}-H )\cdot(H^{\epsilon}-H )-2\int_{\partial\Omega}(H^{\epsilon}-H)\cdot(BH-n\times\omega_H).
\end{align*}
We note that
\begin{align*}
&\Big{|}\int_{\partial\Omega}B(v^{\epsilon}-v )\cdot(v^{\epsilon}-v )+\int_{\partial\Omega}B(H^{\epsilon}-H )\cdot(H^{\epsilon}-H )\Big{|}\nonumber\\
\leq \,&C\,(|v^{\epsilon}-v|^2_{L^2({\partial\Omega})}+|H^{\epsilon}-H|^2_{L^2({\partial\Omega})})\nonumber\\
\leq \,&\delta\,(\|\omega_ v^{\epsilon}-\omega_v\|^2+\|\omega_ H^{\epsilon}-\omega_H\|^2)+C_\delta(\| v^{\epsilon}-v\|^2+\|H^{\epsilon}-H\|^2),\\
&\Big{|}2\int_{\partial\Omega}(v^{\epsilon}-v)\cdot(Bv-n\times\omega_v)+2\int_{\partial\Omega}(H^{\epsilon}-H)\cdot(BH-n\times\omega_H)\Big{|}\nonumber\\
\leq\,& \,C\,(|v^{\epsilon}-v|_{L^1({\partial\Omega})}+|H^{\epsilon}-H|_{L^1({\partial\Omega})})\nonumber\\
\leq\,&\,\delta\,(\|\omega_ v^{\epsilon}-\omega_v\|^2+\|\omega_ H^{\epsilon}-\omega_H\|^2)+C_\delta(\| v^{\epsilon}-v\|^2+\|H^{\epsilon}-H\|^2)
\end{align*}
for some $\delta$ small enough. Consequently, we get
\begin{align*}
&\|\omega_ v^{\epsilon}-\omega_v\|^2+\|\omega_ H^{\epsilon}-\omega_H\|^2+\frac{\epsilon}{2}\int_0^t(\| P\Delta(v^{\epsilon}-v) \|^2
+\| P\Delta( H^{\epsilon}-H)\|^2)\nonumber\\
\leq\,&  C\int_0^t(\|\omega_ v^{\epsilon}-\omega_v \|^2+\|\omega_ H^{\epsilon}-\omega_H \|^2)+C\epsilon^\frac{1}{2}.
\end{align*}
By using Gronwall's inequality, we have
\begin{equation}
\|\omega_ v^{\epsilon}-\omega_v\|^2+\|\omega_ H^{\epsilon}-\omega_H\|^2\leq C\epsilon^\frac{1}{2}\quad \text{on} \quad[0,T_2].
\end{equation}
Thus
\begin{align*}
\|v^{\epsilon}-v\|^2_{H^1}&+\| H^{\epsilon}-H\|^2_{H^1}\\
&+\epsilon\int_0^t(\| P\Delta(v^{\epsilon}-v)\|^2
+\|P\Delta( H^{\epsilon}-H)\|^2)\leq\, C\epsilon^\frac{1}{2}.
\end{align*}
From Lemmas \ref{L5.1} and \ref{L4.3}, we get
\begin{equation}
\epsilon\int_0^t(\| v^{\epsilon}-v \|^2_{H^2}
+\|H^{\epsilon}-H\|_{H^2}^2)\leq\, C\epsilon^\frac{1}{2}.
\end{equation}

Note that the following inequality holds
\begin{equation}
\|\nabla(u^{\epsilon}-u)\|^p_{L^p}\leq C\|\nabla(u^{\epsilon}-u)\|^{p-2}_{L^\infty}\|\nabla(u^{\epsilon}-u)\|^2.
\end{equation}
Hence, we obtain
\begin{equation}
\|\nabla(v^{\epsilon}-v)\|^p_{L^p}+\|\nabla(H^{\epsilon}-H)\|^p_{L^p}\leq C\epsilon^{\frac{1}{2}}.
\end{equation}
This completes the proof of Lemma \ref{L5.2}.
\end{proof}
From Lemmas \ref{L5.1} and \ref{L5.2}, we easily get Theorem $\ref{Th3}$.

\medskip%\medskip
% \noindent
{\bf Acknowledgements:}
Li is supported partially by NSFC (Grant No. 11271184) and
   PAPD.

%\newpage
%%%%%%%%%%%-----------------%%%%%%%%%%%%%%%%%%%%%%%%%%%%%%%%%%%%%%%%%%%%%%%%%%

\end{document}